\documentclass[12pt]{elsarticle}
\makeatletter
\def\ps@pprintTitle{%
   \let\@oddhead\@empty
   \let\@evenhead\@empty
   \let\@oddfoot\@empty
   \let\@evenfoot\@oddfoot
}
\makeatother

\usepackage{fullpage}
\usepackage[colorlinks]{hyperref}
\usepackage[utf8]{inputenc}

\biboptions{square,sort,comma,numbers}
\usepackage{graphicx}
\usepackage{amscd,amsmath,amsthm,amssymb}
\usepackage{verbatim}
\usepackage[all]{xy}
\usepackage{color}
\usepackage{hyperref}
\usepackage{booktabs}
\usepackage{tikz, float} \usetikzlibrary {positioning}
\usetikzlibrary{calc}

\usepackage{makeidx}         
\usepackage{graphicx}        
\usepackage{multicol}        
\usepackage[bottom]{footmisc}

\usepackage{array}

\usepackage{newtxmath}      

\makeindex             

\newcommand{\BLambda}{{B_\ell}}
\newcommand{\id}{\text{id}}
\newcommand{\ini}{\text{in}}

\def\frk{\mathfrak}               

\def\Phi{{\frk n}}
\def\Phi{{\frk N}}

\def\wb{{\mathbf w}}

\def\A{{\mathcal A}}

\def\opn#1#2{\def#1{\operatorname{#2}}} 
\opn\chara{char} \opn\length\ell \opn\pd{pd} \opn\rk{rk}
\opn\projdim{proj\,dim} \opn\injdim{inj\,dim} \opn\rank{rank}
\opn\depth{depth} \opn\grade{grade} \opn\height{height}
\opn\embdim{emb\,dim} \opn\codim{codim}

\opn\Tr{Tr} \opn\bigrank{big\,rank}
\opn\superheight{superheight}\opn\lcm{lcm}
\opn\trdeg{tr\,deg}
\opn\reg{reg} \opn\lreg{lreg} \opn\ini{in} \opn\lpd{lpd}
\opn\size{size} \opn\sdepth{sdepth}
\opn\link{link}\opn\fdepth{fdepth}\opn\lex{lex}
\opn\LM{LM}
\opn\LC{LC}
\opn\NF{NF}
\opn\Merge{Merge}
\opn\sgn{sgn}

\opn\div{div} \opn\Div{Div} \opn\cl{cl} \opn\Pic{Pic}
\opn\Prin{Prin}
\opn\op{op}
\opn\indeg{indeg} \opn\outdeg{outdeg}
\opn\red{red}
\opn\Spec{Spec} \opn\Supp{Supp} \opn\supp{supp} \opn\Sing{Sing}
\opn\Ass{Ass} \opn\Min{Min}\opn\Mon{Mon} \opn\val{val}
\opn\Ann{Ann} \opn\Rad{Rad} \opn\Soc{Soc}
 \opn\Ker{Ker} \opn\Coker{Coker} \opn\Am{Am}
\opn\Hom{Hom} \opn\Tor{Tor} \opn\Ext{Ext} \opn\End{End}
\opn\Aut{Aut} \opn\id{id}

\opn\nat{nat}
\opn\pff{pf}
\opn\Pf{Pf} \opn\GL{GL} \opn\SL{SL} \opn\mod{mod} \opn\ord{ord}
\opn\Gin{Gin} \opn\Hilb{Hilb}\opn\sort{sort}
\opn\Image{Image}
\opn\vol{Vol}
\opn\aff{aff} \opn\con{conv} \opn\relint{relint} \opn\st{st}
\opn\lk{lk} \opn\cn{cn} \opn\core{core} \opn\vol{vol}
\opn\link{link} \opn\star{star}\opn\lex{lex}\opn\set{set}
\opn\dist{dist}
\opn\gr{gr}

\def\pot#1#2{#1[\kern-0.28ex[#2]\kern-0.28ex]}

\opn\dirlim{\underrightarrow{\lim}}
\opn\inivlim{\underleftarrow{\lim}}
\def\Implies{\ifmmode\Longrightarrow \else
        \unskip${}\Longrightarrow{}$\ignorespaces\fi}
\def\implies{\ifmmode\Rightarrow \else
        \unskip${}\Rightarrow{}$\ignorespaces\fi}
\def\iff{\ifmmode\Longleftrightarrow \else
        \unskip${}\Longleftrightarrow{}$\ignorespaces\fi}

\let\:=\colon
\newtheorem{theorem}{Theorem}[section]
\newtheorem{lemma}[theorem]{Lemma}
\newtheorem{corollary}[theorem]{Corollary}
\newtheorem{proposition}[theorem]{Proposition}

\newtheorem{conjecture}[theorem]{Conjecture}
\newtheorem{question}[theorem]{Question}
\theoremstyle{remark}
\newtheorem{remark}[theorem]{Remark}

\theoremstyle{definition}
\newtheorem{example}[theorem]{Example}

\newtheorem{definition}[theorem]{Definition}

\DeclareMathOperator{\Gr}{Gr}
\DeclareMathOperator{\Flag}{Fl}

\let\kappa=\varkappa
\def\qed{\ifhmode\textqed\fi
      \ifmmode\ifinner\quad\qedsymbol\else\dispqed\fi\fi}
\def\textqed{\unskip\nobreak\penalty50
       \hskip2em\hbox{}\nobreak\hfil\qedsymbol
       \parfillskip=0pt \finalhyphendemerits=0}
\def\dispqed{\rlap{\qquad\qedsymbol}}

\opn\dis{dis}
\def\pnt{{\raise0.5mm\hbox{\large\bf.}}}

\opn\Lex{Lex}
\opn\syz{{\rm syz}}
\opn\spoly{{\rm spoly}}
\opn\LM{{\rm LM}}
\opn\lm{{\rm lm}}
\opn\lcm{{\rm lcm}} \opn\A{\mathcal A}

\numberwithin{equation}{section}

\newcommand{\inwb}{{\rm in}_{{\bf w}_\ell}}
\newcommand{\inw}[1]{{\rm in}_{{\bf w}_{#1}}}
\newcommand{\ul}[1]{\underline{#1}}

\DeclareMathOperator{\trop}{trop}

\DeclareMathOperator{\init}{in}

\usepackage{multirow}

\begin{document}

\begin{frontmatter}

\title{Standard monomial theory and toric degenerations of \\ Schubert varieties from matching field tableaux}
\author{
 Oliver Clarke and Fatemeh Mohammadi}

\begin{abstract}
We study Gr\"obner degenerations of Schubert varieties inside flag varieties. 
We consider toric degenerations of flag varieties induced by matching fields and semi-standard Young tableaux. We describe an analogue of matching field ideals for Schubert varieties inside the flag variety and give a complete characterization of toric ideals among them. We use a combinatorial approach to standard monomial theory to show that block diagonal matching fields give rise to toric degenerations. Our methods and results use the combinatorics of permutations associated to Schubert varieties, matching fields and their corresponding tableaux. 
\end{abstract}

\end{frontmatter}
{
  \hypersetup{linkcolor=black}
  \tableofcontents
}
\section{Introduction}
In this note we provide a new family of toric degenerations of Schubert varieties inside the full flag variety. 
Computing toric degenerations of a variety is a valuable tool that allows us to study general spaces using results from toric geometry and combinatorics, for instance, see \cite{An13, Cox2}. A toric degeneration of a given variety $X$ is a $1$-parameter family over the affine line $\mathcal{F} \rightarrow \mathbb{A}^1$ such that the fiber over $0$, often called the special fiber, is a toric variety and all other fibers are isomorphic to $X$. Most algebraic invariants of toric varieties have combinatorial counterparts such as polyhedral fans and polytopes. This makes the study of toric varieties particularly fruitful and motivates the search for toric degenerations of varieties. More precisely, a toric degeneration is a flat family and so we can calculate invariants of the original variety by calculating them for the toric fiber. This converts various abstract problems in algebraic geometry into questions about polytopes. For example, calculating the degree of a variety given a toric degeneration can be achieved by computing the volume of the moment polytope of the toric fiber.

\medskip

Toric degenerations have been studied extensively in the literature for flag varieties and their Schubert varieties, see e.g. \cite{FFL16, GL96,serhiyenko2019cluster}. Closely related are toric degenerations for Grassmannians, which have been widely studied, see e.g. \cite{rietsch2017newton, BFFHL, OllieFatemeh}. For all of these varieties, one of the most well-known examples of toric degeneration is the  Gelfand-Tsetlin degeneration which is readily understood through standard monomial theory and semi-standard Young tableaux \cite{FvectorGC, KOGAN}. Natural questions to ask are; what are the other possible toric degenerations of these varieties? And how are they related to each other? For instance, it has been shown that plabic graphs, arising from the cluster algebra structure of the Grassmannian, parametrize certain toric degenerations, see \cite{BFFHL,rietsch2017newton}. 

\medskip

One approach to study toric degenerations of varieties is by way of Gr\"obner degeneration. Given a variety $X$, any weight vector $w$ gives rise to a one-parameter family for $X$ where the ideal of the special fiber is the initial ideal $\init_w(I(X))$. Therefore, we search for weight vectors $w$ such that the initial ideal $\init_w(I(X))$ is toric, i.e.\! a prime binomial ideal. The tropicalization $\trop(X)$, see \cite{M-S}, is the collection of weight vectors for which the initial ideal $\init_w(I(X))$ does not contain any monomials and has the structure of a polyhedral fan. So natural candidates for weight vectors giving rise to toric degenerations are interior points of top-dimensional cones of the tropicalization, see e.g. \cite{kaveh2019khovanskii,KristinFatemeh,bossinger2017computing}.
In the case of $\Gr(2,n)$, it was shown in \cite{speyer2004tropical} that every such point 
gives rise to a toric degeneration of $\Gr(2,n)$.
A combinatorial approach to finding such points in the tropicalization of $\Gr(3,n)$ is taken in \cite{KristinFatemeh} in which the authors, following the work of \cite{sturmfels1993maximal,fink2015stiefel}, study the so-called \textit{coherent matching fields}. More precisely, the authors classify which matching fields give rise to toric degenerations of $\Gr(3,6)$ and provide a family of matching fields called block diagonal matching fields that exhibit, up to isomorphism, all but one of the possible Gr\"obner degenerations of $\Gr(3,6)$.
In \cite{OllieFatemeh2}, it is shown that the weight vectors arising from block diagonal matching fields give rise to toric degenerations of the flag variety. Furthermore, by \cite{sturmfels1996grobner}, whenever a toric degeneration is obtained via a matching field, the Pl\"ucker variables form a SAGBI basis 
for the corresponding Pl\"ucker algebra.

\medskip

We consider a family of toric degenerations of Schubert varieties which are parametrized by matching fields, in the sense of Sturmfels-{Zelevinsky} \cite{sturmfels1993maximal}. It is shown in \cite{OllieFatemeh2} that all so-called block diagonal matching fields give rise to toric degenerations of the full flag variety. The associated toric ideals can be directly read from the matching field and so are called matching field ideals. In this note, we extend the results of \cite{OllieFatemeh2} by considering how these toric degenerations restrict to certain subvarieties of the flag variety, namely its Schubert varieties. For each Schubert variety, indexed by some permutation $w \in S_n$ and a block diagonal matching field $B_\ell$, we define the restricted matching field ideal by setting some variables of the matching field ideal to zero. Our main results are Theorems~A, B and C.
Theorems~B and C give combinatorial conditions on $w$ and $B_\ell$ such that the restricted matching field ideal is monomial-free. Theorem~A shows that a matching field $B_\ell$ gives rise to a toric degeneration of the corresponding Schubert variety subject to the condition that the initial ideal is generated in degree two, which we show for some particular matching fields.  
Our methods use combinatorial properties of the permutations, which parametrize the Schubert varieties, and properties of the generating sets of matching field ideals. Moreover, we use semi-standard Young tableaux to construct monomial bases for each restricted matching field ideals. As a result, we obtain new families of monomial bases for the full flag variety that are compatible with its Schubert varieties. Moreover, we obtain minimal generating sets of the ideals arising from Gr\"obner degenerations of Schubert varieties. 

\medskip

\noindent{\bf Structure of the paper.} 
In \S\ref{sec:prelim} we give definitions and fix our notation throughout the note. In particular, we define the full flag variety and its Schubert varieties by their defining ideals, see \S\ref{sec:flag_def} and \S\ref{sec:schubert_def}, respectively. In \S\ref{sec:matching_field_def} we define matching fields along with a particular family called block diagonal matching fields, see Definition~\ref{def:block_diagonal_matching_field}. 
In \S\ref{sec:init_schubert_intro} we introduce the restricted matching field ideals and matching field tableaux, see Definitions~\ref{def:ideals} and \ref{def:ssyt_matching_field_tableaux} respectively.
In \S\ref{sec:schubert} we state our main results. This includes Theorem~A which relates the monomial-free restricted matching field ideals with the initial ideals of Schubert varieties, Theorems~B and C which characterize the family of binomial, zero and non-binomial ideals and Theorem~\ref{thm:P_ell=T_n+Z_n} which is a non-inductive reformulation of Theorem~C. To explain these results clearly, we give examples and use Figure~\ref{flowchart:inductive_reln} to give a visual representation of these results. In \S\ref{sec:pf_thm_a} we give the proof of Theorem~B, which is broken into three claims. In \S\ref{sec:pf_thm_b} we turn our attention to non-zero binomial ideals and the proof of Theorem~C. In Figure~\ref{flowchart:dependency} we display the dependency relations among the results required for the proof of Theorem~C. We then proceed to prove each of the three parts of Theorem~C in the subsequent subsections:  \S\ref{sec:proofs_diag}, \S\ref{sec:proofs_semi_diag} and \S\ref{sec:proofs_other}, which relate to diagonal, semi-diagonal and the other block diagonal matching fields respectively.
{In \S\ref{sec:standard_monomials} we study monomial bases for the restricted matching field ideals and prove Theorem~A. In \S\ref{sec:mon_bases_gen_set} we prove results about the generating sets of restricted matching field ideals. A detailed proof of Theorem~A is given in \S\ref{sec:mon_bases_main}.}

\medskip

\noindent{\bf Acknowledgement.} We thank  Narasimha Chary and J\"urgen Herzog
for 
many helpful conversations. We are  grateful to the anonymous referees for very helpful comments on earlier versions of this paper.
FM was partially supported by a BOF Starting Grant of Ghent University and EPSRC Early Career Fellowship EP/R023379/1.
OC is supported by EPSRC Doctoral Training Partnership (DTP) award EP/N509619/1.

\section{Preliminaries}\label{sec:prelim}
Throughout we fix a field $\mathbb{K}$ with char$(\mathbb{K})=0$. We are mainly interested in the case when 
$\mathbb{K}=\mathbb{C}$. We let $[n]$ be the set $\{1, \dots, n \}$ and by $S_n$ we denote the symmetric group on $[n]$. A permutation $w \in S_n$,  unless stated otherwise, is written $w = (w_1, \dots, w_n)$ where $w_i = w(i)$ for each $1 \le i \le n$, which is often called single line notation. It will be convenient for us to have the elements of a set be in increasing order so we write $J = \{j_1 < \dots < j_s \}$ for the set with elements $j_1, \dots, j_s$ in increasing order. However, unless otherwise stated, sets are not ordered.

\subsection{Flag varieties}\label{sec:flag_def} 
A full flag is a sequence of vector subspaces of $\mathbb{K}^n$: $$\{0\}= V_0\subset V_1\subset\cdots\subset V_{n-1}\subset V_n=\mathbb{K}^n$$ where ${\rm dim}_{\mathbb{K}}(V_i) = i$. The set of all full flags is called the flag variety denoted by $\Flag_n$, which is naturally embedded  in a product of Grassmannians using the Pl\"ucker variables.
Each point in the flag variety can be represented by an $n\times n$ matrix $X=(x_{i,j})$ whose first $k$ rows span $V_k$. Each $V_k$ corresponds to a point in the Grassmannian $\Gr(k,n)$. The ideal of $\Flag_n$, denoted by $I_n$ is the kernel of the polynomial map
\vspace{-1mm}
\[
\varphi_n:\  \mathbb{K}[P_J:\ \varnothing\neq J\subsetneq \{1,\ldots,n\}]\rightarrow \mathbb{K}[x_{i,j}:\ 1\leq i\leq n-1,\ 1\leq j\leq n]
\]
sending each Pl\"ucker variable
$P_J$ to the determinant of the submatrix of $X$ with row indices $1,\ldots,|J|$ and column indices in $J$.  We refer to \cite[\S14.2]{MS05} for a detailed introduction to Pl\"ucker ideals.

\begin{remark}
{By abuse of notation we use $P_J$ to denote both the variable in the ring $\mathbb{K}[P_J]$ and also for the image of $P_J$ under the map $\varphi_n$. Later we will introduce the notion of weights on variables in both rings $\mathbb{K}[P_J]$ and $\mathbb{K}[x_{i,j}]$. However, the weight on $P_J$ will be induced by the weights on $x_{i,j}$ so this abuse of notation will not cause problems for weights.}
\end{remark}

\subsection{Schubert varieties} \label{sec:schubert_def}
Let SL$(n,\mathbb C)$ be the set of $n\times n$ matrices with determinant $1$,  and let $B$ be its subgroup consisting of upper triangular matrices.  There is a natural transitive action of SL$(n,\mathbb C)$ on the flag variety $\Flag_n$ which identifies $\Flag_n$ with the set of left cosets SL$(n,\mathbb C)/B$, since $B$ is the stabilizer of the standard flag
$0\subset \langle e_1\rangle \subset\cdots \subset \langle e_1, \ldots, e_n\rangle=\mathbb C^n $. Given a permutation $w\in S_n$, we denote by $\sigma_w$ the $n\times n$ permutation matrix with 1's in the positions $(w(i),i)$ for all $i$.
By the Bruhat decomposition, we can write the aforementioned set of cosets as 
${\rm SL}(n,\mathbb C)/B= \coprod_{w\in S_n}B\sigma_wB/B.$
Given a permutation $w$, its Schubert variety is $$X(w)=\overline{B\sigma_wB/B} \subseteq \Flag_n$$ which is the Zariski closure of the corresponding cell in the Bruhat decomposition. 
The ideal of the Schubert variety $X(w)$ is obtained from $I_n$ by setting $P_J$ to zero for each $J \in S_w$
where
$$
S_w=\{J : J\subset[n] \ \text{with} \ J \not\leq \{ w_{1},w_{2},\ldots,w_{|J|} \} \}.
$$
Where $\{a_1 < \dots < a_m \} \le \{b_1 < \dots < b_m \}$ means that $a_i \le b_i$ for each $1 \le i \le m$.

\begin{example}\label{example:s_w_schubert_ideal}
Suppose $n = 4$ and $w = (3,2,1,4) \in S_n$ is a permutation written in single line notation. To calculate $S_w$ we take each subset of $[n]$, for example $\{1,2 \} \subseteq [4]$, and compare it to $\{w_1, w_2 \}$. In this case $\{1,2 \} \le \{2,3\} = \{ w_1, w_2\}$ and so $\{1,2 \} \not \in S_w$. Continuing this process for all other subsets we obtain
\[
S_w = \{4, 14, 24, 34, 124, 134, 234\}.
\]
The ideal of $X(w)$ is obtained from $I_n$ by setting $P_J$ to zero for each $J \in S_w$:
\[
I(X(w))=\langle P_3 P_{12} - P_2 P_{13} + P_1 P_{23} \rangle \subset \mathbb{K}[P_J].
\]
\end{example}
\subsection{Matching fields}\label{sec:matching_field_def}

\begin{definition}\label{def:matching_field}
A \emph{matching field} is a map $\Lambda_n : \{J: \varnothing \neq J \subsetneq [n] \} \rightarrow S_n$.
For ease of notation we write $\Lambda$ for $\Lambda_n$ if there is no ambiguity. Suppose $J = \{j_1 < \dots < j_k \} \subset [n]$,
we think of the permutation $\sigma=\Lambda(J)$ as inducing an ordering on the elements of $J$, 
where the position of $j_s$ is $\sigma(s)$.
\end{definition}
Given a matching field $\Lambda$ and a $k$-subset $J = \{j_1, \dots, j_k \} \subset [n]$ with $j_1 < \dots < j_k$, let $\sigma = \Lambda(J)$. We represent the Pl\"ucker form $P_J$ as a $k \times 1$ tableau whose entry in position $(\sigma(\ell),1)$ is $j_\ell$ for each $1 \le \ell \le k$. Let $X=(x_{i,j})$ be an $n \times n$ matrix of indeterminates. To each subset $J \subset [n]$ as before, we associate the monomial 
$\textbf{x}_{\Lambda(J)}:=x_{\sigma(1) j_{1}}x_{\sigma(2)j_2}\cdots x_{{\sigma(k)j_k}}.$
A {\em matching field ideal} $J_\Lambda$ is defined as the kernel of the monomial map
\begin{eqnarray}\label{eqn:monomialmap_Lambda}
\phi_{\Lambda} \colon\  & \mathbb{K}[P_J]  \rightarrow \mathbb{K}[x_{ij}]  
\quad\text{with}\quad
 P_{J}   \mapsto \text{sgn}(\Lambda(J)) \textbf{x}_{\Lambda(J)},
\end{eqnarray}
where sgn denotes the sign of the permutation $\Lambda(J)$. We define the \emph{algebra associated to $\Lambda$} to be $\mathbb{K}[P_J] / \ker(\phi_{\Lambda})$.

\begin{definition}
A matching field $\Lambda$  is \emph{coherent} if there exists an $n\times n$ matrix $M$ with entries in $\mathbb{R}$ 
such that for every proper non-empty subset $\varnothing \neq J\subsetneq [n]$ 
the initial form of the  Pl\"ucker form  $P_J \in \mathbb{K}[x_{ij}]$,  $\text{in}_M (P_J) $ is $ \text{sgn}(\Lambda(J)) \mathbf{x}_{\Lambda(J)}$. Where $\text{in}_M (P_J)$ is the sum of all terms in $P_J$ with the lowest weight with respect to $M$.
In this case, we say that the matrix $M$ \emph{induces the matching field} $\Lambda$.
\end{definition}

\begin{example}
Let us see an example of a non-coherent matching field. Suppose that $n \ge 4$ and we have a matching field $\Lambda$ such that $\Lambda(\{1,2,3 \}) = id$ and $\Lambda(\{1,2,4 \}) = (1,2)$ is the transposition which swaps $1$ and $2$. Suppose by contradiction that $\Lambda$ is a coherent matching field, then there exists an $n \times n$ matrix $M$ which induces $\Lambda$. Let us consider the submatrix $M'$ of $M$ which consists of the first two rows and first two columns.
\[
M' = \begin{bmatrix}
m_{1,1} & m_{1,2} \\
m_{2,1} & m_{2,2}
\end{bmatrix}.
\]
Since $\Lambda(\{1,2,3 \}) = id$, this implies that $m_{1,1} + m_{2,2} < m_{1,2} + m_{2,1}$. However $\Lambda(\{1,2,4 \}) = (1,2)$ implies that $m_{1,2} + m_{2,1} < m_{1,1} + m_{2,2}$, a contradiction.
\end{example}

\begin{definition}\label{def:matching_field_weight}
Let $\Lambda$ be a coherent matching field induced by the matrix $M$. We define ${\bf w}_{M}$ to be the \emph{weight vector} induced by $M$ on the Pl\"ucker variables. That is, the entry of the vector ${\bf w}_{M}$ corresponding to the variable $P_J \in \mathbb{K}[P_J]$ is the minimum weight of monomials appearing in $\varphi_n(P_J)$ with respect to $M$. The weight of a monomial is the sum of the corresponding terms in the weight matrix $M$. For ease of notation we write ${\bf P}^\alpha$ for the monomial $P_{J_1}^{\alpha_1} \dots P_{J_s}^{\alpha_s}$ where $\alpha = (\alpha_1, \dots, \alpha_s)$. And so the weight of ${\bf P}^\alpha$ is simply $\alpha \cdot {\bf w}_{M} $.
\end{definition}

\begin{definition}\label{def:initial}
Let $\Lambda$ be a coherent matching field induced by $M$. We denote 
\emph{the initial ideal of $I_n$ with respect to $\wb_{M}$} by $\inw{M}(I_n)$. The ideal $\inw{M}(I_n)$ is generated by polynomials $\inw{M}(f)$ for all $f\in I_n$, where 
\[\inw{M}(f)=\sum_{\alpha_j\cdot \wb_{M}=d}{c_{{\bf \alpha}_j}\bf P}^{{\bf \alpha}_j}\quad\text{for}\quad f=\sum_{i=1}^t c_{{\bf \alpha}_i}{\bf P}^{{\bf \alpha}_i}\quad\text{and}\quad d=\min\{\alpha_i\cdot \wb_{M}:\ i=1,\ldots,t\}.\]
\end{definition}

\begin{example}\label{example:block_diag_matching_field_2_2}
Consider the matching field $\Lambda$ induced by the matrix
\[
M = 
\begin{bmatrix}
0  &0  &0  &0  \\
2  &1  &4  &3  \\
8  &6  &4  &2  \\
12 &9  &6  &3  \\
\end{bmatrix}\, .
\]

The single column tableaux arising from the matching field are:
\[
\begin{tabular}{|c|} \hline 1 \\ \hline \multicolumn{1}{c}{\,} \\ \multicolumn{1}{c}{\,} \\ \end{tabular}\, \,
\begin{tabular}{|c|} \hline 2 \\ \hline \multicolumn{1}{c}{\,} \\ \multicolumn{1}{c}{\,} \\ \end{tabular}\, \,
\begin{tabular}{|c|} \hline 3 \\ \hline \multicolumn{1}{c}{\,} \\ \multicolumn{1}{c}{\,} \\ \end{tabular}\, \,
\begin{tabular}{|c|} \hline 4 \\ \hline \multicolumn{1}{c}{\,} \\ \multicolumn{1}{c}{\,} \\ \end{tabular}\, \,
\begin{tabular}{|c|} \hline 1 \\ 2 \\ \hline \multicolumn{1}{c}{\,} \\ \end{tabular}\, \,
\begin{tabular}{|c|} \hline 3 \\ 1 \\ \hline \multicolumn{1}{c}{\,} \\ \end{tabular}\, \,
\begin{tabular}{|c|} \hline 4 \\ 1 \\ \hline \multicolumn{1}{c}{\,} \\ \end{tabular}\, \,
\begin{tabular}{|c|} \hline 3 \\ 2 \\ \hline \multicolumn{1}{c}{\,} \\ \end{tabular}\, \,
\begin{tabular}{|c|} \hline 4 \\ 2 \\ \hline \multicolumn{1}{c}{\,} \\ \end{tabular}\, \,
\begin{tabular}{|c|} \hline 3 \\ 4 \\ \hline \multicolumn{1}{c}{\,} \\ \end{tabular}\, \,
\begin{tabular}{|c|} \hline 1 \\ 2 \\ 3 \\ \hline \end{tabular}\, \,
\begin{tabular}{|c|} \hline 1 \\ 2 \\ 4 \\ \hline \end{tabular}\, \,
\begin{tabular}{|c|} \hline 3 \\ 1 \\ 4 \\ \hline \end{tabular}\, \,
\begin{tabular}{|c|} \hline 3 \\ 2 \\ 4 \\ \hline \end{tabular}\, .
\]
So the Pl\"ucker variables are 
\[P_1, P_2, P_3, P_4, P_{12}, P_{31}, P_{41}, P_{32}, P_{42}, P_{34}, P_{123}, P_{124}, P_{314}, P_{324}
\]
and the corresponding weight vector is 
$\wb_{M} = (0,0,0,0,1,2,2,1,3,3,5,3,4,3)$. Performing the calculation in $\mathtt{Macaulay2}$ \cite{M2}, we obtain the following generating set for $\inw{M}(I_4)$:
\begin{multline*}
    \inw{M}(I_4) = \langle 
    P_{42} P_{314} - P_{41} P_{324}, 
    P_{32} P_{314} - P_{31} P_{324}, 
    P_{32} P_{124} - P_{12} P_{324},
    P_{31} P_{124} - P_{12} P_{314}, \\ 
    P_{41} P_{32}  - P_{31} P_{42}, 
    P_{3}  P_{124} + P_{1}  P_{324}, 
    P_{4}  P_{32} -  P_{3}  P_{42},
    P_{4}  P_{31} -  P_{3}  P_{41},
    P_{4}  P_{12} +  P_{1}  P_{42},
    P_{3}  P_{12} +  P_{1}  P_{32}
\rangle \, .
\end{multline*}
\end{example}

\begin{example}\label{example:diagonal_matching_field}

Let $\Lambda$ be the matching field induced by the matrix
\[
M = 
\begin{bmatrix}
0  &0  &0  &0  \\
4  &3  &2  &1  \\
8  &6  &4  &2  \\
12 &9  &6  &3
\end{bmatrix}
\]
The single column tableaux arising from the matching field are:
\[
\begin{tabular}{|c|} \hline 1 \\ \hline \multicolumn{1}{c}{\,} \\ \multicolumn{1}{c}{\,} \\ \end{tabular}\, \,
\begin{tabular}{|c|} \hline 2 \\ \hline \multicolumn{1}{c}{\,} \\ \multicolumn{1}{c}{\,} \\ \end{tabular}\, \,
\begin{tabular}{|c|} \hline 3 \\ \hline \multicolumn{1}{c}{\,} \\ \multicolumn{1}{c}{\,} \\ \end{tabular}\, \,
\begin{tabular}{|c|} \hline 4 \\ \hline \multicolumn{1}{c}{\,} \\ \multicolumn{1}{c}{\,} \\ \end{tabular}\, \,
\begin{tabular}{|c|} \hline 1 \\ 2 \\ \hline \multicolumn{1}{c}{\,} \\ \end{tabular}\, \,
\begin{tabular}{|c|} \hline 1 \\ 3 \\ \hline \multicolumn{1}{c}{\,} \\ \end{tabular}\, \,
\begin{tabular}{|c|} \hline 1 \\ 4 \\ \hline \multicolumn{1}{c}{\,} \\ \end{tabular}\, \,
\begin{tabular}{|c|} \hline 2 \\ 3 \\ \hline \multicolumn{1}{c}{\,} \\ \end{tabular}\, \,
\begin{tabular}{|c|} \hline 2 \\ 4 \\ \hline \multicolumn{1}{c}{\,} \\ \end{tabular}\, \,
\begin{tabular}{|c|} \hline 3 \\ 4 \\ \hline \multicolumn{1}{c}{\,} \\ \end{tabular}\, \,
\begin{tabular}{|c|} \hline 1 \\ 2 \\ 3 \\ \hline \end{tabular}\, \,
\begin{tabular}{|c|} \hline 1 \\ 2 \\ 4 \\ \hline \end{tabular}\, \,
\begin{tabular}{|c|} \hline 1 \\ 3 \\ 4 \\ \hline \end{tabular}\, \,
\begin{tabular}{|c|} \hline 2 \\ 3 \\ 4 \\ \hline \end{tabular}\, .
\]
So the Pl\"ucker variables are $P_1, P_2, P_3, P_4, P_{12}, P_{13}, P_{14}, P_{23}, P_{24}, P_{34}, P_{123}, P_{124}, P_{134}, P_{234}$ and the corresponding weight vector is $\wb_{M} = (0,0,0,0,1,2,2,1,3,3,5,3,4,3)$. Performing the calculation in $\mathtt{Macaulay2}$ \cite{M2}, we obtain the following generating set for $\inw{M}(I_4)$:
\begin{multline*}
    \inw{M}(I_4) = \langle 
    P_{24} P_{134} - P_{14} P_{234}, 
    P_{23} P_{134} - P_{13} P_{234}, 
    P_{23} P_{124} - P_{12} P_{234},
    P_{13} P_{124} - P_{12} P_{134}, \\ 
    P_{14} P_{23}  - P_{13} P_{24}, 
    P_{2}  P_{134} - P_{1}  P_{234}, 
    P_{3}  P_{24} -  P_{2}  P_{34},
    P_{3}  P_{14} -  P_{1}  P_{34},
    P_{2}  P_{14} -  P_{1}  P_{24},
    P_{2}  P_{13} -  P_{1}  P_{23},
\rangle \, .
\end{multline*}

\smallskip

Note that the entries in each tableau are strictly increasing. We call such matching fields \emph{diagonal} and we denote them by $D_n$, or $D$ when there is no confusion. Their corresponding degenerations are called Gelfand-Tsetlin degenerations in \cite{KOGAN}. 
\end{example}

\begin{definition}\label{def:block_diagonal_matching_field}
Given $n$ and $0\leq\ell\leq n$, we define the \emph{block diagonal matching field}
denoted by $B_\ell=(1\cdots \ell|\ell+1\cdots n)$
as a map from the power set of $[n]=\{1,\ldots,n\}$ to $S_n$ such that
\[
 B_\ell(J)= \left\{
     \begin{array}{@{}l@{\thinspace}l}
      id  &: \text{if $\lvert J|=1$ or $\lvert J \cap \{1,\ldots,\ell\}\rvert \ge 2$}\\
      (12)  &: \text{otherwise} \\
     \end{array}
   \right.
\]
\end{definition}
The matching field $B_\ell$ is induced by the weight matrix:

\begin{center}
\resizebox{0.95 \textwidth}{!}{
$
\displaystyle{
M_\ell=\begin{bmatrix}
    0       & 0     & \cdots    & 0             &0              & 0             & \cdots    & 0  \\
    \ell    &\ell-1 & \cdots    & 1             &n              &n-1            & \cdots    & \ell+1\\
    2n      &2(n-1) & \cdots    &2(n-\ell+1)    &2(n-\ell)      &2(n-\ell-1)    & \cdots    & 2  \\
    \vdots  &\vdots & \ddots    & \vdots        &\vdots         & \vdots        & \ddots    & \vdots   \\
    (n-1)n  &(n-1)^2& \cdots    &(n-1)(n-\ell+1)& (n-1)(n-\ell) &(n-1)(n-\ell-1)& \cdots    & n-1    \\
\end{bmatrix}.
}
$
}
\end{center}

Therefore, all block diagonal matching fields are coherent. We denote ${\bf w}_\ell$ for the weight vector induced by $M_\ell $ on the Pl\"ucker variables. The case $\ell=0$ or $n$ corresponds to the diagonal matching field.  The weight vector ${\bf w}_\ell$ is explicitly given as follows. For each $J = \{j_1 < \dots < j_s \} \subset [n]$ the component of ${\bf w}_\ell$ corresponding to $P_J$ is given by
\[
{\bf w}_\ell(P_J) = \left\{
\begin{tabular}{ll}
    $0$, & if $s = 1$, \\
    $(n+\ell + 1 - j_2) + \sum_{k = 3}^s (k - 1)(n+1-j_k)$, & if $s \ge 1$ and $|J \cap \{1, \ldots, \ell \}| = 0 $, \\
    $(\ell + 1 - j_1) + \sum_{k = 3}^s (k - 1)(n+1-j_k)$, & if $s \ge 1$ and $|J \cap \{1, \ldots, \ell \}| = 1$, \\
    $(\ell+1 - j_2) + \sum_{k = 3}^s (k - 1)(n+1-j_k)$, & if $|J \cap \{1, \ldots, \ell \}| \ge 2$.
\end{tabular}\right.
\]

With the above notation,
we denote $F_{n,\ell}$ for the matching ideal of $B_\ell$ which is 
the kernel of the monomial map \begin{eqnarray}\label{eqn:monomialmap}
\phi_\ell \colon\  \mathbb{K}[P_J]  \rightarrow \mathbb{K}[x_{ij}]
\quad\text{with}\quad
 P_{J}   \mapsto {\rm sgn}(B_\ell(J)) \ini_{{\bf w}_\ell}(P_J).
\end{eqnarray}
The following result highlights the motivation and the importance of the understudied family of degenerations induced by matching fields.

\begin{corollary}[{\cite[Corollary~4.13 and Theorem~3.3]{OllieFatemeh2}}]
\label{cor:block_diag_degen_flag}
Each block diagonal matching field produces a toric degeneration of $\Flag_n$. Equivalently, $\inwb(I_n)$ is toric for all $n$ and $0 \le \ell \le n$, and it equals to $F_{n,\ell}$. Moreover, the ideal $F_{n,\ell}$ is generated by quadratic binomials.
\end{corollary}

\subsection{Initial ideals of Schubert varieties inside \texorpdfstring{$\Flag_n$}{Fln}
}\label{sec:init_schubert_intro}

{
Here we introduce the family of ideals $F_{n,\ell,w}$ that are closely related to the initial ideals $\inwb(I(X(w)))$ where $I(X(w))$ is the ideal of the corresponding Schubert variety. In general, the initial ideals of Schubert varieties are difficult to calculate, see Remark~\ref{rmk:compute_init_schu}. However the ideals $F_{n,\ell,w}$, which arise from matching fields, have a canonical generating set which we exploit in order to generalize Corollary~\ref{cor:block_diag_degen_flag} to Schubert varieties.
}
{
\begin{definition}[Restricted matching field ideals]\label{def:ideals}
Given a block diagonal matching field $B_\ell$ and a permutation $w$ in $S_n$, we define the ideal 
\begin{eqnarray}\label{eq:F_n,l,w}
F_{n,\ell,w} = (F_{n,\ell} + \langle P_J : J \in S_w \rangle) \cap \mathbb K[P_J : J \subseteq [n], \ J \notin S_w],
\end{eqnarray}
which can be computed in $\mathtt{Macaulay2}$ \cite{M2} as an elimination ideal as follows
\[
F_{n,\ell,w} = {\tt eliminate }({\rm in}_{{\bf w}_\ell}(I_n)+\langle P_J:\, J\in S_w\rangle,\ \{P_J:\, J\in S_w\}).
\]
\end{definition}
We may think of $F_{n,\ell,w}$ as the ideal obtained from $F_{n,\ell} = {\rm in}_{{\bf w}_\ell}(I_n)$ by setting the variables $\{P_J : J \in S_{w} \}$ to be zero. And so we say that the variable $P_J$ vanishes in 
$F_{n,\ell,w}$ if 
$J \in S_w$. If $P_J$ does not vanish we write $P_J \neq 0$. More generally, we say that a polynomial $g \in \mathbb K[P_I]$ vanishes in $F_{n,\ell,w}$ if $g \in \langle P_I : I \in S_w \rangle \subseteq \mathbb K[P_I]$. We will often use the language of vanishing polynomials when determining which terms of generators in $F_{n,\ell}$ vanish in $F_{n,\ell,w}$.
}
\begin{example}
Let us continue Example~\ref{example:s_w_schubert_ideal} where $n = 4$ and $w = (3,2,1,4)$. Consider the matching field from Example~\ref{example:block_diag_matching_field_2_2} which is the block diagonal matching field $B_{2}$. We begin by calculating the initial ideal $\init_{{\bf w}_2}(I_4)$ which is 
\begin{align*}
\langle & 
P_{13}P_{124}-P_{12}P_{134}, \ 
P_{24}P_{134}-P_{14}P_{234}, \ 
P_{23}P_{134}-P_{13}P_{234}, \\
& P_{23}P_{124}-P_{12}P_{234}, \
P_{3}P_{124}+P_{1}P_{234}, \ 
P_{14}P_{23}-P_{13}P_{24}, \\ 
&P_{4}P_{12}+P_{1}P_{24}, \
P_{4}P_{13}-P_{3}P_{14}, \ 
P_{3}P_{12}+P_{1}P_{23}, \ 
P_{4}P_{23}-P_{3}P_{24} \rangle.
\end{align*}
So we can now calculate the ideal $F_{4,2,w}$ which is
$F_{4,2,w} = \langle P_3 P_{12} - P_1 P_{23} \rangle $. Note that this is the same as finding $\init_{{\bf w}_\ell}(I(X(w)))$ where $I(X(w))$ is the ideal of the Schubert variety which we found in Example~\ref{example:s_w_schubert_ideal}. Also note that the resulting ideal $F_{n,\ell,w}$ is binomial. This also follows from Theorem~C and in particular the other matching fields $\BLambda$ and permutations $w$ which give rise to binomial ideas, where $n=4$, can be found in Table~\ref{fig:num4}.
\end{example}

One of our main results is Theorem~A, which shows that if a restricted matching field ideal is monomial-free then it coincides with initial ideals of the corresponding Schubert variety. To prove this result, we show that semi-standard Young tableaux are in bijection with a set of standard monomials for $F_{n,\ell,w}$.

\begin{definition}[semi-standard Young tableaux]\label{def:ssyt_matching_field_tableaux} 
A tableau $T = [I_1 I_2 \dots I_k]$ is an ordered collection of columns where each column is an ordered subset $I_j \subseteq [n]$ for each $j \in [k]$. If the order of the entries in each column coincides with the order induced by a fixed matching field $B$, then we say the tableau is a \textit{matching field tableau} for $B$. Write $I_j = \{i_{1,j}, i_{2,j}, \dots i_{t_j, j} \}$ for each $j \in [k]$.  We say $T$ is a semi-standard Young tableau if the following hold.
\begin{itemize}
    \item The size of the columns weakly decreasing, i.e. if $1 \le i < j \le k$ then $|I_i| \ge |I_j|$.
    \item The entries in each column are increasing, i.e. $I_j = \{i_{1,j} < i_{2,j} < \dots < i_{t_j, j}\}$ for each $j \in [k]$.
    \item The entries in each row are weakly increasing, i.e. $i_{j, 1} \le i_{j, 2} \le \dots i_{j, m_j}$ for each $j \in \{1, \dots, |I_1| \}$ where $m_j = \max\{i \in [k] : |I_i| \ge j\}$.
\end{itemize}
\end{definition}

\begin{example}
Let $n = 4$ and consider the tableaux below.
\[
T_1 = 
\begin{tabular}{cc}
    \hline
    \multicolumn{1}{|c|}{1} & \multicolumn{1}{c|}{2} \\ \hline
    \multicolumn{1}{|c|}{2} & \multicolumn{1}{c|}{3} \\ \hline
    \multicolumn{1}{|c|}{4} &   \\ \cline{1-1}
\end{tabular} \ ,
\quad
T_2 = 
\begin{tabular}{cc}
    \hline
    \multicolumn{1}{|c|}{3} & \multicolumn{1}{c|}{1} \\ \hline
    \multicolumn{1}{|c|}{1} & \multicolumn{1}{c|}{2} \\ \hline
    \multicolumn{1}{|c|}{4} &   \\ \cline{1-1}
\end{tabular} \ .
\]
The tableau $T_1$ is a semi-standard Young tableau. The monomial represented by $T_1$ is the image of $P_{125}P_{24}$ under the diagonal matching field map: $\phi_{B_0}(P_{124}P_{23}) = x_{1,1}x_{1,2}x_{2,2}x_{2,3}x_{3,4}$. The tableau $T_2$ is not a semi-standard Young tableau. However the columns of $T_2$ are ordered by the matching field $B_1$, see Example~\ref{example:block_diag_matching_field_2_2}, and so $T_2$ is called a matching field tableau. The tableau $T_2$ represents the image of $P_{134}P_{12}$ under the block diagonal matching field map: $\phi_{B_1}(P_{134}P_{12}) = x_{1,3}x_{1,1}x_{2,1}x_{2,2}x_{3,4}$.
\end{example}

In order to characterize permutations $w \in S_n$ for which the ideals $F_{n, \ell, w}$ are monomial free, see Theorem~\ref{thm:P_ell=T_n+Z_n}, we require the following definitions about permutations.
\begin{definition}[Permutation avoidance]
We say that two finite sequences $w = (w_1, \dots, w_s)$ and $v = (v_1, \dots, v_s)$ have the same \emph{type} if their respective entries satisfy all the same pairwise comparisons, i.e. $w_i < w_j$ if and only if $v_i < v_j$ for all $i, j \in [s]$. We say that a permutation $w = (w_1, \dots, w_n) \in S_n$ \emph{avoids} another permutation $v \in S_m$ where $m \leq n$ if every subsequence $(w_{i_1}, \dots, w_{i_m})$ of $w$ has a different type to $v$. If $w$ avoids $v$ then we also say that $w$ is $v$-free.
\end{definition}

\begin{example}
The sequences $(4,1,2,3)$ and $(6,1,2,5)$ have the same type but neither has the same type as $(5,3,1,4)$. The permutation $(1,5,2,4,3)$ does not avoid $(1,4,3,2)$ because the subsequence $(1,5,4,3)$ has the same type as $(1,4,3,2)$. However, the permutation $(1,5,2,4,3)$ does avoid $(2,3,1)$.
\end{example}

\begin{definition}
Let $w = (w_1, \dots, w_n) \in S_n$ be a permutation and $m \le n$ be a natural number. The \emph{restriction} of $w$ to $[m]$ is the permutation $w|_m \in S_m$ obtained from $w$ by removing the values $m+1, \dots, n$.
\end{definition}

\begin{example}
Let $w = (1,4,2,3)$ then the restrictions of $w$ are as follows.
\[
w|_4 = (1,4,2,3), \quad w|_3 = (1,4,2), \quad w|_2 = (1,2), \quad w|_1 = (1). 
\]
\end{example}

\vspace{-5mm}
\section{Schubert varieties inside flag varieties}\label{sec:schubert}
This section aims to answer the following question on Schubert varieties; this is a reformulation of {\em Degeneration Problem} posed by Caldero \cite{caldero2002toric} in our setting.

\begin{question}\label{question:flag}
{Characterize toric initial ideals of the Pl\"ucker ideals of Schubert varieties inside flag varieties. In other words, determine the toric ideals of form
$\textrm{in}_{\textbf{w}_\ell}(I(X(w)))$. 
}\end{question}

{
In \S\ref{sec:standard_monomials} 
we study the relationship between the ideals $\inwb(I(X(w)))$ and $F_{n,\ell,w}$ by way of standard monomial theory and prove the following result.
}

\medskip

\noindent\textbf{Theorem A.}
Suppose that $\inwb(I(X(w)))$ is generated in degree two. If $F_{n,\ell,w}$ is monomial-free then $F_{n,\ell,w} = \inwb(I(X(w)))$. Moreover $\textrm{in}_{{\bf w}_\ell}(I(X(w))$ is the kernel of a monomial map, hence it is a toric (prime binomial) ideal.

\medskip

As an immediate corollary of Theorem~A and \cite[Theorem~11.4]{sturmfels1996grobner} we have that:
\begin{corollary}\label{cor:toric_deg}
The block diagonal matching fields give rise to a family of toric degenerations of the Schubert varieties inside the full flag variety. Moreover, the Pl\"ucker variables $P_I$ form a finite Khovanskii basis for the corresponding Pl\"ucker algebras.
\end{corollary}
\begin{remark}
In the forthcoming paper \cite{Akihiro}, we study the polytopes arising from toric varieties in Corollary~\ref{cor:toric_deg}. In particular, we show that such polytopes are related 
by sequences of combinatorial mutations. 
\end{remark}

\medskip

Our computational results lead us to the following conjecture.
\begin{conjecture}\label{conj:J2_quad_gen}
The ideal $\inwb(I(X(w)))$ is generated in degree two. 
\end{conjecture}

In \S\ref{sec:standard_monomials} we show that this conjecture holds if $\ell = 0$ and $F_{n,\ell,w}$ is monomial-free. We have also verified this conjecture for all ideals $\inwb(I(X(w)))$ where $n \in \{3, 4, 5\}$. If this conjecture holds then the conclusion of Theorem~A holds for all block diagonal matching fields.

\begin{remark}\label{rmk:compute_init_schu}
We use $\mathtt{Macaulay2}$ to calculate the ideals $F_{n,\ell,w}$ and check whether they are toric, i.e. they are non-zero prime binomial ideals. The code is available on Github:
\begin{center}
\texttt{\href{https://github.com/ollieclarke8787/toric_degenerations_schubert_flag}{https://github.com/ollieclarke8787/toric\_degenerations\_schubert\_flag}}
\end{center}
{We verify inclusions of the ideals $F_{n,\ell,w}$ with the ideals $\inwb(I(X(w)))$ where $I(X(w))$ is the ideal obtained from $I_n$ by setting the variables $\{P_J : J \in S_w \}$ to be zero. We perform all calculations for $\Flag_4$ and $\Flag_5$. We also include documentation which allows users to produce similar code for different flag varieties. For $\Flag_6$ our computations did not terminate on a standard desktop computer.
In all cases for which computations terminated, we see that $\inwb(I(X(w)))$ is generated in degree two, verifying Conjecture~\ref{conj:J2_quad_gen} in those cases.
}
\end{remark}

{
To answer Question~\ref{question:flag}, in light of Theorem~A, we provide a complete characterization of ideals of type $F_{n,\ell,w}$ introduced in Definition~\ref{def:ideals} into the categories: zero or non-zero and binomial or non-binomial. An ideal $F_{n,\ell,w}$ is monomial-free, hence toric and equal to $\textrm{in}_{\textbf{w}_\ell}(I(X(w)))$, if and only if $F_{n,\ell,w}$ is either zero or binomial.} In particular, Theorem~B determines which ideals $F_{n,\ell,w}$ are zero and Theorem~C determines which ideals $F_{n,\ell,w}$ are non-zero and binomial.
We illustrate our main results in Figure~\ref{flowchart:inductive_reln} 
by providing a pictorial survey. 

\medskip

\noindent{\bf Notation}. Before stating further results, we fix the following notation.

\begin{itemize}
    \item  From this section and on, $I$ and $J$ will denote subsets of $[n]$ that index variables $P_I$ and $P_J$. This should not be confused with the Pl\"ucker ideal $I_n$. If the ideal does appear, then it will be made clear.
    \item Given a block diagonal matching field $B=(E_1|E_2)$ on $[n-1]$ we denote by $\overline{B}$ the matching field $(E_1 | E_2 \cup n)$ on $[n]$. Similarly, given a block diagonal matching field $B = (E_1|E_2)$ on $[n]$ for $n \ge 2$ with $E_2 \neq \varnothing$, we denote by $\ul{B}$ the block diagonal matching field $(E_1|E_2 \backslash n )$. In which case we say $\ul{B}$ is the restriction of $B$ to $[n-1]$.
    \item   Given a permutation $w = (w_1, \dots, w_{n-1})$ on $[n-1]$ and $t \in \{0,1, \dots, n-1 \}$, we denote by $\overline{w}$ for the permutation $(w_1, \dots, w_t, n, w_{t+1}, \dots, w_{n-1})$ on $[n]$. Similarly, given a permutation $w = (w_1, \dots, w_n)$ on $[n]$ with $w_s = n$, we denote by $\ul{w}$ the permutation $(w_1, \dots, w_{s-1}, w_{s+1}, \dots, w_n)$ on $[n-1]$. Note $\underline w = w|_{n-1}$ is a special example of a restriction.
    \item   If $B$ is the diagonal matching field on $[n-1]$, we can regard this either as the block diagonal matching field $(\varnothing \mid 1, \dots, n-1)$ or $(1, \dots, n-1 \mid \varnothing)$. This gives $\overline{B}$ to be $(\varnothing \mid 1, \dots, n)$, i.e. the diagonal matching field on $[n]$, or $(1, \dots, n-1 \mid n)$, a non-diagonal matching field. Where necessary we distinguish between these, otherwise if left unstated all results apply to both cases.
\end{itemize}

Here, we state our main results on Schubert varieties.

\medskip

\noindent{\bf Theorem~B} (Theorem~\ref{thm:zero}){\bf .} \label{theoremA:corollary:zero_flag}
For each $\ell$,  $F_{n,\ell,w}=0$ if and only if $w\in Z_n$, where
$$Z_{n} = \{ s_{i_1} \dots s_{i_p} \in S_n : \lvert i_k - i_\ell \rvert \ge 2, \text{ for all } k, \ell\}.$$
Here, $s_i=(i,i+1) \in S_n$ is the transposition interchanging $i$ and $i+1$.

\begin{definition}\label{def:zero}
For each block diagonal matching field $B_\ell$, we let
$$T_{n,\ell} = \{w \in S_n : F_{n,\ell,w} \text{ is   binomial} \} \quad\text{and}\quad Z_{n}  = \{w \in S_n : F_{n,n,w} = 0\},$$
along with $N_{n,\ell}  = S_n \backslash (T_{n,\ell} \cup Z_{n} )$ for the set of permutations for which $F_{n,\ell,w}$ is  non-binomial.
Note that $B_n$ is the diagonal matching field denoted by $D$. 
\end{definition}

\begin{definition}
 We say that a permutation $w\in S_n$ has the \textit{descending property} if for $w_t=n$ we have that $n = w_t > w_{t+1} > \dots > w_n$. We denote $S_n^>$ for the set of permutations in $S_n$ with descending property.
\end{definition}

\begin{definition}\label{def:notation}
For each block diagonal matching field $B_\ell$, we let
\begin{itemize}
\item $\ul{Z}_{n}= \{w \in S_n : \ul{w} \in Z_{n-1}\}$,
\item $\ul{T}_{n,\ell}= \{w \in S_n : \ul{w} \in T_{n-1,\ell}\}$,
\item $A_1 =\ul{Z}_{n}\cap\{w\in S_n:\ w_n=n-2\ \text{and}\ \{w_{n-2},w_{n-1}\}=\{n-1, n\}\}$,
\item $A_2 = \ul{T}_{n,\ell}\cap\{w \in S_n: \ul{w} \in S_{n-1}^> \text{ and if } w_s = n-1, w_t = n \text{ then } t \ge s-1\}
   $,
\item $A_3 = \ul{T}_{n,\ell}\cap\{w \in S_n: \ul{w} \in S_{n-1}\backslash S_{n-1}^> \text{ and if } w_s = n-1, w_t = n \text{ then } t \ge s+2 \},$
\item $A'_2 = A_2\backslash\{(n-1, n, n-2, n-3, \dots, 1)\}$,
\item $\tilde{A}_1 = A'_2\cap \ul{T}_{n,n-2}$, where $\ell=n-1$ in $A_2$, 
\item $\tilde{A}_2 = (\ul{T}_{n,n-1}\backslash \ul{T}_{n,n-2})\cap\{w \in S_n:\
     \text{if } w_s = n-1, w_t = n \text{ then } t \ge s+1 \}.$
\end{itemize}
\end{definition}

In the following theorem, we classify all binomial ideals arising from block diagonal matching fields inductively, i.e., in terms of the sets defined above which are themselves written in terms of $T_{n-1,\ell}$ and $Z_{n-1}$. Note that for $n=1$ and $n=2$ we have $Z_n = S_n$ so there are no non-zero ideals of the form $F_{n,\ell,w}$. 
The toric ideals of the form $F_{3,\ell,w}$ appear in Table~\ref{fig:num4} as the binomial ideals.
Note that all the ideals are principal so it is straightforward to determine when these binomial ideals are prime, hence toric.

\begin{figure}
    \centering
    \includegraphics[scale = 0.65]{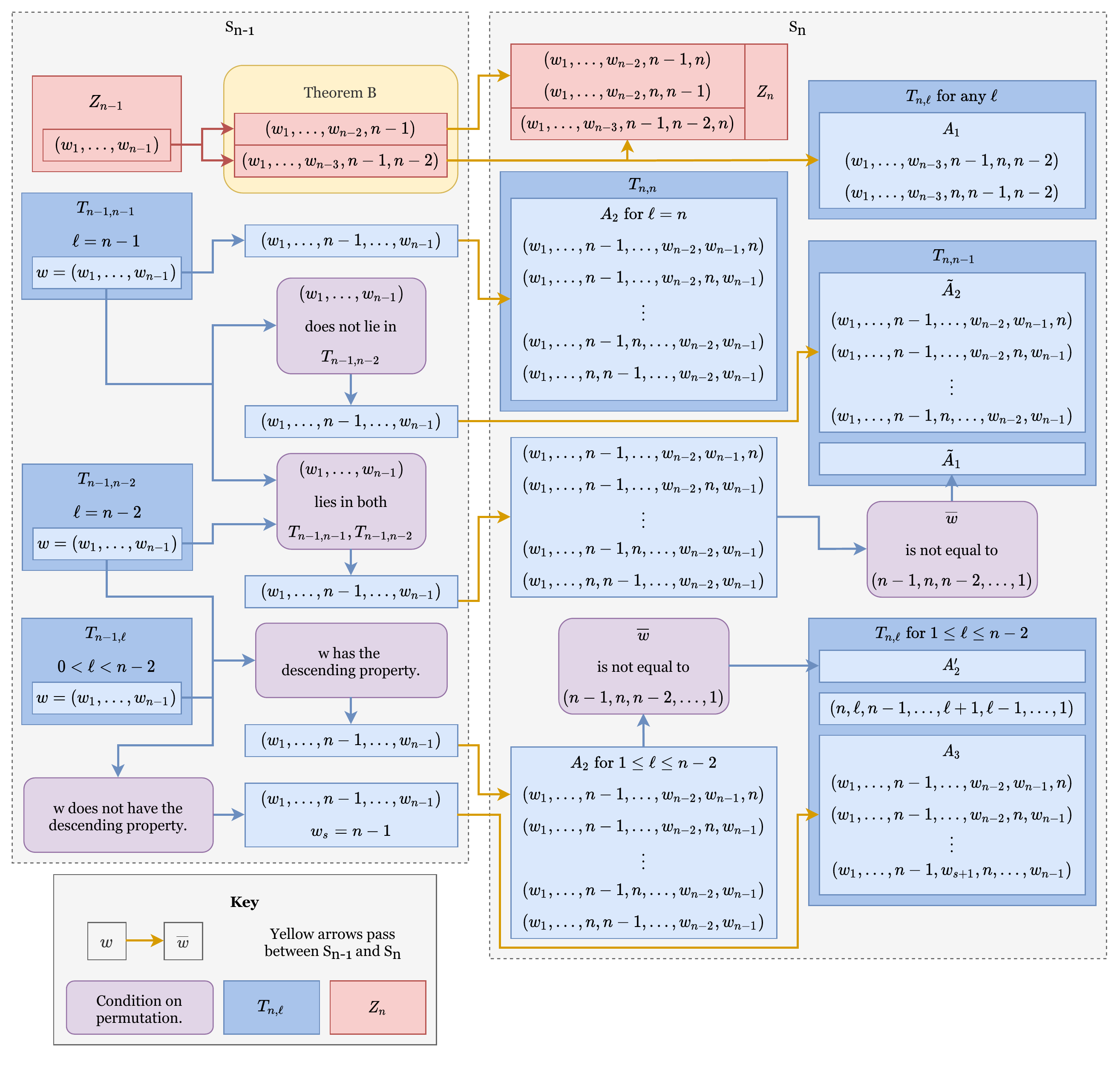}
    \caption{The above diagram depicts Theorem~B and C and shows how each permutation in $T_{n,\ell}$ and $Z_n$ is obtained from $T_{n-1, \ell}$ and $Z_{n-1}$. The starting and ending boxes are shown in darker blue for $T_{n,\ell}$ and red for $Z_n$. In each starting box we fix a permutation $w$. We move to adjacent boxes along arrows until reaching an ending box. Purple boxes are conditions for permutations. A permutation passes through a purple box only if the condition is satisfied. The yellow arrows indicate a transition from $w$ in $S_{n-1}$ to $\overline{w}$ in $S_n$. The boxes before and after a yellow arrow indicate the position in which $n$ is added to $w$ to obtain $\overline{w}$.
    }
    \label{flowchart:inductive_reln}
\end{figure}

\medskip

\noindent{\bf Theorem~C} (Theorems~\ref{Intro:ToricFlag}, \ref{Intro:Semi} and \ref{Intro:l}){\bf .}\label{Intro:ToricFlagBlock1}
Let $n \ge 4$. With the notation above, we have:
\begin{itemize}
      \item[ ] ${\bf C.1.}$ 
      $T_{n,n} = A_1 \cup A_2$, where $\ell = n$,
     \item[ ] ${\bf C.2.}$ 
     $T_{n,n-1} = A_1 \cup \tilde{A}_1 \cup  \tilde{A}_2$, where $\ell = n-1$,
 \item[ ]   ${\bf C.3.}$ 
 $T_{n,\ell} = A_1 \cup A'_2\cup A_3 \cup \{(n, \ell, n-1, n-2, \dots, \ell +1, \ell -1, \dots, 1)\}$ for $1\leq\ell\leq n-2$.
\end{itemize}

\begin{remark}
Note that $\bigcap_{\ell=1}^n T_{n,\ell} \supset A_1$.
For $n=4$, this indicates that the permutations $1342, 1432$ appear in all rows in Table~\ref{fig:num4}. 
\end{remark}

\begin{example}\label{example:n_3_4_calculation}
For $n=3$ and $4$ we have that 
$$Z_3=\{ 123, 132, 213\}\quad {\rm and}\quad Z_4=\{1234, 1243, 1324, 2134, 2143\}.$$
Table~\ref{fig:num4} shows all non-zero ideals $F_{3,\ell,w}$ and all permutations $w$ for which $F_{4,\ell,w}$ is binomial. In each case we have verified that all binomial ideals are in fact prime, hence toric.
In addition, we can calculate the ideals $F_{n,\ell,w}$ for each $3 \le n  \le 6$, $0 \le \ell \le n-1$ and $w \in S_n$. Table~\ref{table:flag_calculation} shows the number of permutations $w \in S_n$, such that $F_{n,\ell,w}$ is  binomial for each given $n$ and $\ell$. For these examples we have also verified that all binomial ideals are prime, and so toric, when $n = 5$.
\end{example}

\begin{table}
\begin{center}
\begin{tabular}{|c|c|c|}
    \hline
  $\ell$             & $w$   & $F_{3,\ell,w}$ \\
    \hline
    \multirow{3}{*}{$0$}        & 231   & $\langle P_2 P_{13} - P_1 P_{23}  \rangle $ \\
                                & 312   & $\langle P_2 P_{13} \rangle $ \\
                                & 321   & $\langle P_2 P_{13} - P_1 P_{23}  \rangle $ \\
    
    \hline
    \multirow{3}{*}{$1$}   & 231   & $\langle P_2 P_{13} \rangle $ \\
                                & 312   & $\langle P_3 P_{21} - P_2 P_{13} \rangle $ \\
                                & 321   & $\langle P_3 P_{21} - P_2 P_{13} \rangle $ \\
                                \hline
    \multirow{3}{*}{$2$}   & 231   & $\langle P_1 P_{32} \rangle $ \\
                                & 312   & $\langle P_3 P_{12} \rangle $ \\
                                & 321   & $\langle P_1 P_{32} - P_3 P_{12} \rangle $ \\
    \hline
\end{tabular}
    \quad
    \begin{tabular}{|c|l|}
        \hline
        $\ell$ & Toric Permutations \\
        \hline
        $0$         & 1342 1432 2314 2341 2431 3214 3241 3421 4321 \\
        \hline
               $1$   & 1342 1432 3124 3142 3214 3241 4132 4321 \\        \hline
 
        $2$   & 1342 1432 3214 3241 4231 4321 \\
        \hline
        $3$   & 1342 1432 2314 2341 3214 3241 4321 \\
        \hline
    \end{tabular}
\end{center}
\caption{The ideals $F_{3,\ell,w}$ where $w \not\in Z_3$ and the list of all $w \not\in Z_4$ for which $F_{4,\ell,w}$ is binomial and prime (toric). 
}\label{fig:num4} 
\end{table}

\begin{table}
    \centering
    \begin{tabular}{|c|cccccc|c|}
        \hline
        Binomial   & \multicolumn{6}{c|}{$\ell$}                   &       \\
        \hline
        $n$     & 0     & 1     & 2     & 3     & 4     & 5     & Total \\
        \hline
        3       & 2     & 1     & 2     &       &       &       & 5     \\
        4       & 9     & 8     & 6     & 7     &       &       & 30    \\
        5       & 34    & 29    & 24    & 26    & 31    &       & 114   \\
        6       & 119   & 99    & 85    & 90    & 104   & 115   & 612   \\
        \hline
    \end{tabular}
    \caption{For each $3 \le n \le 6$ and $0 \le \ell \le n-1$ we calculate the number of permutations $w \in S_n$ for which $F_{n,\ell,w}$ is binomial. For each row of this table where $n \le 5$ we have verified that each binomial ideal is in fact prime, hence toric.}
    \label{table:flag_calculation}
\end{table}

Before giving the proofs of the main results, we state the following corollary which shows that the descending property characterizes \textit{many} of the permutations in $T_{n,\ell}$.

\begin{corollary}\label{cor:non-desc}
For each block diagonal matching field $B_\ell$, there is at most one permutation $w$ for which $F_{n,\ell,w}$ is binomial and $w$ does not have the descending property. More precisely, the only exceptions are for $1 \le \ell \le n-2$ and $w = (n, \ell, n-1, \dots, \ell +1, \ell -1, \dots, 1)$.
\end{corollary}

\begin{proof}
We take cases on $B_\ell$.

\textbf{Case 1.} Let $\ell \in \{n, n-1\}$. By Lemmas~\ref{lem:WDecrease} and \ref{lem:WDec2} we have $T_{n,\ell} \subset S_n^>$ and so $T_{n,\ell} \backslash S_n^> = \varnothing$.

\textbf{Case 2.} Let $\ell \in \{1, \dots, n-2 \}$. By Corollary~\ref{cor:l1...n-2_exceptional} we have that 
\[T_{n,\ell}\backslash S_n^> = \{(n, \ell, n-1, \dots, \ell+1, \ell-1, \dots, 1) \}.\]
\end{proof}

Using the language of permutation avoidance, we give a simple description of the permutations $w$ for which $F_{n,\ell,w}$ is monomial-free.

\begin{definition}\label{def:perms_P_ell}
Fix $\ell \in \{1, \dots, n-1 \}$. We define $P_\ell \subseteq S_n$ to be the collection of permutations $w \in S_n$ such that the following hold.
\begin{itemize}
    \item If $w$ is not $312$-free then $w_1 > w_2 = \ell$ and $w \backslash \ell$ is $312$-free.
    \item If $w|_m = (m-1, m, m-2, \dots, 1)$ for some $3 \le m \le n$ then $w_1 < w_2 \le \ell$ and $w|_{w_2} = (w_1, w_2, w_2 - 1, \dots, w_1 + 1, w_1 - 1, \dots, 1)$.
\end{itemize}
We define $P_n \subseteq S_n$ to be the collection of $312$-free permutations.
\end{definition}

\begin{theorem}\label{thm:P_ell=T_n+Z_n}
The ideal $F_{n,\ell,w}$ is monomial-free if and only if $w \in P_\ell$.
\end{theorem}

Note that  Theorem~C gives an inductive description of these permutations. Showing that the sets of permutations defined in Theorem~C and $P_\ell$ coincide is non-trivial and the proof is given in \S\ref{sec:mon_bases_main}.

\section{Zero initial ideals}\label{sec:pf_thm_a}

In this section, we examine the permutations $w$ for which $F_{n,\ell,w}$ is the zero ideal. We show that the statement $F_{n,\ell,w} = 0$ is independent of the choice of $\ell$ and so we need to  only check the permutation $w$ to decide if $F_{n,\ell,w}$ is zero. In particular, this means that Definition~\ref{def:zero} for $Z_n$ is well-defined. The main result of this section is the following.

\begin{theorem}\label{thm:zero}
For each $\ell$,  $F_{n,\ell,w}=0$ if and only if $w\in Z_n$, where
$$Z_{n} = \{ s_{i_1} \dots s_{i_p} \in S_n : \lvert i_k - i_\ell \rvert \ge 2, \text{ for all } k, \ell\}.$$
Here, $s_i=(i,i+1) \in S_n$ is the transposition interchanging $i$ and $i+1$.
\end{theorem}

\begin{proof}
The result follows from the following claims. 

\medskip

\noindent{\bf Claim 1.} If $F_{n-1,\ell,w}=0$, then 
$F_{n,\ell,\overline{w}}=0$ for
$\overline{w}=(w_1,\ldots,w_{n-1},n)$.

Suppose $F_{n-1,\ell,w} = 0$ where $w = (w_1, \dots, w_{n-1})$. We will take $n \ge 4$ since for $n = 1, 2$ the ideals are trivial and for $n=3$ the direct computation gives the required results, see Table~\ref{fig:num4}.
Suppose that $F_{n,\ell,\overline{w}} \neq 0$ for $\overline{w} = (w_1, \dots, w_{n-1}, n)$. Then there exists a product of variables $P_I P_J$ which appears in $F_{n,\ell,\overline{w}}$ as a monomial or part of a relation $P_I P_J - P_{I'} P_{J'}$. Since $P_I P_J$ does not vanish in $F_{n,\ell,\overline{w}}$, then $\lvert I \rvert, \lvert J \rvert \le n-1$ hence $n \not\in I\cup J$. However, $w, \overline{w}$ are identical on $w_1, \dots, w_{n-1}$ so $P_I P_J$ must also appear in $F_{n, \ell, w}$, a contradiction. So $F_{n,\ell,\overline{w}} = 0$.

\medskip

\noindent{\bf Claim 2.}
If $F_{n-1,\ell,w}=0$ and $w_{n-1}=n-1$, then
$F_{n, \ell, \overline{w}} = 0$ for $\overline{w}=(w_1,\ldots,w_{n-2},n,n-1)$. 

Suppose by contradiction that $F_{n,\ell,\overline{w}} \neq 0$ and so it contains $P_I P_J$ either as a monomial or as part of a relation, $P_I P_J - P_{I'} P_{J'}$ from ${\rm in}_{{\bf w}_\ell}(I_n)$. Now if $n \not\in I \cup J$ then by the same argument as before we have that $P_I P_J$ appears in $F_{n-1,\ell,w}$, a contradiction. So without loss of generality let us assume that $n \in I$ and $\lvert I \rvert, \lvert J \rvert \le n-1$. Since $P_I \neq 0$ in $F_{n,\ell,\overline{w}}$, therefore $I \le (w_1, \dots, w_{\lvert I \rvert})$. Since $n \in I$ we must have $n \in (w_1, \dots, w_{\lvert I \rvert})$. Since $w_{n-1} = n$ we deduce that $\lvert I \rvert = n-1$. Now if we also have that $n \in J$ then $P_{I \setminus n} P_{J \setminus n} \neq 0$ in $F_{n-1,\ell,w}$ and belongs to the non-trivial relation $P_{I \setminus n} P_{J \setminus n} - P_{I' \setminus n} P_{J' \setminus n}$ in ${\rm in}_{{\bf w}_\ell}(I_{n-1})$. Note that this is a true relation among the variables regardless of the block diagonal matching field $B_\ell$ since by assumption $n \ge 4$ and hence $\lvert I \rvert \ge 3$. So $F_{n-1,\ell,w} \neq 0$, a contradiction. So we deduce that $n \not\in J$. Since $n$ is contained in exactly one of the subsets $ I'$,$J'$, we may assume $n \in I'$.

Now consider $P_{I \setminus n} P_{J} - P_{I' \setminus n} P_{J'}$. This is a (possibly trivial) relation with $P_{I \setminus n} P_{J} \neq 0$ in $F_{n-1,\ell,w}$. Again note that this is a true relation among the variables regardless of $B_\ell$ since $n \ge 4$. However $F_{n-1,\ell,w} = 0$. Thus this relation must be trivial, otherwise $P_{I \setminus n} P_{J}$ would be contained in $F_{n-1,\ell,w}$. By assumption the relation $P_I P_J - P_{I'} P_{J'}$ is non-trivial so $I \neq I'$ and $J \neq J'$. Therefore we must have $I \setminus n = J'$ and $J = I' \setminus n$. We deduce that $\lvert I \setminus n \rvert = \lvert J \rvert = n-2$. Since $P_{I \setminus n} P_{J} \neq 0$ we have that $I\setminus n$ and $J \le (w_1, \dots, w_{n-2}) = \{1, \dots, n-2\}$. However, from this we deduce that $I\setminus n = J = \{1, \dots, n-2 \}$ and so $I = I'$, a contradiction. Therefore $F_{n,\ell,\overline{w}} = 0$.
\medskip

\noindent{\bf Claim 3.} If $F_{n,\ell,w}=0$, then either $w=(w_1,\ldots,w_{n-1},n)$ or $w=(w_1,\ldots,w_{n-2},n,n-1)$. 

First we show that $w_{n} \ge n-1$. So suppose by contradiction $w_{n} < n-1$. Then we have either $w = (\alpha, n, \beta, n-1, \gamma, w_{n})$ or $w = (\alpha, n-1, \beta, n, \gamma, w_{n})$ for some ordered subsets $\alpha, \beta, \gamma$ of $[n]$. Now suppose $\lvert \alpha \cup \beta \rvert \ge 1$ in which case consider the relation
$$P_{\alpha,\beta,n-1} P_{\alpha,w_{n},\beta, n} - P_{\alpha,w_{n},\beta} P_{\alpha,\beta,n-1, n}.$$
We must justify that this is indeed a relation for non-diagonal matching field cases. Since $\lvert \alpha \cup \beta \rvert \ge 1$ the above relation has the form $P_M P_{N,n} - P_N P_{M,n}$ where $\lvert M \rvert, \lvert N \rvert \ge 2$. It follows immediately from the definition of $B_\ell$ that $B_\ell(M) = B_\ell(M \cup \{n\})$ and $B_\ell(N) = B_\ell(N \cup \{ n \})$ for any $\ell$. Hence this is a true relation among the variables.

Observe that none of the variables in the above relation vanishes in $F_{n,\ell,w}$ for $w = (\alpha, n, \beta, n-1, \gamma, w_{n})$ and $w = (\alpha, n-1, \beta, n, \gamma, w_{n})$. So $F_{n,\ell,w} \neq 0$, a contradiction.

Next suppose $\alpha \cup \beta = \varnothing$. Then $w$ either has the form $(n-1, n, \gamma, w_n)$ or $(n, n-1, \gamma, w_n)$. Now we take cases on $B_\ell$, the block diagonal matching field. We have that either $\ell = 0$, $1 \le \ell \le n-2$ or $\ell = n-1$.

\textbf{Case 1.} Let $\ell = 0$, the diagonal matching field. Then it is easy to check that the relation $P_{n-1} P_{1, n} - P_{1} P_{n-1, n}$ does not vanish in $F_{n,\ell,w}$.

\smallskip

\textbf{Case 2.} Let $1 \le \ell \le n-2$. Then we have the relation $P_{n-1} P_{n, 1} - P_{n} P_{n-1, 1}$. Note that $P_{n-1} P_{n, 1}$ does not vanish in $F_{n,\ell,w}$.

\smallskip

\textbf{Case 3.} Let $\ell = n-1$. Then consider the relation $P_{1} P_{n, n-1} - P_{n} P_{1, n-1}$. Note that $P_{1} P_{n, n-1}$ does not vanish in $F_{n,\ell,w}$.

Therefore, we have shown that $F_{n,\ell,w} \neq 0$ which is a contradiction. So we conclude that $w_n \ge n-1$.

It remains to show that if $w_n = n-1$ then $w_{n-1} = n$. So suppose by contradiction that $w$ is of the form $w = (\alpha, n, \beta, n-1)$ for some ordered subsets $\alpha, \beta $ of $[n]$, where $\lvert \beta \rvert \ge 1$. Let $b \in \beta$ be an arbitrary element. Now we take two cases on the matching field $B_\ell$.

\textbf{Case 3a.} $B_\ell$ is the diagonal matching field or $\alpha \cup \beta \setminus b \neq \varnothing$.  Consider the following relation
$$P_{\alpha,\beta \setminus b, n-1} P_{\alpha,\beta,n} - P_{\alpha, \beta} P_{\alpha, \beta \setminus b, n-1, n}.$$
We must justify that this is indeed a relation for non-diagonal matching field cases. Since $\lvert \alpha \cup \beta \setminus b\rvert \ge 1$ the above relation has the form $P_M P_{N,n} - P_N P_{M,n}$ where $\lvert M \rvert, \lvert N \rvert \ge 2$. So, as above, this is a true relation among the variables. It is easy to check that $P_M P_{N,n}$ does not vanish in $F_{n,\ell,w}$ and so $F_{n,\ell,w} \neq 0$, a contradiction.

\textbf{Case 3b.} $B_\ell$ is not a diagonal matching field and $\alpha \cup \beta \setminus b = \varnothing$. Then $w = (3,1,2)$. We now refer to Example~\ref{example:n_3_4_calculation} where we observe that $F_{n,\ell,w}$ is non-zero for each $\ell$.
Hence we have a contradiction, so if $w_{n} = n-1$ then $w_{n-1} = n$. Therefore $w$ must be of the desired form.

\medskip

In the above series of claims we have shown that for each $\ell$, $F_{n,\ell,w}=0$ if and only if $F_{n-1, \ell, \underline{w}} = 0$ and either $w = (w_1, \dots, w_{n-2}, n, n-1)$ or $w = (w_1, \dots, w_{n-1}, n)$. We now proceed by induction on $n$. If $n = 1$ then it is clear that $F_{n, \ell, w} = 0$ where $w = (1)$. We observe that the set $Z_n$ satisfies the inductive relation:
\begin{align*}
    Z_{n} = \{w \in S_n : \ul{w} \in Z_{n-1}, \text{  and either } &w = (w_1, \dots, w_{n-1}, n) \\
    \text{ or }  &w = (w_1, \dots, w_{n-2}, n, n-1)\}.
\end{align*}
This is the same inductive relation shown in the claims which completes the proof.
\end{proof}

As an immediate corollary of the above theorem we have:

\begin{corollary}\label{prop:z_n}
$|Z_n|=|Z_{n-2}|+|Z_{n-1}|$ for all $n$.
\end{corollary}
\begin{proof}
Using the formulation of $Z_{n}$ in the proof of Theorem~B, we can verify that
$$Z_n = \{w \in Z_n: w_n = n \} \sqcup \{w \in Z_n: w_n = n-1, w_{n-1} = n \}.$$
But if $w_n = n$, then $w$ is determined by its first $n-1$ entries, and so the cardinality of the first set is $\lvert Z_{n-1} \rvert$.
And if $w_n = n-1$ and $w_{n-1} = n$, then $w$ is determined by its first $n-2$ entries. Hence,
the cardinality of the second set is $ \lvert Z_{n-2} \rvert$, as desired.
\end{proof}

\section{Binomial initial ideals}\label{sec:pf_thm_b}

In this section, we present the main ingredients required for the proof of Theorem~C. We will prove results that connect key properties of permutations $w$, matching fields $B_\ell$ and the ideal $F_{n, \ell, w}$. We begin by showing that $A_1 \subset T_{n,\ell}$ for all $n$ and $\ell$. In the following, we divide the results into three parts. Firstly the diagonal case with $\ell = n$, secondly the semi-diagonal case, i.e. $\ell = n-1$, and finally all remaining cases. Figure~\ref{flowchart:dependency} shows the dependencies among the results required for the proof of Theorem~C. The different colours in the diagram indicate the different sections in which the results can be found.

\begin{figure}
    \centering
    \includegraphics[scale = 0.8]{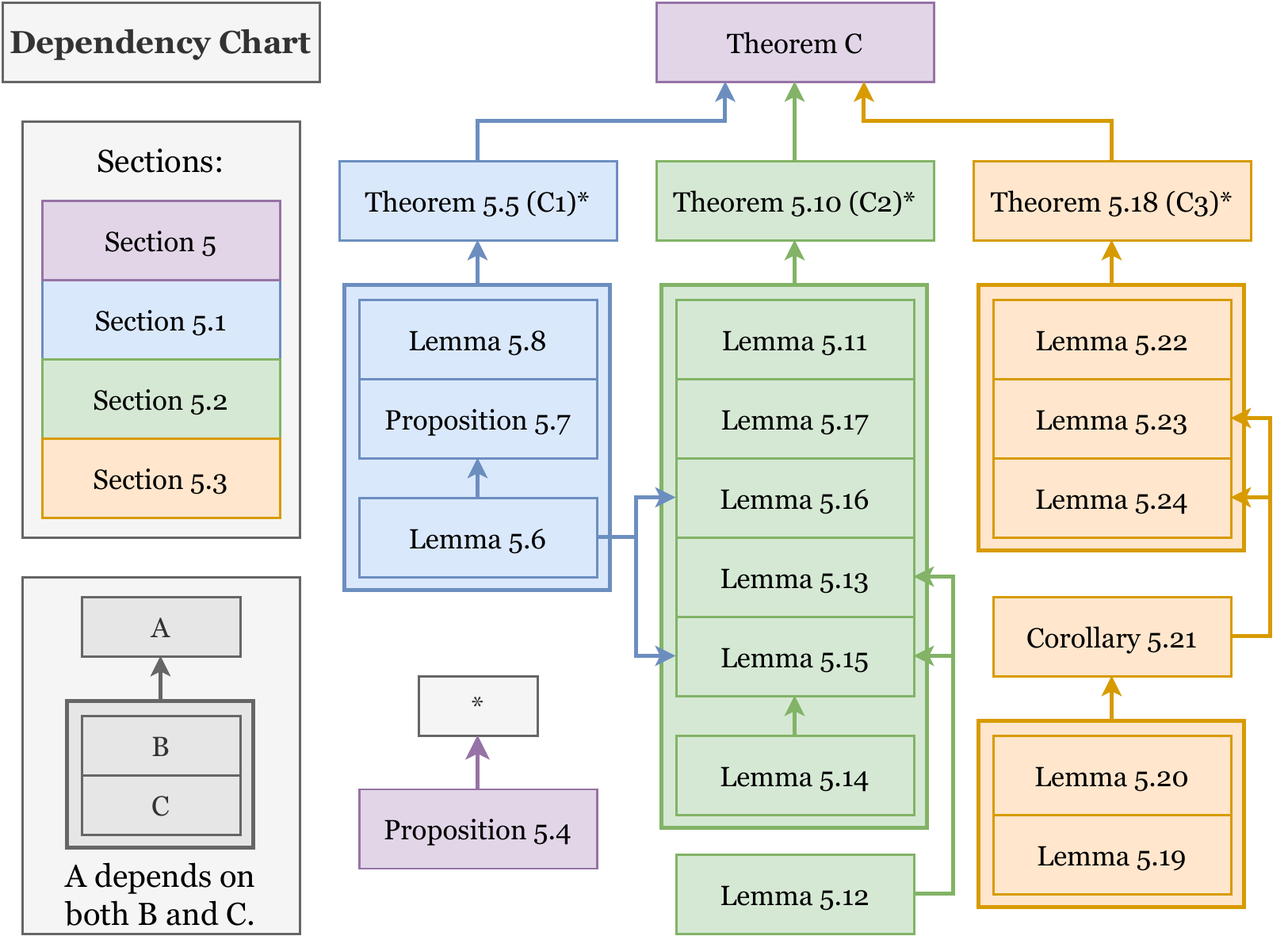}
    \caption{
    The dependency chart above shows the key steps in the proof of Theorem~C which classifies the permutations in $T_{n,\ell}$. We split Theorem~C into three cases: C1, C2 and C3 corresponding to $\ell = n$, $1 \le \ell \le n-2$ and $\ell = n-1$ respectively. Note that each of these cases requires Proposition~\ref{lem:thm_ToricFlag_conv_1} since $A_1 \subset T_{n,\ell}$ for each $\ell$.}
    \label{flowchart:dependency}
\end{figure}

Many results of this section are inductive in nature. In the next example, we explicitly calculate the ideals $F_{3,\ell,w}$ for each matching field $B_\ell$ and each permutation $w \in S_3$. Therefore, we will assume $n > 3$ in the subsequent sections.

\begin{example}\label{example:n=4_toric_principal}
Let $n = 4$. For each $w \in A_1 = \{1342, 1432\}$ and matching field $B_\ell$ we calculate the ideal $F_{4,\ell,w}$. In particular, we note that each such ideal is principal and toric, i.e. binomial and prime. 
\[
\begin{tabular}{|c|c|c|}
    \hline
    $\ell$             & $w$   & $F_{4,\ell,w}$ \\
    \hline
    \multirow{2}{*}{$0$}        & 1342  & $\langle P_{13} P_{124} - P_{12} P_{134} \rangle $ \\
                                & 1432  & $\langle P_{13} P_{124} - P_{12} P_{134} \rangle $ \\
    \hline
    \multirow{2}{*}{$1$}  & 1342  & $\langle P_{31} P_{214} - P_{21} P_{314} \rangle $ \\
                                & 1432  & $\langle P_{31} P_{214} - P_{21} P_{314} \rangle $ \\
    \hline
    \multirow{2}{*}{$2$}  & 1342  & $\langle P_{31} P_{124} - P_{12} P_{314} \rangle $ \\
                                & 1432  & $\langle P_{31} P_{124} - P_{12} P_{314} \rangle $ \\
    \hline
    \multirow{2}{*}{$3$}  & 1342  & $\langle P_{13} P_{124} - P_{12} P_{134} \rangle $ \\
                                & 1432  & $\langle P_{13} P_{124} - P_{12} P_{134} \rangle $ \\
    
    \hline
\end{tabular}
\]
\end{example}

\begin{definition}\label{def:Scw}
 Let $w = (w_1, \dots, w_n)$. Recall that $S_w = \{I: I \not \le w_{(I)} \} $. We denote its complement by $S_w^c = \{I : \varnothing \neq I \subsetneq [n] \} \setminus S_w = \{I : I \le \{w_1, \dots, w_{|I|} \} \}$ and for $1 \le t \le n$ we define its projection as $$S_{(w_1, \dots, w_t)}^c = \{I: \lvert I \rvert \le t, I \in S_w^c \}.$$
\end{definition}
\vspace{-2mm}

Note that $S_w^c$ is the collection of subsets $I \subset [n]$ for which $P_I$ do not vanish in $F_{n,\ell,w}$ for any $\ell$. From Definition~\ref{def:Scw} we obtain the following description of $S_w^c$ for specific cases.

\begin{corollary}\label{lem:Swc_1}
Let $w\in S_n$ with $w_{t}=n$ for $t\in\{1,\ldots,n\}$. Then 
\[
S_{w}^c = S_{(w_1, \dots, w_{t-1})}^c 
\cup \{I : I = \{i_1 <  \dots < i_\ell \},\ \ell \ge  t,\ I \setminus i_\ell \in S_{\ul{w}}^c \}.
\] 
Moreover, for $w\in S_n^>$ we have that 
\[
S_{w}^c = S_{(w_1, \dots, w_{t-1})}^c
\cup \{I : I = \{i_1 < \dots < i_\ell \},\ \ell \ge t,\ (i_1, \dots, i_{t-1}) \in S_{(w_1, \dots, w_{t-1})}^c\}.
\]
\end{corollary}

For each block diagonal matching field $B_\ell$ we have:

\begin{proposition}\label{lem:thm_ToricFlag_conv_1}
For each $w\in A_1$ and $0 \le \ell \le n$, $F_{n,\ell,w}$ is a principal toric ideal. In particular, $A_1\subset T_{n,\ell}$.
\end{proposition}
\begin{proof}
Assume that $w\in S_n$ has the form
    \[
    w = (w_1, \dots, w_{n-3}, n, n-1, n-2)\ {\rm or }\ 
    w = (w_1, \dots, w_{n-3}, n-1, n, n-2).
    \]
We prove that if $\ul{w}$ is in $Z_{n-1}$, then $F_{n,\ell,w}$ is toric and principal, i.e. generated by a single polynomial. For $n = 3$, we have that $A_1 = \{321\}$ and the result follows from the calculation in Example~\ref{example:n_3_4_calculation}. Similarly for $n=4$, the result follows from Example~\ref{example:n=4_toric_principal}. Now we assume that $n > 4$. 

We first show that $F_{n,\ell,w} \neq 0$.
Let $\alpha = \{ w_1, \dots, w_{n-3}\} = \{1, \dots, n-3 \}$. Note that $\lvert \alpha \rvert \ge 2$. Now consider the following relation in $F_{n,\ell,w}$:
$$P_{\alpha, n-2} P_{\alpha, n-1, n} - P_{\alpha, n-1} P_{\alpha, n-2, n}.$$
Notice that none of these variables vanish in $F_{n,\ell,w}$ for either $w$ above. We must check that this relation does not depend on the block diagonal matching field $B_\ell$. This follows from two basic properties of the matching field $B_\ell$. Firstly, the matching field permutes only entries of $\alpha$ and fixes all others. And secondly, if $\beta \subset [n]$ is disjoint from $\alpha$ then $B_\ell(\alpha) = B_\ell(\alpha \cup \beta)$.

\medskip

Now suppose that we have two variables $P_I$ and $P_J$ which do not vanish in $F_{n,\ell,w}$ and belong to a relation $P_I P_J - P_{I'} P_{J'} \in \inwb(I_n)$. We will show that $|I|, |J| > n-3$ by contradiction. So without loss of generality we assume that $|I| \le |J|$ and $|I| \le n-3$. Additionally, we may assume that $\lvert I \rvert = \lvert I' \rvert, \lvert J \rvert = \lvert J' \rvert$. We proceed by taking cases on $|J|$.

\textbf{Case 1.} Let $\lvert J \rvert \le n-3 $. Since neither $P_I$ nor $P_J$ vanish, we have that $I \le (w_1, \dots, w_{\lvert I \rvert })$ and $J \le (w_1, \dots, w_{\lvert J \rvert })$. Since $w_i < n-2$ for all $1 \le i \le n-3$ we deduce that $P_I P_J - P_{I'} P_{J'}$ appears in $F_{n-1, \ell, \underline{w}}$. However $\underline{w} \in Z_{n-1}$ so $F_{n-1, \ell, \underline{w}} = 0$, a contradiction.

\smallskip

\textbf{Case 2.} Let $\lvert J \rvert > n-3 $. We have that $(w_1, \dots, w_{n-3}) = \{1, \dots, n-3 \}$. Since $I \le (w_1, \dots, w_{\lvert I \rvert})$ clearly we must have $I \subseteq \{1, \dots, n-3 \}$. Similarly, $J \le (w_1, \dots, w_{\lvert J \rvert})$ and so we deduce that $J = \{1, \dots, n-3, j_{n-2}, \dots, j_{\lvert J \rvert} \}$. Hence $I \subset J$ and it is easy to check that $P_I P_J - P_{I'} P_{J'}$ is a trivial relation, a contradiction.

\smallskip

So $\lvert I \rvert, \lvert J \rvert > n-3$. Note that $n \ge 5$. Since $(w_1, \dots, w_{n-3}) = (1, \dots, n-3)$ we have $I = \{1 < \dots < n-3 < i_{n-2} < \dots < i_{\lvert I \rvert} \}$ and similarly $J = \{1 < \dots < n-3 < j_{n-2} < \dots < j_{\lvert J \rvert} \}$. The relation $P_I P_J - P_{I'} P_{J'}$ can be seen to arise from a relation in $\inw{D}(I_3)$ under the diagonal matching field. This relation is obtained by removing $\{1, \dots, n-3 \}$ from $I, J, I', J'$, and then subtracting $n-3$ from each remaining entry. Note that this process does not depend on the matching field $B_\ell$ because $n \ge 5$ and so $B_\ell$ permutes only the entries in $\{1, \dots, n-3 \}$ and fixes all others. However, $\inw{D}(I_3)$ is principal and generated by $P_1 P_{23} - P_2 P_{13}$. So the relation $P_I P_J - P_{I'} P_{J'}$ must be $P_{\alpha,n-2} P_{\alpha, n-1, n} - P_{\alpha, n-1} P_{\alpha, n-2, n}$ where $\alpha = \{1, \dots, n-3 \}$. It is clear that this relation is contained in $F_{n,\ell,w}$ for each $w$ above. Since $P_I P_J$ was arbitrary, it follows that $F_{n,\ell,w} = \langle P_{\alpha,n-2} P_{\alpha, n-1, n} - P_{\alpha, n-1} P_{\alpha, n-2, n} \rangle$ is a principal ideal.
\end{proof}

\subsection{Diagonal matching fields}\label{sec:proofs_diag} 
Recall that the set $T_{n,n}$ is the collection of all permutations $w \in S_n$ such that $F_{n,n,w}$ is a non-zero binomial ideal. The sets of permutations $A_1$ and $A_2$ are defined inductively from $Z_{n-1}$ and $T_{n-1,n-1}$ respectively by `inserting' $n$ into the permutation in allowed places.

Here, we state our main theorem for the diagonal matching fields. 
Note that the classification of $F_{n,\ell,w}$ is simpler for the diagonal case than for the other matching fields
and serves as a good template for the proofs in the following sections. In particular, we will see analogues for Lemmas~\ref{lem:WDecrease}, \ref{lem:nontoric_extn} and Proposition~\ref{lem:thm_ToricFlag_conv_2} for other matching fields in later sections.

\begin{theorem}\label{Intro:ToricFlag}  $T_{n,n} = A_1 \cup A_2.$
\end{theorem}

\begin{proof}
By Propositions~\ref{lem:thm_ToricFlag_conv_1} and \ref{lem:thm_ToricFlag_conv_2} we have that $A_1\cup A_2\subseteq T_{D,n}$. To prove the other direction suppose that $F_{n,n,w}$ is binomial and write the permutation $w = (w_1, \dots, w_t, n, w_{t+1}, \dots, w_{n-1})$ for some $t \in \{0, 1, \dots, n-1 \}$. Now by Lemma~\ref{lem:nontoric_extn} $F_{n-1,n-1,\ul{w}}$ is either zero or binomial.

Firstly, suppose that $F_{n-1,n-1,\ul{w}} = 0$. Theorem~B implies that $\ul{w}$ is of form
$$\ul{w} = (w_1, \dots, w_{n-2}, n-1)\quad \text{or}\quad
\ul{w} = (w_1, \dots, w_{n-3}, n-1, n-2).$$
Since $F_{n,n,w}$ is binomial, by Lemma~\ref{lem:WDecrease} we have $n > w_{t+1} > \dots > w_{n-1}$. So if $\ul{w} = (w_1, \dots, w_{n-2}, n-1)$ then we have that 
\[
w = (w_1, \dots, w_{n-2}, n-1, n)\quad\text{or}\quad w = (w_1, \dots, w_{n-2}, n, n-1).
\]
However, in both cases we have that $F_{n,n,w} = 0$ by Theorem~B, a contradiction. So we must have that $\ul{w} = (w_1, \dots, w_{n-3}, n-1, n-2)$. Now by Lemma~\ref{lem:WDecrease} we have that $w$ is one of the following permutations: 
\begin{itemize}
    \item $w = (w_1, \dots, w_{n-3}, n-1, n-2, n)$,
    \item $w = (w_1, \dots, w_{n-3}, n-1,n , n-2)$,
    \item $w = (w_1, \dots, w_{n-3}, n, n-1, n-2)$.
\end{itemize}
However, if $w = (w_1, \dots, w_{n-3}, n-1, n-2, n)$ then by Theorem~B we have that $F_{n,n,w} = 0 $, a contradiction. The remaining cases are of the desired form.

Secondly, suppose that $F_{n-1,n-1,\ul{w}}$ is binomial and $w_s = n-1$. By Lemma~\ref{lem:WDecrease} we have that $w_s > w_{s+1} > \dots > w_{n-1}$ and $n >{w_{t+1}} > \dots > w_{n-1}$. Thus we must have that $t \ge s-1$, otherwise $n > w_{s-1}$ and $w_{s-1} < w_{s}$ which contradicts Lemma~\ref{lem:WDecrease}.
\end{proof}

\begin{lemma}\label{lem:WDecrease}
If $F_{n,n,w}$ is binomial, then $w \in S_n^>$. 
\end{lemma}

\begin{proof}
Let $w = (w_1,\dots, w_{n})$ with $w_t = n$. Suppose by contradiction that there exists $k > t$ such that $w_k < w_{k+1}$. Without loss of generality, take $k$ to be the minimum such index. We will show that $F_{n,n,w}$ contains a monomial. Let
$$I = \{w_1, \dots, w_{k-1}, w_{k+1}\}\quad\text{and}\quad I' = \{w_1, \dots, w_k \}.$$
By this construction $P_I$ vanishes and $P_{I'}$ does not vanish in $F_{n,n,w}$ because $w_{k+1} > w_k$ and so $I > \{ w_1, \dots, w_k \} = I'$. 
Now we write $I \cup I'$ as an ordered set as $(\alpha, w_k, \beta, w_{k+1}, \gamma, n)$ for some ordered subsets $\alpha, \beta, \gamma$ of $[n]$, so $I = (\alpha, \beta, w_{k+1}, \gamma, n)$ and $I' = (\alpha, w_k, \beta, \gamma, n)$. Now define $J = (\alpha, w_k, \beta, \gamma)$ and $J' = (\alpha, \beta, w_{k+1}, \gamma)$.
By construction, we have $P_I P_J - P_{I'} P_{J'}$ is a relation in $\inwb(I_n)$ where $P_{I}$ vanishes and $P_{I'}$ does not vanish in $F_{n,n,w}$. However $P_{J'}$ does not vanish because
\[\{w_1, \dots, w_{k-1}\} = (\alpha, \beta, \gamma,n) \ge (\alpha,\beta,w_{k+1},\gamma) = J'.\]
And so, we have shown that the monomial $P_{I'}P_{J'}$ appears in the binomial ideal $F_{n,n,w}$, which is a contradiction.
\end{proof}

\begin{proposition}\label{lem:thm_ToricFlag_conv_2}
Suppose that $F_{n-1,n-1,\ul{w}}$ is binomial, where $\ul{w} = (w_1, \dots, w_{n-1})$ and $w_s = n-1$. Then $F_{n,n,w}$ is  binomial for $w = (w_1, \dots, w_{t}, n, w_{t+1}, \dots, w_{n-1})$, where $t \ge s-1$.
\end{proposition}

\begin{proof}
Suppose that $P_I P_J - P_{I'} P_{J'} \in \inwb(I_n)$ and $P_I P_J \neq 0$ in $F_{n,n,w}$. We show that $P_{I'} P_{J'} \neq 0$ by taking cases on $\lvert I \rvert$ and $\lvert J \rvert$.  {Without loss of generality we assume $|I| \le |J|$ so we have that either $|I|,|J| \le t$ or $|I| \le t$ and $|J| > t$ or $|I|, |J| > t$.}

\smallskip

\textbf{Case 1.} Let $\lvert I \rvert, \lvert J \rvert \le t $. Since $\ul{w}, w$ are identical from $w_1$ to $w_t$, we deduce that $P_I P_J \neq 0$ in $F_{n-1,n-1,\ul{w}}$. Since $F_{n-1,n-1,\ul{w}}$ is binomial we have that $P_{I'} P_{J'} \neq 0$ in $F_{n-1,n-1,w}$ and so it is non-zero in $F_{n,n,w}$.

\smallskip

\textbf{Case 2.} Let $\lvert I \rvert \le t, \lvert J \rvert > t $. Since $P_I P_J \neq 0$ in $F_{n,n,w}$, by Corollary~\ref{lem:Swc_1} we have $P_I \neq 0$ and $P_{J \setminus j_{\lvert J \rvert} } \neq 0$ in $F_{n-1,n-1,\ul{w}}$ where $J = \{j_1 < \dots < j_{\lvert J \rvert} \}$. Now, $P_I P_{J \setminus j_{\lvert J \rvert} } - P_{I'} P_{J' \setminus j_{\lvert J \rvert} }$ is a valid (possibly trivial) relation among the variables in $F_{n-1,n-1,\ul{w}}$. Since this ideal is   binomial we have that $P_{I'} \neq 0$ and $P_{J' \setminus j_{\lvert J \rvert} } \neq 0$. By Corollary~\ref{lem:Swc_1}, $P_{I'} \neq 0$ and $P_{J'} \neq 0$ in $F_{n,n,w}$. 

\smallskip

\textbf{Case 3.} Let $\lvert I \rvert, \lvert J \rvert > t$. Write $I = \{i_1 < \dots < i_p \}$ and $J = \{j_1 < \dots < j_q \}$. By Lemma~\ref{lem:WDecrease} we have that $n > w_{t+1} > \dots > w_{n-1}$ and so we may apply Corollary~\ref{lem:Swc_1} as follows.
$P_I \neq 0$ in $F_{n,n,w}$ if and only if $(i_1, \dots, i_{t}) \le (w_1, \dots, w_{t})$. Similarly, $P_J \neq 0$ in $F_{n,n,w}$ if and only if $(j_1, \dots, j_{t}) \le (w_1, \dots, w_{t})$. Next let us write $I' = \{i_1' < \dots < i_p' \}$ and $J' = \{j_1' < \dots < j_q' \}$. Now suppose without loss of generality that $p \le q$. Since we are working with the diagonal matching field, for each $1 \le e \le p$ we have that $i_e', j_e' \in \{i_e, j_e \}$. Hence $(i_1', \dots, i_t') \le (w_1, \dots, w_{t})$ and $(j_1', \dots, j_{t}') \le (w_1, \dots, w_{t})$. By Corollary~\ref{lem:Swc_1}, we have $P_{I'} \neq 0$ and $P_{J'} \neq 0$ in $F_{n,n,w}$. Hence $F_{n,n,w}$ is   binomial.
\end{proof}

\begin{lemma}\label{lem:nontoric_extn}
If $F_{n-1,n-1,\ul{w}}$ is  non-binomial, then $F_{n,n,w}$ is  non-binomial.
\end{lemma}

\begin{proof}
Suppose $F_{n-1,n-1,\ul{w}}$ is  non-binomial. Then there exists a monomial $P_I P_J \in F_{n-1,n-1,\ul{w}}$. Suppose this monomial arises from the relation $P_I P_J - P_{I'} P_{J'} \in \inwb(I_{n-1})$ such that $P_{I'} P_{J'} = 0$ in $F_{n-1,n-1,\ul{w}}$. Without loss of generality assume that $\lvert I \rvert = \lvert I' \rvert \le \lvert J \rvert = \lvert J' \rvert$. Write $\ul{w} = (w_1, \dots, w_{n-1})$ and $w = (w_1, \dots, w_{t}, n, w_{t+1}, \dots, w_{n-1})$ for some $t \in \{0,1,\dots, n-1 \}$. We take cases based on $t, \lvert I \rvert$ and $\lvert J \rvert$. In particular, we must either have $|J| \le t$ or $|I| \le t < |J|$ or $t < |I|$.

\smallskip

\textbf{Case 1.} Let $\lvert J \rvert \le t$. In this case we have that $P_I P_J$ is a monomial in $F_{n,n,w}$ because $(w_1, \dots, w_t)$ determines whether the variables in the above relation vanish in $F_{n,n,w}$ and $\ul{w}, w$ are identical on $(w_1, \dots, w_t)$.

\smallskip

\textbf{Case 2.} Let $\lvert I \rvert \le t < \lvert J \rvert$. Consider the relation $P_I P_{J \cup \{ n \} } - P_{I'} P_{J' \cup \{ n \}}$ in $\inwb(I_{n})$. By Corollary~\ref{lem:Swc_1} we have that $P_I P_J \neq 0$ and $P_{I'} P_{J'} = 0$ in $F_{n-1,n-1, \ul{w}}$ if and only if $P_I P_{J \cup \{ n \}} \neq 0$ and $P_{I'} P_{J' \cup \{ n \}} = 0$ in $F_{n,n,w}$.

\smallskip

\textbf{Case 3.} Let $t < \lvert I \rvert$. Consider the relation $P_{I \cup \{ n \}} P_{J \cup \{ n \}} - P_{I' \cup \{ n \}} P_{J' \cup \{ n \}} \in \inwb(I_n)$. Applying Corollary~\ref{lem:Swc_1} we have that $P_I P_J \neq 0$ and $P_{I'} P_{J'} = 0$ in $F_{n-1,n-1, \ul{w}}$ if and only if $P_{I \cup \{ n \}} P_{J \cup \{ n \}} \neq 0$ and $P_{I' \cup \{ n \}} P_{J' \cup \{ n \}} = 0$ in $F_{n,n,w}$.

In each case we have shown that $F_{n,n,w}$ is  non-binomial, as desired.
\end{proof}

\begin{remark}
Note that the converse to Lemma~\ref{lem:nontoric_extn} is false. For example if $w = (4,2,3,1)$ then $F_{4,4,w}$ is  non-binomial, however $F_{3,3,\ul{w}}$ is   binomial for $\ul{w} = (2,3,1)$. 
\end{remark}

\subsection{Semi-diagonal matching fields}\label{sec:proofs_semi_diag}

Below, we state and prove the main result for $\ell=n-1$, which decomposes the collection of permutations $T_{n,n-1}$ into three parts: $A_1$, $\tilde{A}_1$ and $\tilde{A}_2$, see Definition~\ref{def:notation}. The proof is split up into five steps. Each step is written with the claim at the beginning, followed by the proof of that claim. Steps a, b and c are very similar to the diagonal case, since we have seen that $A_1 \subset T_{n,n}$. Steps d and e are particular to the semi-diagonal case and show how the subsets $\tilde{A}_1$ and $\tilde{A}_2$ arise in the decomposition of $T_{n,n-1}$.

\begin{theorem}\label{Intro:Semi}
$T_{n,n-1} = A_1 \cup \tilde{A}_1 \cup  \tilde{A}_2$.
\end{theorem}

\begin{proof}
We will use Lemmas~\ref{lem:WDec2}, \ref{lem:toric_thm2_case_reduce}, \ref{lem:toric_thm2_a2} and \ref{lem:toric_thm2_a3}.
We will  break down the proof into the following steps.

\smallskip

\noindent{\bf Step a.} $A_1\cup \tilde{A}_1\cup \tilde{A}_2 \subset T_{n,n-1}$ and so $RHS \subseteq LHS$.

First 
by Proposition~\ref{lem:thm_ToricFlag_conv_1}, $F_{n,\ell,w}$ is   binomial for each $w \in A_1$, so $A_1 \subset T_{n, n-1}$. Next $\tilde{A}_1 \subset T_{n, n-1}$ and $\tilde{A}_2 \subset T_{n, n-1}$ by Lemma~\ref{lem:toric_thm2_a2} and Lemma~\ref{lem:toric_thm2_a3}, respectively. So we have shown $A_1 \cup \tilde{A}_1 \cup \tilde{A}_2 \subseteq T_{n, n-1}$.

\medskip

\noindent{\bf Step  b.} For any $w \in T_{n,n-1}$, $\ul{w} \in Z_{n-1} \cup T_{ n-1,n-1}$. 

Now take $w \in T_{n, n-1}$. By Lemma~\ref{lem:toric_thm2_non_toric_extn}, $\ul{w} \in T_{n-1, n-1} \cup Z_{n-1}$. By Lemma~\ref{lem:WDec2} we have that $w$ has the descending property. We denote $w_t = n$ and $w_s = n-1$.
\medskip

\noindent{\bf Step  c.} If $\ul{w} \in Z_{n-1} $ then $w \in A_1$.

First suppose $\ul{w} \in Z_{n-1}$. By Theorem~B, either $\ul{w} = (w_1, \dots, w_{n-2}, n-1)$ or $\ul{w} = (w_1, \dots, w_{n-3}, n-1, n-2)$. If $\ul{w}= (w_1, \dots, w_{n-2}, n-1)$ then $w = (w_1, \dots, w_{n-2}, n, n-1)$ or $w = (w_1, \dots, w_{n-2}, n-1, n)$ since $w$ has the descending property. However $F_{n,n-1,w} = 0$ by Theorem~B, a contradiction. So $\ul{w} = (w_1, \dots, w_{n-3}, n-1, n-2)$. If $w = (w_1, \dots, w_{n-3}, n-1, n-2, n)$ then $F_{n,n-1,w} = 0$, so $w = (w_1, \dots, w_{n-3}, n-1, n, n-2)$ or $w = (w_1, \dots, w_{n-3}, n, n-1, n-2)$. Therefore $w \in A_1$.
\medskip

\noindent{\bf Step  d.} If $\ul{w} \in T_{n-1, n-1} \cap T_{n-1,n-2}$ then $w \in \tilde{A}_1$.

Next suppose $\ul{w} \in T_{n-1, n-1} \cap T_{n-1,n-2}$. Since $w$ has the descending property, we have $t \ge s-1$. By Lemma~\ref{lem:a2_exception} we have $w \neq (n-1, n, n-2, \dots, 1)$. Therefore $w \in \tilde{A}_1$.

\medskip

\noindent{\bf Step  e.} If $\ul{w} \in T_{n-1, n-1} \cap N_{n-1,n-2}$ then $w \in \tilde{A}_2$.

Finally suppose $\ul{w} \in T_{n-1, n-1} \backslash T_{n-1,n-2} = T_{n-1, n-1} \cap N_{n-1,n-2}$. Since $w$ has the descending property, $t \ge s-1$. We show that $t \ge s+1$ by contradiction. Note that we cannot have $t = s$. Suppose $t = s-1$. Then by Lemma~\ref{lem:diag_extn_except} we have $w \in N_{n,\ell} $, a contradiction. Therefore $w \in \tilde{A}_2$.
\end{proof}
\begin{lemma} \label{lem:WDec2} 
If $w \in T_{n, n-1}$ then $w \in S_n^>$.
\end{lemma}

\begin{proof}
Let $w = (w_1, \dots, w_n)$ with $w_t = n$. Suppose by contradiction that there exists $k > t$ such that $w_k < w_{k+1}$. Let $I = \{ w_1, \dots, w_{k}\}$ and $I' = \{w_1, \dots, w_{k-1}, w_{k+1} \}$. Note that $P_I$ does not vanish in $F_{n,n-1,w}$ but $P_{I'}$ does vanish. Let us write in ascending order $I \cup I' = \{\alpha, w_k, \beta, w_{k+1}, \gamma, n\}$ for some subsets $\alpha, \beta, \gamma$ of $[n]$. Note that $I = \{ \alpha, w_k, \beta, \gamma, n \}$ and $I' = \{\alpha, \beta, w_{k+1}, \gamma, n \}$. Now we take cases on $\lvert \alpha \cup \beta \cup \gamma \rvert$.

\textbf{Case 1.} Let $\lvert \alpha \cup \beta \cup \gamma \rvert \ge 1$. Let $J = \{\alpha, \beta, w_{k+1}, \gamma \}$ and $J' = \{\alpha, w_{k}, \beta, \gamma \}$. Then it is easy to check that $P_I P_J - P_{I'} P_{J'}$ 
is a valid relation in in$_{\wb_{n-1}}(I_n)$. This is because for each $L \in \{I, J, I', J'\}$, we have $B_{n-1}(L) = id$. However, $P_I P_J$ does not vanish in $F_{n,n-1,w}$, so $F_{n,n-1,w}$ contains the monomial $P_I P_J$ since $P_{I'}$ vanishes in $F_{n,n-1,w}$. Therefore $F_{n,n-1,w}$ is  non-binomial, a contradiction. 

\textbf{Case 2.} Let $\lvert \alpha \cup \beta \cup \gamma \rvert = 0$. So $w = (n, w_2, w_3, \dots, w_n)$ with $k = 2$. Consider the relation $P_{n} P_{w_2 w_3} - P_{w_2} P_{n w_3}.$ in in$_{\wb_{n-1}}(I_n)$.
This relation holds because $B_{n-1}(\{w_2,w_3\}) = id$ and  $B_{n-1}(\{w_3,n\}) = (12)$ is a transposition. However $P_n$ vanishes in $F_{n,n-1,w}$, so $F_{n,n-1,w}$ contains the monomial $P_{w_2} P_{n w_3}$ and so is  non-binomial, a contradiction.
\end{proof}

\begin{lemma}\label{lem:toric_thm2_case_reduce}
Let $w \in S_n$ and $B_{n-1} = (1 \dots n-1 | n)$ be a block diagonal matching field. Suppose that $P_{I}P_{J}$ is a monomial in $F_{n,n-1,w}$ arising from the relation $P_I P_J - P_{I'} P_{J'}$ in $\inw{n-1}(I_n)$. If $B_{n-1}(I) = B_{n-1}(J) = id$ then $B_{n-1}(I') = B_{n-1}(J') = id$.
\end{lemma}

\begin{proof}
We show the result by contradiction. Suppose without loss of generality that $B_{n-1}(I') \neq id$. So $I' = \{i, n \}$ for some $1 \le i \le n-1$. We have either $n \in I$ or $n \in J$. Without loss of generality suppose $n \in J$. After ordering $I'$ according to the matching field $B_{n-1}$ we see that $n$ is the first element. So $n$ is the first element of $J$. However $B_{n-1}(J) = id$ so we deduce that $J = \{ n \}$. Since $\lvert J \rvert \neq \lvert I' \rvert$, we have $\lvert I \rvert = \lvert I' \rvert$ and $\lvert J \rvert = \lvert J' \rvert$. Write $I = \{a, i \}$ for some $a < i$. Then the relation is given by 
\[P_{a,i} P_n - P_{n,i} P_a.\]
We have that $P_n$ does not vanish in $F_{n,n-1,w}$ so $w = (n, w_2, \dots w_n)$. Since $P_{a,i}$ does not vanish we have $i \le w_2 $. On the other hand $P_a$ does not vanish but $P_{n,i} P_a$ vanishes so $P_{n,i}$ must vanish in $F_{n,n-1,w}$. Therefore $\{n,i \} \not \le \{n, w_2 \}$ so $i > w_2$, a contradiction.
\end{proof}

\begin{lemma} \label{lem:diag_extn_except}
Let $\ul{w} \in T_{n-1, n-1} \cap N_{n-1,n-2}$ with $w_s = n-1$ and $w_t = n$. If $t = s-1$ then $w \in N_{n,n-1}$.
\end{lemma}

\begin{proof}
Let $P_I P_J$ be a monomial appearing in $F_{n-1,n-2,\ul{w}}$ which arises from the relation $P_{I} P_{J} - P_{I'} P_{J'}$ in $\inw{n-2}(I_{n-1})$. 

Note that by definition of $B_{n-2} = (1, \dots, n-2\ | \ n-1)$, the only subsets $L \subseteq [n-1]$ for which $B_{n-2}(L) \neq id$ are those with $\lvert L \rvert = 2$ and $n-1 \in L$. If $B_{n-2}(I) = B_{n-2}(J) = B_{n-2}(I') = B_{n-2}(J') = id$ then $P_I P_J - P_{I'} P_{J'}$ is a relation in $\inw{D}(I_{n-1})$. This relation gives rise to a monomial in $F_{D,\ul{w}, n-1}$ but by assumption $w \in T_{n-1, n-1}$, a contradiction. So by Lemma~\ref{lem:toric_thm2_case_reduce} we may assume without loss of generality that $B_{n-2}(I) \neq id$. We write $I = \{i, n-1\}$ for some $1 \le i \le n-2$. We take cases on $\lvert J \rvert$.

\smallskip

\textbf{Case 1.} Let $\lvert J \rvert = 1$. Let us write $J = \{ j \}$. The relation is given by
$$P_{n-1, i} P_{j} - P_{j,i} P_{n-1}.$$
It follows that $j < i$. Now consider $\ul{w} = (w_1, w_2, \dots, w_{n-1})$. Since $P_I P_J$ does not vanish in $F_{n-1,n-2,\ul{w}}$ we have $j \le w_1$ and $\{i, n-1 \} \le \{ w_1, w_2 \}$. Therefore either $w_1 = n-1$ or $w_2 = n-1$. Since $j \le n-1$ we have $\{j,i \} \le \{w_1, w_2 \}$. By assumption $P_{j,i} P_{n-1}$ vanishes so $P_{n-1} $ vanishes in $F_{n-1,n-2,\ul{w}}$. 
We deduce that $w_1 \neq n-1$ and so $w_2 = n-1$. So by our assumption $w = (w_1, n, n-1, \dots )$. Now consider
$$P_{n, i} P_j - P_{j,i} P_n.$$
This is a relation in $\inw{n-1}(I_n)$. Note that $P_{n, i} P_j$ does not vanish but $P_{j,i} P_n$ does vanish in $F_{n,n-1,w}$. So we have shown that $w \in N_{n,n-1} $.

\smallskip

\textbf{Case 2.} Let $\lvert J \rvert \ge 2$. First we show that $B_{n-2}(J) = id$. Suppose by contradiction that $B_{n-2}(J) \neq id$. 
By definition of $B_{n-2} = (1, \dots, n-2 \ | \ n-1)$, we have $B_{n-2} (J)\neq id$ implies that $\lvert J \cap \{1,\ldots,n-2\} \rvert = 1$,
and so $\lvert J \rvert = 2$ 
and $n-1 \in J$. Hence, $J = \{j, n-1\}$ for some $1 \le j \le n-2 $. The relation is given by  $P_{n-1, i} P_{n-1, j} - P_{n-1, j} P_{n-1, i}$ which is trivial, a contradiction. So $B_{n-2}(J) = id$. Let us write $J = \{j_1 < j_2 < \dots < j_{\lvert J \rvert} \}$. The relation is given by
$$P_{n-1, i} P_{j_1, j_2, j_3, \dots, j_{\lvert J \rvert}} - 
P_{n-1, j_2} P_{j_1, i  , j_3, \dots, j_{\lvert J \rvert}}.$$

Next we show that $i < j_2$ by contradiction. Note that $i \neq j_2$ otherwise the above relation is trivial. Suppose that $i > j_2$. Since $P_{n-1, i}$ does not vanish, $P_{n-1, j_2}$ does not vanish in $F_{n-1,n-2,\ul{w}}$. So $P_{j_1, i, j_3, \dots, j_{\lvert J \rvert}}$ vanishes in $F_{n-1,n-2,\ul{w}}$. By Lemma~\ref{lem:WDecrease}, $\ul{w}$ has the descending property since $\ul{w} \in T_{n-1,n-1}$.
So if $w_1 = n-1$ then $\ul{w} = (n-1, n-2, \dots, 2,1)$ and $F_{\ul{B'},\ul{w},n-1}$ is   binomial, a contradiction. Since $P_{n-1, i}$ does not vanish in $F_{\ul{B'},\ul{w},n-1}$ and $w_1 \neq n-1$, it follows that $w_2 = n-1$.
Now by applying Corollary~\ref{lem:Swc_1} to $\underline{w}$ we have that $P_{j_1, i, j_3, \dots, j_{\lvert J \rvert}}$ vanishes in $F_{n-1,n-2,\underline{w}}$ if and only if $\{j_1, i \} \not \le \{w_1, w_2 \}$. On the other hand, $\{j_1, i \} \not \le \{w_1, w_2 \}$ implies that $I = \{n-1, i \} \not \le \{ w_1, w_2\}$, a contradiction.

So we have $i < j_2$. We deduce that $P_{j_1, i  , j_3, \dots, j_{\lvert J \rvert}}$ does not vanish in $F_{n-1,n-2,\ul{w}}$ so $P_{n-1, j_2}$ does vanish and $j_2 > w_1$. Since $P_{n-1, i}$ does not vanish, $w_2 = n-1$. Now $w = (w_1, n, n-1, \dots)$ has the descending property because $\ul{w}$ does. Consider the relation
$$P_{n, i} P_{j_1, j_2, j_3, \dots, j_{\lvert J \rvert}} - 
P_{n, j_2} P_{j_1, i  , j_3, \dots, j_{\lvert J \rvert}}.$$
This is a valid relation in $\inw{n-1}(I_n)$. By Corollary~\ref{lem:Swc_1} we see that $P_{n, i} P_{j_1, j_2, j_3, \dots, j_{\lvert J \rvert}}$ does not vanish in $F_{n,n-1,w}$ but $P_{n,j_2}$ does vanish in $F_{n,n-1,w}$. So $F_{n,n-1,w}$ is  non-binomial and $w \in N_{n,n-1} $.
\end{proof}

\begin{lemma}\label{lem:a2_exception}
Fix $\ell \in \{1, \dots, n-1 \}$. Then $F_{n,\ell,w}$ is  non-binomial for $w = (n-1, n, n-2, \dots, 1)$.
\end{lemma}

\begin{proof}
If $\ell \neq n-1$, then consider the following relation in $\inwb(I_n)$:
$$P_{n-1} P_{n,1} - P_{n} P_{n-1, 1}.$$
The monomial $P_{n-1} P_{n,1}$ does not vanish in $F_{n,\ell,w}$ whereas $P_n$ vanishes in $F_{n,\ell,w}$. So $F_{n,\ell,w}$ contains the monomial $P_{n-1}P_{n,1}$, hence $F_{n,\ell,w}$ is  non-binomial.

If $\ell = n-1$, then we consider the relation
$$P_{1} P_{n, n-1} - P_{n}P_{1,n-1}$$
in $\inwb(I_n)$. The term $P_{1} P_{n, n-1}$ does not vanish in $F_{n,\ell,w}$ whereas $P_{n}$ does vanish in $F_{n,\ell,w}$. So $F_{n,\ell,w}$ contains the monomial $P_{1}P_{n,n-1}$ and hence is  non-binomial.
\end{proof}

\begin{lemma}\label{lem:toric_thm2_a2}
We have that $\tilde{A}_1 \subset T_{n, n-1}$.
\end{lemma}

\begin{proof}
Let $w \in \tilde{A}_1$. Note that $\tilde{A}_1 \subseteq A_2 = \ul{T}_{n,n-1} \cap \{w \in S_n : \ul{w} \in S_{n-1}^> \text{ and if } w_s = n-1, w_t = n \text{ then } t \ge s - 1 \}$, so by Theorem~C part C1, we have that $w \in T_{n,n}$. Hence, by Lemma~\ref{lem:WDecrease}, we conclude that $w \in S^>_n$.
By contradiction suppose that $F_{n,n-1,w}$ is  non-binomial. So there exists a monomial $P_I P_J \in F_{n,n-1,w}$ arising from a relation $P_I P_J - P_{I'} P_{J'}$ in $\inwb(I_n)$. We assume without loss of generality that $\lvert I \rvert = \lvert I' \rvert$ and $\lvert J \rvert = \lvert J' \rvert$. 
 If $B_{n-1}(I) = B_{n-1}(J) = id$ then by Lemma~\ref{lem:toric_thm2_case_reduce} we have $B_{n-1}(I') = B_{n-1}(J') = id$ so $P_I P_J$ would be a monomial in $F_{n,n,w}$, a contradiction. So without loss of generality suppose that $B_{n-1}(I) \neq id$. Write $I = \{i, n\}$ for some $1 \le i \le n-1$. If $B_{n-1}(J) \neq id$ then the relation would be trivial since we would have $J = \{j, n\}$ for some $j$.  Therefore $B_{n-1}(J) = id$.

We have that $P_{n,i}$ does not vanish in $F_{n,n-1,w}$ so $\{n,i \} \le \{w_1, w_2 \}$ and so either $w_1 = n$ or $w_2 = n$. If $w_1 = n$ then $w = (n, n-1, \dots, 1)$ since $w \in S_n^>$. In this case, for all $L \subseteq [n]$, $L \in S_w^c$ so in particular $P_{I'} P_{J'}$ does not vanish in $F_{n,n-1,w}$, a contradiction. So $w_2 = n$. Note that $w_1 \neq n-1$ otherwise $w = (n-1, n, n-2, \dots, 1)$ contradicting our assumption. So we have deduced that $w = (w_1, n, n-1, \dots)$ where $w_1 \le n-2$. Now we take cases on $\lvert J \rvert$.

If $\lvert J \rvert = 1$ write $J = \{ j \}$, then the relation is given by 
$$P_{n,i} P_j - P_n P_{j,i}.$$
We show that $i = n-1$ by contradiction. Suppose that $i \neq n-1$ so $i < n-1$. Consider the relation $P_{n-1, i} P_{j} - P_{n-1} P_{j,i}$ in $\inw{n-2}(I_{n-1})$. Since $w = (w_1, n, n-1, \dots)$ we have that $\ul{w} = (w_1, n-1, \dots)$. It is easy to check that $P_{n-1, i} P_{j}$ does not vanish but $P_{n-1}$ does vanish in $F_{n-1,n-2,\ul{w}}$. This contradicts the assumption that $\ul{w} \in T_{n-1,n-2}$.

So $i = n-1$ and $I = \{n, n-1 \}$. Since $P_I$ does not vanish in $F_{n,n-1,w}$ we have $\{n-1, n\} \le \{w_1, w_2 \}$. Therefore $w_1 = n-1$ and $w_2 = n$. Since $w \in S_n^>$ we deduce that $w = (n-1, n, n-2, \dots, 1)$, a contradiction.

If $\lvert J \rvert \ge 2$ then write $J = \{j_1 < j_2 < \dots < j_{\lvert J \rvert} \}$. The relation is given by
$$P_{n,i} P_{j_1, j_2, j_3, \dots, j_{\lvert J \rvert}} =
P_{n,j_2} P_{j_1, i  , j_3, \dots, j_{\lvert J \rvert}}.$$
If $P_{n, j_2}$ does not vanish in $F_{n,n-1,w}$ then $P_{j_1, i, j_3, \dots, j_{\lvert J \rvert}}$ must vanish. By Corollary~\ref{lem:Swc_1}, $j_1 > w_1$. But $w_1 < j_1 < i \le w_1$, a contradiction. So $P_{n, j_2}$ vanishes in $F_{n,n-1,w}$. Since $P_I$ does not vanish and $w \neq (n-1, n, n-2, \dots, 1)$ we have that $i < n-1$.

Now we claim that $j_2 < n-1$. If $j_2 = n$ then $\lvert J \rvert = 2$ and so $B_n-1(J) \neq id$, a contradiction. If $j_2 = n-1$ then the relation is given by $P_{n,i} P_{j_1, n-1} - P_{n, n-1} P_{j_1, i}$. Since $i < n-1$ we have $j_1 < n-2$. Now consider the relation $P_{n-1, i} P_{j_1, n-2} - P_{n-1, n-2} P_{j_1, i}$ in $\inw{n-2}(I_{n-1})$. Note that $\ul{w} = (w_1, n-1, \dots)$ so $P_{n-1, i} P_{j_1, n-2}$ does not vanish in $F_{n-1,n-2,\ul{w}}$. Since $\ul{w} \in T_{n-1,n-2}$ it follows that $P_{n-1, n-2} P_{j_1, i}$ does not vanish. In particular $P_{n-1, n-2}$ does not vanish. Therefore $w_1 = n-2$ so $\ul{w} = (n-2, n-1, n-3, \dots, 1)$, however this contradicts Lemma~\ref{lem:a2_exception}. And so $j_2 < n-1$. Now we consider the relation $P_{n-1, i} P_{j_1, j_2} - P_{n-1, j_2} P_{j_1, i}$ in $\inw{n-2}(I_{n-1})$. Clearly $P_{n-1, i} P_{j_1, j_2} $ does not vanish in $F_{n-1,n-2,\ul{w}}$. Since $P_{n, j_2}$ vanishes in $F_{n,n-1,w}$ it follows that $P_{n-1, j_2}$ vanishes in $F_{n-1,n-2,\ul{w}}$. And so we have shown $F_{n-1,n-2,\ul{w}}$ is  non-binomial, a contradiction.
\end{proof}

\begin{lemma}\label{lem:toric_thm2_a3}
We have that $\tilde{A}_2 \subset T_{n, n-1}$.
\end{lemma}

\begin{proof}
Take $w \in \tilde{A}_2$. Since $\tilde{A}_2 \subset \ul{T}_{n,n-1} \cap \{w \in S_n : \text{ if } w_s = n-1, w_t = n \text{ then } t \ge s + 1 \}$ it follows that $\ul{w} \in S_n^>$. So by Theorem~C part C1 we have $w \in T_{n,n}$. Suppose $P_I P_J - P_{I'} P_{J'}$ is a relation in $\inw{n-1}(I_n)$ with $P_I P_J$ non-vanishing in $F_{n,n-1,w}$. We show that $P_{I'}P_{J'}$ is non-vanishing by taking cases on $\lvert I \rvert$ and $\lvert J \rvert$. We assume without loss of generality that $|J| \le |I|$ and so we must have that either $|I|, |J| <t$ or $|I| \ge t$ and $|J| < t$ or $|I|,|J| \ge t$.

\textbf{Case 1.} Let $\lvert I \rvert, \lvert J \rvert < t$.
Since $I \le \{w_1, \dots, w_{\lvert I \rvert} \}$, $J \le \{w_1, \dots, w_{\lvert J \rvert} \}$ and $w_t = n$, we deduce that $n \not \in I$ and $n \not \in J$. Therefore $B_{n-1}(I) = B_{n-1}(J) = id$ so $B_{n-1}(I') = B_{n-1}(J') = id$. Since $w \in T_{n,n}$ we have that $P_{I'} P_{J'}$ does not vanish in $F_{n,n,w}$ and so does not vanish in $F_{n,n-1,w}$.

\textbf{Case 2.} Let $\lvert I \rvert \ge t$, $\lvert J \rvert < t$. Note that $B_{n-1}(J) = id$. We show that $B_{n-1}(I) = id$ by contradiction. Suppose $B_{n-1}(I) \neq id$ then $I = \{i, n \}$ for some $1 \le i \le n-1$. Since $\lvert J \rvert < \lvert I \vert$ then $J = \{j \}$ for some $1 \le j \le n$. So the relation is given by
$$P_{n,i}P_j - P_n P_{j,i}.$$
Since $P_{n,i} \neq 0$ in the ideal $\inw{n-1}(I_n)$, we have that $\{i,n\} \le \{w_1, w_2 \}$. Hence, $n \in \{w_1, w_2 \}$. 
Note that $w \in T_{n,n}$ has the descending property by Lemma~\ref{lem:WDecrease}. This together with the assumption that $t \ge s+1$ imply that  $w = (n-1, n, n-2, \dots, 1)$. However $\ul{w} = (n-1, n-2, \dots, 1)$ and so $F_{n-1,n-2,\ul{w}} = \inw{n-2}(I_{n-1})$ is   binomial, a contradiction. 

So $B_{n-1}(I) = id$. Then similarly to Case 1, we deduce that $P_{I'}P_{J'}$ does not vanish in $F_{n,n-1,w}$.

\textbf{Case 3.} Let $\lvert I \rvert, \lvert J \rvert \ge t$. We show that $B_{n-1}(I) = B_{n-1}(J) = id$ by contradiction. Suppose $B_{n-1}(I) \neq id $ so $I = \{ i, n\}$ for some $1 \le i \le n-1$. Since $P_I$ does not vanish in $F_{n,n-1,w}$, we have that $n \in \{w_1, w_2 \}$. Since $t \ge s+1$ and $w$ has the descending property, we deduce that $w = (n-1, n, n-2, \dots, 1)$. So $\ul{w} = (n-1,n-2 \dots, 1) \in T_{n-1,n-2}$, a contradiction.

So we have $B_{n-1}(I) = B_{n-1}(J) = id$, hence $P_{I'}P_{J'}$ does not vanish in $F_{n,{n-1},w}$. 
\end{proof}

\begin{lemma}\label{lem:toric_thm2_non_toric_extn}
Let $w \in S_n$. If $\ul{w} \in N_{n-1,n-1}$ then $w \in N_{n,n-1} $.
\end{lemma}

\begin{proof}
We write $w = (w_1, \dots, w_n)$ and $w_t = n$ for some $t \in \{1, \dots, n\}$. Suppose $P_{I}P_{J}$ is a monomial in $F_{n,n-1,\ul{w}}$ arising from a relation $P_{I}P_{J} - P_{I'} P_{J'}$ in $\inw{n-1}(I_{n-1})$. Without loss of generality, $\lvert I \rvert = \lvert I' \rvert \ge \lvert J \rvert = \lvert J' \rvert$. We take cases on $\lvert I \rvert, \lvert J \rvert$ and $t$. In particular we must have that either $|I|, |J| <t$ or $|I| \ge t$ and $|J| < t$ or $|I|,|J| \ge t$.

\textbf{Case 1.} Let $\lvert I \rvert, \lvert J \rvert < t$. We have that $P_{I}P_{J} - P_{I'} P_{J'}$ is a relation in $\inw{n-1}(I_n)$. Since $w$ and $\ul{w}$ agree on $w_1, \dots, w_{t-1}$ we have that $P_I P_J$ is a monomial in $F_{n,n-1,w}$ and so $w \in N_{n,n-1} $.

\textbf{Case 2.} Let $\lvert I \rvert \ge t$, $\lvert J \rvert <t$. Note that we have $\lvert I \rvert \ge 2$ so we let $\tilde{I} = I \cup \{ n\}$ and $\tilde{I'} = I' \cup \{n \}$. By Corollary~\ref{lem:Swc_1}, $P_{\tilde{I}}$ does not vanish in $F_{n,n-1,w}$. Since $\lvert I \rvert \ge 2$, $B_{n-1}(\tilde{I}) = B_{n-1}(\tilde{I'}) = id$. Hence we have the following is a relation in $\inw{n-1}(I_n)$:
$$P_{\tilde{I}}P_{J} - P_{\tilde{I'}} P_{J'}.$$
By Corollary~\ref{lem:Swc_1}, $P_{\tilde{I'}} P_{J'}$ vanishes in $F_{n,n-1,w}$. So $w \in N_{n,n-1} $.

\textbf{Case 3.} Let $\lvert I \rvert, \lvert J \rvert \ge t$. We write $\tilde{L} = L \cup \{ n\}$ for each $L \in \{I,J,I',J' \}$. Suppose $\lvert I \rvert, \lvert J \rvert \ge 2$ then we have $B_{n-1}(\tilde{I}) = B_{n-1}(\tilde{J}) = B_{n-1}({\tilde{I'}}) = B_{n-1}(\tilde{J'}) = id$ and so we have the following relation in $\inw{n-1}(I_n)$:
$$P_{\tilde{I}}P_{\tilde{J}} - P_{\tilde{I'}} P_{\tilde{J'}}.$$
By Corollary~\ref{lem:Swc_1} we have that $P_{\tilde{I}}P_{\tilde{J}}$ does not vanish and $P_{\tilde{I'}} P_{\tilde{J'}}$ does vanish in $F_{n,{n-1},w}$. So $w \in N_{n,{n-1}} $.

Now suppose $\lvert I \rvert = 1$ and so $t = 1$. Since the relation $P_{I}P_{J} - P_{I'} P_{J'}$ is non-trivial, $\lvert J \rvert \ge 2$ and so $B_{n-1}(\tilde{J}) = B_{n-1}(\tilde{J'}) = id$. We write $I = \{i \}$ and $J = \{j_1, j_2, \dots, j_{\lvert J \rvert} \}$. The relation is given by:
$$P_{i}P_{j_1, j_2, \dots, j_{\lvert J \rvert}} = 
P_{j_1}P_{i, j_2, \dots, j_{\lvert J \rvert}}.$$
Since $P_{J}$ does not vanish in $F_{n-1,n-1,\ul{w}}$ we have that $j_1 \le \min\{w_1, \dots, w_{\lvert J \rvert} \} \le w_1$ and so $P_{j_1}$ does not vanish. We deduce that $P_{i, j_2, \dots, j_{\lvert J \rvert}}$ vanishes in $F_{n-1,n-1,\ul{w}}$. Consider the relation $P_{I}P_{\tilde{J}} - P_{I'} P_{\tilde{J'}}$ in $\inw{n-1}(I_n)$ given by:
$$P_{i}P_{j_1, j_2, \dots, j_{\lvert J \rvert}, n} = 
P_{j_1}P_{i, j_2, \dots, j_{\lvert J \rvert}, n}.$$
By Corollary~\ref{lem:Swc_1}, $P_{\tilde{J}}$ does not vanish in $F_{n,{n-1},w}$ and $P_{\tilde{J'}}$ does vanish. And so $w \in N_{n,{n-1}} $.
\end{proof}

\subsection{Non-diagonal and non-semi-diagonal matching fields}\label{sec:proofs_other}

Throughout this section, unless otherwise stated, we assume that $\ell \in \{1, \dots, n-2 \}$. 
We recall the definitions of the sets $A'_2$ and $A_3$ from Definition~\ref{def:notation}. Below, we state and prove the main result of this section. Similarly to the semi-diagonal case, the main result of this section decomposes $T_{n,\ell}$ into three main parts: $A_1, A_2'$ and $A_3$ along with the exceptional permutation $(n,\ell,n-1,n-2, \dots, \ell + 1, \ell - 1, \dots, 1)$. The proof is similar to the semi-diagonal case, in fact steps a, b and c follow the same structure. Steps d and e carefully use the structure of the matching field to show how the sets $A_2', A_3$ and the exceptional permutation arises in the decomposition of $T_{n,\ell}$.
\begin{theorem}\label{Intro:l} $T_{n,\ell} = A_1 \cup A'_2\cup A_3 \cup \{(n, \ell, n-1, n-2, \dots, \ell+1, \ell-1, \dots, 1)\}$ for $1\leq\ell\leq n-2$.
\end{theorem}
\begin{proof}
Before stating the proof we first note that for every block diagonal matching field $B = (1, \dots, \ell \mid \ell+1, \dots, n)$ for $1 \le \ell \le n-2$ and every subset $I = \{i_1 < i_2 < \dots < i_{\lvert I \rvert} \}$ of $[n]$ we have the following cases: 

\begin{itemize}
\item If $B_\ell(I) = id$, then either $i_1, i_2 \in \{\ell+1, \dots, n \}$, or $i_1, i_2 \in \{1, \dots, \ell \}$ or $\lvert I \rvert = 1$,

\item If $B_\ell(I) \neq id$ then $i_1 \in \{1, \dots, \ell \}$ and $i_2 \in \{\ell+1, \dots, n \}$.
\end{itemize}

We will now break down the proof into the following steps.

\medskip
\noindent{\bf Step a.} $A_1\cup A_2'\cup A_3 \subset T_{n,\ell}$ and so $RHS \subseteq LHS$. 

First we show $A_1 \subset T_{n,\ell}$. Suppose $\ul{w} = (w_1, \dots, w_{n-3}, n-1, n-2) \in Z_{n-1}$. Then for $w = (w_1, \dots, w_{n-3}, n, n-1, n-2)$ and $w = (w_1, \dots, w_{n-3}, n-1, n, n-2)$, $F_{n,\ell,w}$ is   binomial by Proposition~\ref{lem:thm_ToricFlag_conv_1} so $A_1 \subset T_{n,\ell}$. Next $A'_2 \subset T_{n,\ell}$ and $A_3 \subset T_{n,\ell}$ by Lemma~\ref{lem:toric_thm1_a2} and Lemma~\ref{lem:toric_thm1_a3} respectively. By Lemma~\ref{lem:toric_thm1_exceptional} we have $(n, \ell, n-1, \dots, \ell+1, \ell-1, \dots, 1) \in T_{n,\ell}$. So we have shown $A_1 \cup A_2' \cup A_3 \cup \{(n, \ell, n-1, \dots, \ell+1, \ell-1, \dots, 1) \} \subseteq T_{n,\ell}$.

\medskip
\noindent{\bf Step b.} For any $w \in T_{n,\ell}$ with $w \neq (n, \ell, n-1, \dots, \ell+1, \ell-1, \dots, 1)$ we have that $\ul{w} \in Z_{n-1} \cup T_{n-1,\ell}$.

Now take $w \in T_{n,\ell}$ with $w \neq (n, \ell, n-1, \dots, \ell+1, \ell-1, \dots, 1)$. By Lemma~\ref{lem:toric_thm1_non_toric_extn}, $\ul{w} \in T_{n-1,\ell} \cup Z_{n-1}$. By Lemma~\ref{lem:WDec3} we have that $w$ has the descending property. We denote $w_t = n$ and $w_s = n-1$.

\medskip
\noindent{\bf Step c.} If $\ul{w} \in Z_{n-1} $ then $w \in A_1$.

First suppose $\ul{w} \in Z_{n-1}$. By Theorem~B, either $\ul{w} = (w_1, \dots, w_{n-2}, n-1)$ or $\ul{w} = (w_1, \dots, w_{n-3}, n-1, n-2)$. If $\ul{w} = (w_1, \dots, w_{n-2}, n-1)$ then $w = (w_1, \dots, w_{n-2}, n, n-1)$ or $w = (w_1, \dots, w_{n-2}, n-1, n)$ since $w$ has the descending property. However $F_{n,\ell,w} = 0$ by Theorem~B, a contradiction. So $\ul{w} = (w_1, \dots, w_{n-3}, n-1, n-2)$. If $w = (w_1, \dots, w_{n-3}, n-1, n-2, n)$ then $F_{n,\ell,w} = 0$, so $w = (w_1, \dots, w_{n-3}, n-1, n, n-2)$ or $w = (w_1, \dots, w_{n-3}, n, n-1, n-2)$. Therefore $w \in A_1$.

\medskip
\noindent{\bf Step d.} If $\ul{w} \in T_{n-1,\ell}$ and has descending property then $w \in A_2'$.

Next suppose $\underline{w} \in T_{n-1,\ell} $ and $\underline{w}$ has the descending property. Since $w$ has the descending property we must have $t \ge s-1$ and so $w \in A_2'$.

\medskip
\noindent{\bf Step e.} If $\ul{w} \in T_{n-1,\ell}$ and does not have descending property then $w \in A_3$.

If $\underline{w} \in T_{n-1,\ell} $ and $\underline{w}$ does not have the descending property then by Corollary~\ref{cor:l1...n-2_exceptional}, $\underline{w} = (n-1, \ell, n-2, \dots, 1)$. Since $w$ has the descending property, we must have $t \ge s+2$. And so we have shown $w \in A_3'$.
\end{proof}
\begin{lemma} \label{lem:toric_thm1_exceptional}
We have $w = (n, \ell, n-1, n-2, \dots, 1) \in T_{n,\ell}$.
\end{lemma}

\begin{proof}
Suppose $P_I P_J - P_{I'} P_{J'}$ is a relation in $\inwb(I_n)$ and $P_{I} P_{J}$ does not vanish in $F_{n,\ell,w}$. We show that $P_{I'} P_{J'}$ does not vanish either and hence $F_{n,\ell,w}$ contains no monomials. Write $I = \{i_1 < \dots < i_{\lvert I \rvert} \}$ and $J = \{j_1 < \dots < j_{\lvert J \rvert} \}$ and assume without loss of generality that $\lvert I \rvert = \lvert I' \rvert$ and $\lvert J \rvert = \lvert J' \rvert$. We take cases on $B_\ell(I)$ and $B_\ell(J)$. In particular, we may assume that either $B_\ell(I) = B_\ell(J) = id$ or $B_\ell(I) \neq id$.

\textbf{Case 1.} Let $B_\ell(I) = B_\ell(J) = id$. If $\lvert I \rvert = 1$ then $\lvert J \rvert \ge 2$ otherwise the relation would be trivial. So the relation is given by:
$$P_{i_1} P_{j_1, j_2, \dots, j_{\lvert J \rvert}} - P_{j_1} P_{i_1, j_2, \dots, j_{\lvert J \rvert}}.$$
Since $B_\ell(J) = id$ then either $j_1, j_2 \in \{1, \dots, \ell \}$ or $j_1, j_2 \in \{\ell+1, \dots, n \}$. 
However, if $j_1, j_2 \in \{ \ell +1, \dots, n\}$ then it follows that $P_J$ vanishes in $F_{n,\ell,w}$, a contradiction. Suppose $j_1, j_2 \in \{1, \dots, \ell \}$. The fact that $J' \not \in S_w$ follows from the following two observations: 
\begin{itemize}
\item[(i)]  if $L \subseteq [n],\ \lvert L \rvert \ge 2$ then $L \in S_w \iff L \in S_v$, where $v = (\ell, n, n-1, \dots, 2, 1)$,
\item[(ii)]  Corollary~\ref{lem:Swc_1} can be applied to $S_v$ because $v$ has the descending property.
\end{itemize}
Since $j_2 \in \{1, \dots, \ell \}$ we have that $P_{i_1, j_2, \dots, j_{\lvert J \rvert}}$ does not vanish in $F_{n,\ell,w}$ by Corollary~\ref{lem:Swc_1}. It is clear that $P_{j_1}$ does not vanish and so we have $P_{j_1} P_{i_1, j_2, \dots, j_{\lvert J \rvert}} - P_{I'} P_{J'}$ does not vanish in $F_{n,\ell,w}$.

If $\lvert I \rvert, \lvert J \rvert \ge 2$ then $i_1, j_1 \le \ell$ because $I \le \{n, \ell, n-1, \dots \}$ and $J \le \{n, \ell, n-1, \dots \}$. Since $B_\ell(I) = B_\ell(J) = id$  we must have $i_1, i_2, j_1, j_2 \in \{1, \dots,\ell \}$ and so $I' \cap \{1, \dots,\ell \} \neq \varnothing$ and $J' \cap \{1, \dots,\ell \} \neq \varnothing$. It is easy to show that $I' \le \{n,\ell, n-1, \dots \}$ and $J' \le \{n,\ell, n-1 \dots \}$. Therefore $P_{I'} P_{J'}$ does not vanish in $F_{n,\ell,w}$.

\textbf{Case 2.} Let $B_\ell(I) \neq id$. We have $i_1 \in \{ 1, \dots,\ell\}$ and $i_2 \in \{\ell+1, \dots, n \}$. If $\lvert J \rvert = 1$ then the relation is
$$P_{i_2, i_1, i_3, \dots, i_{\lvert I \rvert}} P_{j_1} - P_{j_1, i_1, i_3, \dots, i_{\lvert I \rvert}} P_{i_2}.$$
Note that we have $i_1 \in I'$ and so $P_{I'}$ does not vanish in $F_{n,\ell,w}$. Therefore $P_{I'} P_{J'}$ does not vanish in $F_{n,\ell,w}$.

If $\lvert I \rvert, \lvert J \rvert \ge 2$, suppose $B_\ell(J) \neq id$. Since $i_1, j_1 \in \{1, \dots,\ell \}$ appear at the same index in $I$ and $J$ respectively, it is easy to check that $I' \cap \{1, \dots,\ell \} \neq \varnothing$ and $J' \cap \{1, \dots,\ell \} \neq \varnothing$. And so $P_{I'} P_{J'}$ does not vanish in $F_{n,\ell,w}$. On the other hand if $B_\ell(J) = id$, since $J \le \{n,\ell, n-1, \dots \}$, then $j_1 \in \{1, \dots,\ell \}$. So $j_2 \in \{1, \dots,\ell \}$ as well, because $B_\ell(J) = id$. We have $i_1, j_1, j_2 \in \{1, \dots,\ell \}$, it follows that $I' \cap \{1, \dots,\ell \} \neq \varnothing$ and $J' \cap \{1, \dots,\ell \} \neq \varnothing$. So $P_{I'}P_{J'}$ does not vanish in $F_{n,\ell,w}$.
\end{proof}

\begin{lemma} \label{lem:WDec3}
If $w \in T_{n,\ell}$ and $w \neq (n,\ell, n-1, \dots,\ell+1,\ell-1, \dots, 1)$ then $w \in S_n^>$.
\end{lemma}

\begin{proof}
Let $w = (w_1, \dots, w_n) \in T_{n,\ell} \backslash \{(n,\ell, n-1, \dots, \ell+1, \ell-1, \dots, 1) \}$ and $w_t = n$. Suppose by contradiction there exists $k > t$ such that $w_k < w_{k+1}$. Let $I = \{w_1, \dots, w_k\}$ and $I' = \{w_1, \dots, w_{k-1}, w_{k+1} \}$. Since $w_k < w_{k+1}$ we have that $P_{I'}$ vanishes in $F_{n,\ell,w}$ whereas $P_I$ does not. Let us write in ascending order $I \cup I' = \{\alpha, w_k, \beta, w_{k+1}, \gamma, n \}$ for some ordered subsets $\alpha, \beta$ and $\gamma$ of $[n]$. Note that $I = \{\alpha, w_k, \beta, \gamma, n \}$ and $I' = \{\alpha, \beta, w_{k+1}, \gamma, n \}$. Let $J = \{\alpha, \beta, w_{k+1}, \gamma \}$ and $J' = \{\alpha, w_k, \beta, \gamma \}$ and note that both $P_J$ and $P_{J'}$ do not vanish in $F_{n,\ell,w}$. We take cases on $p = \lvert (\alpha \cup \beta \cup \gamma ) \cap \{1, \dots, k \} \rvert$. In particular, we either have that $p \ge 2$ or $p = 1$ or $p = 0$. Where necessary we will need to take cases on the permutation $B_\ell(I)$ which is either the identity or the transposition $(1,2)$.

\textbf{Case 1.} Let $p \ge 2$. We have $B_\ell(I) = B_\ell(J) = B_\ell(I') = B_\ell(J') = id$ hence $P_I P_J - P_{I'} P_{J'}$ is a relation in $\inwb(I_n)$. And so $P_I P_J$ is a monomial in $F_{n,\ell,w}$ hence $w \not \in T_{n,\ell}$, a contradiction.

\textbf{Case 2.} Let $p = 1$. We now consider cases for $B_\ell(I)$. 

\textbf{Case 2a.} Let $B_\ell(I) = id$. It follows that $w_k \in \{ 1, \dots,\ell\}$. If $w_{k+1} \in \{1, \dots\ell \}$ then $B_\ell(I) = B_\ell(J) = B_\ell(I') = B_\ell(J') = id$ and so $P_I P_J - P_{I'} P_{J'}$ is a relation in $\inwb(I_n)$. Therefore $F_{n,\ell,w}$ contains the monomial $P_I P_J$, a contradiction.

If $w_{k+1} \in \{\ell+1, \dots, n \}$ then $B_\ell(I') \neq id$. We see that $P_I P_J - P_{I'} P_{J'}$ is a valid relation in $F_{n,\ell,w}$ as follows. Let $M = \{\alpha, w_k, \beta, \gamma \}$ and $N = \{\alpha, \beta, w_{k+1}, \gamma \}$. The relation can be written as
$$P_{M \cup \{n \}} P_{N} - P_{N \cup \{n \}} P_{M}.$$
Since $\lvert M \rvert = \lvert N \rvert \ge 2$, it follows that $B_\ell(M) = B_\ell(M \cup \{n \})$ and $B_\ell(N) = B_\ell(N \cup \{n \})$ and so this is a relation in $\inwb(I_n)$. Hence $P_I P_J$ is a monomial in $F_{n,\ell,w}$, a contradiction.

\textbf{Case 2b.} Let $B_\ell(I) \neq id$. We have $w_k \in \{\ell+1, \dots, n \}$ and so $w_{k+1} \in \{\ell+1, \dots, n \}$ hence $B_\ell(I') \neq id$. It follows that $B_\ell(J) \neq id$ and $B_\ell(J') \neq id$. We deduce that $P_I P_J - P_{I'} P_{J'}$ is a relation in $\inwb(I_n)$ and so $P_I P_J$ is a monomial in $F_{n,\ell,w}$, a contradiction.

\textbf{Case 3.} Let $p = 0$. We consider cases for $B_\ell(I)$.

\textbf{Case 3a.} Let $B_\ell(I) = id$. So $w_k \in \{\ell+1, \dots, n \}$ and so $w_{k+1} \in \{\ell+1, \dots, n \}$ hence $B_\ell(I') = id$. It follows that $B_\ell(J) = B_\ell(J') = id$ and so $P_I P_J - P_{I'} P_{J'}$ is a valid relation in $\inwb(I_n)$. Therefore $F_{n,\ell,w}$ contains the monomial $P_I P_J$, a contradiction. 

\textbf{Case 3b.} Let $B_\ell(I) \neq id$. So $w_k \in \{\ell+1, \dots,  n\}$. We take cases on $w_{k+1}$.

\textbf{Case 3b.i.} Let $w_{k+1} \in \{1, \dots,\ell \}$. So $B_\ell(I') \neq id$. If $\lvert \alpha \cup \beta \cup \gamma \rvert > 0$ then we have that $B_\ell(J) \neq id$ and $B_\ell(J') \neq id$ and so $P_I P_J - P_{I'} P_{J'}$ is a relation in $\inwb(I_n)$. Therefore $P_I P_J$ is a monomial in $F_{n,\ell,w}$, a contradiction. 

If $\lvert \alpha \cup \beta \cup \gamma \rvert = 0$ then we have $w = (n, w_2, w_3, \dots, w_n )$ with $k = 2$. Consider the relation in $\inwb(I_n)$
$P_n P_{w_2 w_3} - P_{w_2} P_{n w_3}.$
This is indeed a valid relation which gives rise to a monomial $P_n P_{w_2 w_3}$ in $F_{n,\ell,w}$, a contradiction.

\textbf{Case 3b.ii.} Let $w_{k+1} \in \{\ell+1, \dots, n \}$. So $B_\ell(I') = id$. If $\lvert \alpha \cup \beta \cup \gamma \rvert > 0$ then it is easy to check that $P_I P_J - P_{I'} P_{J'}$ is a relation in $\inwb(I_n)$ where $B_\ell(J) = id$ and $B_\ell(J') \neq id$. So $P_I P_J$ is a monomial in $F_{n,\ell,w}$, a contradiction.

If $\lvert \alpha \cup \beta \cup \gamma \rvert = 0$ then $w = (n, w_2, w_3, \dots, w_n)$ with $k = 2$. Now without loss of generality we may assume that $w_3 > w_4 > \dots > w_n$ otherwise we may use one of the previous cases. So $w = (n, w_2, n-1, n-2, \dots, w_2+1, w_2-1, \dots, 1)$. Also by assumption we have $w_2 \neq\ell$ so $w_2 \le\ell-1$. Consider the relation
$P_n P_{w_2,\ell} - P_{w_2} P_{n,\ell}.$
Clearly this is a relation in $\inwb(I_n)$. The monomial $P_n P_{w_2,\ell}$ does not vanish in $F_{n,\ell,w}$ but $P_{n,\ell}$ does vanish and so $F_{n,\ell,w}$ is  non-binomial, a contradiction.
\end{proof}

As an immediate corollary of Lemma~\ref{lem:toric_thm1_exceptional} and Lemma~\ref{lem:WDec3} we have that:
\begin{corollary}\label{cor:l1...n-2_exceptional}
If $w \in T_{n,\ell}\backslash S_n^>$, then $w = (n,\ell, n-1, \dots,\ell+1,\ell-1, \dots, 1)$.
\end{corollary}

\begin{lemma}\label{lem:toric_thm1_a2}
We have $A_2' \subset T_{n,\ell}$.
\end{lemma}

\begin{proof}
Let $w \in A_2'$. Let $P_I P_J - P_{I'} P_{J'}$ be a relation in $\inwb(I_n)$ where $P_I P_J$ does not vanish in $F_{n,\ell,w}$. We show that $P_{I'} P_{J'}$ does not vanish in $F_{n,\ell,w}$ by taking cases on $\lvert I \rvert, \lvert J \rvert$ and $t$. We may assume that $|I| \ge |J|$ and so we must either have $|I|,|J| <t$ or $|I| \ge t$ and $|J| < t$ or $|I|,|J| \ge t$.

\smallskip

\textbf{Case 1.} Let $\lvert I \rvert, \lvert J \rvert < t$. Since $w$ and $\underline{w}$ agree on $w_1, \dots, w_{t-1}$, we deduce that $P_{I}P_{J} - P_{I'}P_{J'}$ is a relation in $F_{n-1,\ell,\underline{w}}$. Since $F_{n-1,\ell,\underline{w}}$ is   binomial, we conclude that $P_{I'}P_{J'}$ does not vanish in $F_{n,\ell,w}$.

\smallskip

\textbf{Case 2.} Let $\lvert I \rvert \ge t, \lvert J \rvert < t$. Write $I = \{i_1, \dots, i_{\lvert I \rvert} \}$. If $B_\ell(I) =B_\ell(I \backslash \{i_{\lvert I \rvert} \})$ and $B_\ell(I') =B_\ell(I' \backslash \{ i_{\lvert I \rvert} \})$, then we have
the following relation in $\inw\ell(I_{n-1})$:
$$P_{I \backslash \{ i_{\lvert I \rvert} \}} P_{J} - P_{I' \backslash \{i_{\lvert I \rvert} \}} P_{J'}.$$
Clearly $P_{I \backslash \{ i_{\lvert I \rvert} \}} P_{J}$ does not vanish in $F_{n-1,\ell,\underline{w}}$. Since $F_{n-1,\ell,\underline{w}}$ is   binomial we have that $P_{I' \backslash \{i_{\lvert I \rvert} \}} P_{J'}$ does not vanish in $F_{n-1,\ell,\underline{w}}$. So by Corollary~\ref{lem:Swc_1}, $P_{I'} P_{J'}$ does not vanish in $F_{n,\ell,w}$.

If $B_\ell(I) \neq B_\ell(I \backslash \{i_{\lvert I \rvert} \})$ or  $B_\ell(I') \neq B_\ell(I' \backslash \{ i_{\lvert I \rvert} \})$, then $\lvert I \rvert = 2$, $t = 2$ and $\lvert J \rvert = 1$. We write $J = \{j \}$. Since $\underline{w}$ has the descending property and $w_2 = n$ we deduce $w = (w_1, n, n-1, \dots, w_1 +1, w_1 -1, \dots, 1)$. First suppose $B_\ell(I) \neq B_\ell(I \backslash \{i_{\lvert I \rvert} \})$, so $B_\ell(I) \neq id$. The relation is given by
$P_{i_2, i_1} P_j - P_{j, i_1} P_{i_2}.$
Note that $j \le w_1$ so $P_{j, i_1}$ does not vanish in $F_{n,\ell,w}$. We show that $P_{i_2}$ does not vanish by contradiction. Suppose $i_2 > w_1$. We have that $w_1 < n-1$ because $w \neq (n-1, n, n-2, \dots, 1)$. Note that $i_1, j \le w_1$ since $P_I P_J$ does not vanish in $F_{n,\ell,w}$. Consider the following relation in $\inw\ell(I_{n-1})$:
$P_{n-1, i_1} P_{j} - P_{j, i_1} P_{n-1}.$
Clearly $P_{n-1, i_1} P_j$ does not vanish in $F_{n-1,\ell,\underline{w}}$ however $P_{n-1}$ does vanish and so $F_{n-1,\ell,\underline{w}}$ is  non-binomial, a contradiction.

Secondly suppose $B_\ell(I) = B_\ell(I \backslash \{i_{\lvert I \rvert} \}) = id $. Then the relation is given by
$P_{i_1, i_2} P_{j} - P_{j, i_2} P_{i_1}.$
Since $P_{i_1, i_2} P_{j}$ does not vanish in $F_{n,\ell,w}$ we have $i_1 \le w_1$ and $j \le w_1$. And so $P_{j, i_2} P_{i_1}$ does not vanish.

\smallskip

\textbf{Case 3.} Let $\lvert I \rvert, \lvert J \rvert \ge t$. We write $I = \{ i_1 < \dots < i_{\lvert I \rvert}\}, J= \{j_1 < \dots < j_{\lvert J \rvert} \}, I' = \{i_1' < \dots < i_{\lvert I \rvert}'\}$ and $J' = \{j_1' < \dots < j_{\lvert J \rvert}' \}$. Suppose $t \ge 2$. By assumption $\underline{w}$ has the descending property and $t \ge s-1$, hence $w$ has the descending property. So by Corollary~\ref{lem:Swc_1} we have $\{i_1, \dots, i_t \} \le \{w_1, \dots, w_t \}$ and $\{j_1, \dots, j_t \} \le \{w_1, \dots, w_t \}$. 
For each $\ell \ge 3$, $\{i'_\ell, j'_\ell \} = \{i_\ell, j_\ell \}$ because $B_\ell (L) $ does not permute any index $\ell \ge 3$ which means that $B_\ell (L) (\ell) = \ell$ for any $L \subseteq [n]$ and $\ell \ge 3$. So $\{i_1', i_2', j_1', j_2' \} = \{i_1, i_2, j_1, j_2 \}$.
It follows that $\{i_1', \dots, i_t' \} \le \{w_1, \dots, w_t \}$ and $\{j_1', \dots, j_t' \} \le \{w_1, \dots, w_t \}$. So by Corollary~\ref{lem:Swc_1}, $P_{I'} P_{J'}$ does not vanish in $F_{n,\ell,w}$.

Suppose $t = 1$. Since $\underline{w}$ has the descending property we have $w = (n, n-1, \dots, 1)$. Clearly $P_{I'} P_{J'}$ does not vanish in $F_{n,\ell,w}$ since $F_{n,\ell,w} = \inwb(I_n)$ and no variable vanishes.
\end{proof}

\begin{lemma}\label{lem:toric_thm1_a3}
We have $A_3 \subset T_{n,\ell}$.
\end{lemma}

\begin{proof}
Suppose $w \in S_n$ with $\underline{w} \in T_{n-1,\ell} $ and $\underline{w}$ does not have the descending property. Let $w_s = n-1$ and $w_t = n$ and suppose $t \ge s+2$. By Corollary~\ref{cor:l1...n-2_exceptional}, $\underline{w} = (n-1,\ell, n-2, \dots,\ell+1,\ell-1, \dots, 1)$. We deduce that $w$ has the descending property.

Let $P_I P_J - P_{I'} P_{J'}$ be a relation in $\inwb(I_n)$ and suppose $P_I P_J$ does not vanish in $F_{n,\ell,w}$. We show that $P_{I'} P_{J'}$ does not vanish by taking cases on $\lvert I \rvert, \lvert J \rvert$ and $t$. We may assume that $|I| \ge |J|$ and so we must have that either $|I|,|J| <t$ or $|I| \ge t$ and $|J| < t$ or $|I|,|J| \ge t$.

\textbf{Case 1.} Let $\lvert I \rvert, \lvert J \rvert < t$. Since $w$ and $\underline{w}$ agree on $w_1, \dots, w_{t-1}$, $P_I P_J - P_{I'} P_{J'}$ is a relation in $F_{n-1,\ell, \underline{w}}$ which is   binomial. So $P_{I'} P_{J'}$ does not vanish in $F_{n,\ell,w}$.

\textbf{Case 2.} Let $\lvert I \rvert \ge t, \lvert J \rvert < t$. We write $I = \{i_1, \dots, i_{\lvert I \rvert} \}$. We show that we cannot have $B_\ell(I) \neq B_\ell(I \backslash \{i_{\lvert I \rvert} \})$ or $B_\ell(I') \neq B_\ell(I' \backslash \{i_{\lvert I \rvert} \})$. Otherwise we would have $\lvert I \rvert = \lvert I' \rvert = 2$ and $t = 2$. But by assumption $t \ge s+2 \ge 3$, a contradiction. So we have $B_\ell(I) = B_\ell(I \backslash \{i_{\lvert I \rvert} \})$ and $B_\ell(I') = B_\ell(I' \backslash \{i_{\lvert I \rvert} \})$. We have the following relation in $F_{n-1,\ell, \underline{w}}$:
$$P_{I \backslash \{i_{\lvert I \rvert} \}} P_{J} - P_{I' \backslash \{i_{\lvert I \rvert} \}} P_{J'}.$$
Since $F_{n-1,\ell, \underline{w}}$ is   binomial we deduce that $P_{I' \backslash \{i_{\lvert I \rvert} \}} P_{J'}$ does not vanish in $F_{n-1,\ell, \underline{w}}$. So by Corollary~\ref{lem:Swc_1}, $P_{I'} P_{J'}$ does not vanish in $F_{n,\ell,w}$.

\textbf{Case 3.} Let $\lvert I \rvert, \lvert J \rvert > t$. Write $I = \{ i_1 < \dots < i_{\lvert I \rvert}\}, J= \{j_1 < \dots < j_{\lvert J \rvert} \}, I' = \{i_1' < \dots < i_{\lvert I \rvert}'\}$ and $J' = \{j_1' < \dots < j_{\lvert J \rvert}' \}$. By assumption, $t \ge 3$ and $w$ has the descending property. So by Corollary~\ref{lem:Swc_1}, $\{i_1, \dots, i_t \} \le \{w_1, \dots, w_t \}$ and $\{j_1, \dots, j_t \} \le \{w_1, \dots, w_t \}$. For each $\ell \ge 3$, $\{i_\ell', j_\ell'\} = \{i_\ell, j_\ell \}$ and $\{i_1', i_2', j_1', j_2' \} = \{i_1, i_2, j_1, j_2 \}$. It follows that $\{i_1', \dots, i_t' \} \le \{w_1, \dots, w_t \}$ and $\{j_1', \dots, j_t' \} \le \{w_1, \dots, w_t \}$. So by Corollary~\ref{lem:Swc_1}, $P_{I'} P_{J'}$ does not vanish in $F_{n,\ell,w}$.
\end{proof}

\begin{lemma} \label{lem:toric_thm1_non_toric_extn}
Let $w \in S_n$ with $ \underline{w} \in N_{n-1,\ell}$ and  $w \neq (n,\ell, n-1, \dots, \ell +1, \ell-1, \dots, 1)$. Then $w \in N_{n,n-1} $.
\end{lemma}

\begin{proof}
Write $w = (w_1, \dots, w_n)$ with $w_t = n$. Suppose $P_I P_J - P_{I'} P_{J'}$ is a relation in $\inw\ell(I_{n-1})$ giving rise to the monomial $P_I P_J \in F_{n-1,\ell, \underline{w}}$. Without loss of generality we assume $\lvert I \rvert = \lvert I' \rvert$ and $\lvert J \rvert = \lvert J' \rvert$. For each $L \in \{I, I', J, J' \}$, let $\tilde{L} = L \cup \{n\}$. We take cases on $\lvert I \rvert, \lvert J \rvert$ and $t$. We may assume that $|I| \ge |J|$ and so we must have that either $|I|,|J| <t$ or $|I| \ge t$ and $|J| < t$ or $|I|,|J| \ge t$.

\textbf{Case 1.} Let $\lvert I \rvert, \lvert J \rvert < t$. Since $w$ and $\underline{w}$ agree on $w_1, \dots, w_{t-1}$ we have that $P_{I} P_{J}$ is a monomial in $F_{n,\ell,w}$ via the same relation and so $w \in N_{n,\ell} $.

\textbf{Case 2.} Let $\lvert I \rvert \ge t, \lvert J \rvert < t$. We have $\lvert I \rvert \ge 2$. By Corollary~\ref{lem:Swc_1}, $P_{\tilde{I}}$ does not vanish and $P_{\tilde{I'}} P_{J'}$ does vanish in $F_{n,\ell,w}$. Since $\lvert I \rvert \ge 2$, $B_\ell(I) = B_\ell(\tilde{I})$ and ${B}_\ell(I') = B_\ell(\tilde{I'})$. So we have the following relation in $\inwb(I_n)$:
$P_{\tilde{I}} P_J - P_{\tilde{I'}} P_{J'}.$
Therefore $F_{n,\ell,w}$ contains the monomial $P_{\tilde{I}} P_{J}$ and so $w \in N_{B,w}$.

\textbf{Case 3.} Let $\lvert I \rvert, \lvert J \rvert \ge t$. If $\lvert I \rvert, \lvert J \rvert \ge 2$ then $B_\ell(I) = B_\ell(\tilde{I}), B_\ell(I') = B_\ell(\tilde{I'}), B_\ell(J) = B_\ell(\tilde{J})$ and $B_\ell(J') = B_\ell(\tilde{J'})$. And so we have the following relation in $\inwb(I_n)$:
$P_{\tilde{I}} P_{\tilde{J}} - P_{\tilde{I'}} P_{\tilde{J'}}.$
By Corollary~\ref{lem:Swc_1}, $P_{\tilde{I}} P_{\tilde{J}}$ does not vanish and $P_{\tilde{I'}} P_{\tilde{J'}}$ does vanish in $F_{n,\ell,w}$. So $F_{n,\ell,w}$ is  non-binomial and $w \in N_{n,\ell} $.

If $\lvert I \rvert = 1$ or $\lvert J \rvert = 1$ then $t = 1$. So $w = (n, w_2, \dots, w_n)$. Assume by contradiction that $w$ has the descending property. Then $w = (n, n-1, \dots, 1)$ so $\underline{w} = (n-1, \dots, 1)$. Clearly $\underline{w} \in T_{n-1,\ell} $ as no variable vanishes, a contradiction. So $w$ does not have the descending property. By assumption $w \neq (n,\ell, n-1, \dots,\ell+1,\ell-1, \dots, 1)$ so by Corollary~\ref{cor:l1...n-2_exceptional}, $F_{n,\ell,w}$ is  non-binomial. And so we have shown $w \in N_{n,\ell} $. 
\end{proof}

\begin{remark}
Finally, we would like to remark that the results of this section can be generalized to Richardson varieties \cite{Ollie4}. Moreover, for Grassmannian varieties, there are other combinatorial constructions leading to toric degenerations \cite{rietsch2017newton,bossinger2020families}. Although most of these degenerations can be realized as Gr\"obner degenerations, this is not true in general; See e.g.
\cite{kateri2015family} for a family of toric degnerations that cannot be identified as a Gr\"obner degeneration.
\end{remark}

\section{Standard monomial theory for Schubert varieties
}\label{sec:standard_monomials}
In this section we study monomial bases for the ideals $F_{n,\ell,w}$ and $\init_{\bf w_\ell}(I(X(w)))$. We show for the diagonal matching field, $\ell = 0$, that if $F_{n,\ell,w}$ is monomial-free then $\init_{\bf w_\ell}(I(X(w)))$ and $F_{n,\ell,w}$ coincide and, moreover, these ideals are toric. We also show that for the other block diagonal matching fields, $\ell \in \{1, \dots, n-1 \}$, the same results hold if the initial ideal $\init_{\bf w_\ell}(I(X(w)))$ is generated in degree two.

\medskip
We begin by defining the monomial map whose kernel will coincide with $\init_{\bf w_\ell}(I(X(w)))$ when $F_{n,\ell,w}$ is monomial-free.

\begin{definition}[Restricted monomial map]\label{def:restricted_monomial_map}
Fix natural numbers $n,\ell$ and let $w$ be a permutation in $S_n$. Let $R = \mathbb K[P_I : I \subseteq [n], |I| \in \{1, \dots, n-1\}, I \notin S_w]$ and $S = \mathbb K[x_{i,j} : i \in \{1, \dots, n-1 \}, j \in \{1, \dots, n \}]$ be polynomial rings. We define the map $\phi_{\ell,w}: R \rightarrow S$ to be the restriction of the monomial map $\phi_\ell$ defined in \eqref{eqn:monomialmap} to the ring $R$.
\end{definition}

\noindent\textbf{Notation.} Fix $n, \ell$ natural numbers and $w$ a permutation. We use the following shorthand notation for ideals of $R$ throughout this section.
\begin{itemize}
  \item $J_1 := F_{n,\ell,w}$, the restricted matching field ideal defined in \eqref{eq:F_n,l,w}.
  \item $J_2 := {\rm in}_{{\bf w}_\ell}(I(X(w)))$, the initial ideal of the ideal of the Schubert variety.
  \item $J_3 := \ker(\phi_{\ell,w})$, the kernel of the restricted monomial map. 
\end{itemize}

\subsection{Generating sets}\label{sec:mon_bases_gen_set}
By studying the generating sets of $J_1$ and $J_3$, we will show that they coincide if and only if $J_1$ is monomial-free and $J_1 \subseteq J_2$. To prove the remaining containments, we will consider monomial bases in the subsequent subsection. Recall that the matching field ideal $F_{n,\ell}$ is quadratically generated and is the kernel of a monomial map, see Corollary~\ref{cor:block_diag_degen_flag}. Starting with a quadratic generating set for $F_{n,\ell}$, we explicitly construct a generating set for $F_{n,\ell,w}$.

\begin{definition}
Let $G \subset \mathbb{K}[x_1, \dots, x_n]$ be a collection of homogeneous quadratic polynomials and $S \subseteq \{x_1, \dots, x_n \}$ be a collection of variables. We identify $S$ with its characteristic vector, i.e. $S_i = 1$ if $x_i \in S$ otherwise $S_i = 0$. For each $g \in G$ we write $g = \sum_{\alpha} c_\alpha x^\alpha$ and define
\[
\hat g = \sum_{S \cdot\alpha = 0} c_\alpha x^\alpha.
\]
We define $G_S = \{\hat g : g \in G \}$ to be the collection of all such polynomials.
\end{definition}

By definition, we have $F_{n,\ell,w} = (F_{n,\ell} + \langle S \rangle) \cap R$ where $R$ is the ring given in Definition~\ref{def:restricted_monomial_map} and $S = \{P_I : I \in S_w \}$ is the set of variables that vanish in $F_{n,\ell,w}$. We show that if $G$ is a quadratic generating set for $F_{n,\ell}$ then $G_S$ is a quadratic generating set for $F_{n,\ell,w}$.

\begin{lemma}\label{lem:elim_ideal_gen_set}
Let $G \subseteq \mathbb{K}[x_1, \dots, x_n]$ be a set of quadratic polynomials and $S \subseteq \{x_1, \dots, x_n \}$ a subset of variables. Then $\langle G_S \rangle = \langle G \cup S \rangle \cap \mathbb{K}[\{x_1, \dots, x_n\} \backslash S]$.
\end{lemma}

\begin{proof}
To show that $G_S \cup S$ and $G \cup S$ generate the same ideal, for each $g \in G$ we write $g = \sum_\alpha c_\alpha x^\alpha$, for some $c_\alpha \in \mathbb{K}$. We have that
\[
g - \hat g = \sum_{S\cdot \alpha \ge 1} c_\alpha x^\alpha.
\]
Each term appearing in the above sum is divisible by some variable in $S$, hence $\hat g \in \langle G \cup S \rangle$ and $g \in \langle G_S \cup S \rangle$.
For any polynomial $ f \in \langle G \cup S \rangle \cap \mathbb{K}[\{x_1, \dots, x_n\} \backslash S]$ we have that
$
f = \sum_{g \in g} c_g h_g g + \sum_{x_i \in S} c_i h_i x_i
$
for some $c_g, c_i \in \mathbb{K}$  and $h_g, h_i \in \mathbb{K}[x_1, \dots, x_n]$. For each $h_g$ we define $\hat h_g$ similarly to $\hat g$ and rewrite this polynomial as
\[
f = \sum_{g \in G} c_g \hat h_g \hat g + \left(\sum_{g \in G} c_g (h_g g - \hat h_g \hat g) + \sum_{x_i \in S} c_i p_i x_i
\right).
\]
All monomials appearing in $\sum_{g \in G} c_g \hat h_g \hat g$ are not divisible by any monomials that lie in $S$. However each monomial appearing in the expressions
$\sum_{g \in G} c_g (h_g g - \hat h_g \hat g)$ and $\sum_{x_i \in S} c_i p_i x_i$ is divisible by some $x_i \in S$. Since $f \in \mathbb{K}[\{x_1, \dots, x_n\} \backslash S]$ it follows that the large bracketed expression above is zero and so $f = \sum_{g \in G} c_g \hat h_g \hat g \in \langle G_S \rangle$.
\end{proof}

\noindent Using the generating sets of $F_{n,\ell,w}$ constructed above, we now consider the ideals $J_1, J_2$ and~$J_3$.

\begin{lemma}\label{lem:J_1=J_3_binomial}
The ideals $J_1$ and $J_3$ coincide if and only if $J_1$ is monomial-free.
\end{lemma}

\begin{proof}
Note that $J_3$ is the kernel of a monomial map that does not send any variables to zero. Therefore $J_3$ does not contain any monomials. If $J_1$ contains a monomial, then $J_1\neq J_3$.

For the converse, suppose $J_1$ does not contain any monomials. Let $G$ be a quadratic generating set for $F_{n,\ell}$ and let $S = \{ P_I : I \in S_w\}$ be the collection of variables that vanish in $F_{n,\ell,w}$. By definition $J_1 = \langle G \cup S \rangle \cap \mathbb K[P_I : I \notin S_w]$. So by Lemma~\ref{lem:elim_ideal_gen_set} we have $J_1$ is generated by $G_S$. Since $J_1$ is monomial-free, we have that $G_S$ does not contain any monomials. By Corollary~\ref{cor:block_diag_degen_flag}, the ideal $F_{n,\ell}$ is the kernel of the monomial map $\phi_\ell$ and by definition $J_3$ is the kernel of the restriction $\phi_{\ell,w}$. Since all binomials $m_1 - m_2 \in G_S$ lie in $F_{n,\ell}$ and contain only the non-vanishing Pl\"ucker variables $P_J$ for $J \notin S_w$, therefore $m_1 - m_2 \in J_3$. And so we have $J_1 \subseteq J_3$. Also, for any polynomial $f \in J_3$ we have that $f \in F_{n,\ell}$. Since $f$ contains only the non-vanishing Pl\"ucker variables, therefore $f \in J_1$.
\end{proof}

\begin{lemma}\label{lem:J_1_subset_J_2}
$J_1 \subseteq J_2$.
\end{lemma}

\begin{proof}
Let $G$ be a quadratic binomial generating set for $F_{n,\ell}$ and $S = \{P_I : I \in S_w^v \}$. Let $\hat f \in G_S \subset J_1$ be any polynomial. By the definition of $G_S$, there exists $f \in G$ such that $\hat f$ is obtained from $f$ by setting some variables to zero. Recall $F_{n,\ell} = \textrm{in}_{{\bf w}_\ell}(F_n)$, so there exists a polynomial $g \in F_n$ such that $f = \textrm{in}_{{\bf w}_\ell}(g)$. Since the leading term of $g$ is not set to zero in $I(X(w))$, it follows that $\hat f \in \textrm{in}_{{\bf w}_\ell}(I(X(w)))$.
\end{proof}

\subsection{Monomial bases for Schubert varieties}\label{sec:mon_bases_schubert}

We begin by recalling a description of a collection of standard monomials for the Schubert variety $X(w)$.

\begin{definition}[Definition V.5. in \cite{kim2015richardson}]
Let $T$ be a semi-standard Young tableau with columns $I_1, \dots, I_t$. Let $\textbf{w} = (w_1, \dots, w_t)$ be a sequence of permutations and write $w_k = (w_{k,1}, \dots, w_{k,n}) \in S_n$ for each $k \in [t]$. We say that $\textbf{w}$ is a \emph{defining chain} for $T$ if the permutations are monotonically increasing $w_1 \le w_2 \le \dots \le w_t$ with respect to the Bruhat order and for each $k \in [t]$ we have $I_k = \{w_{k,1}, \dots, w_{k, |I_k|} \}$.
\end{definition}

Let $\textbf{w} = (w_1, \dots, w_t)$ and $\textbf{v} = (v_1, \dots, v_t)$ be two defining chains for a fixed semi-standard Young tableau. There is a natural partial order on the set of defining chains given by $\textbf{v} \le \textbf{w}$ if $v_k \le w_k$ for all $k \in [t]$. It turns out there exists a unique minimum defining chain. The standard monomials for Schubert varieties can be determined by the minimum defining chain. In the following theorem we summarise these results.

\begin{theorem}[Lemma V.9, Proposition V.13 and Theorem V.14 in \cite{kim2015richardson}]
Let $w$ be a permutation. The collection of monomials corresponding to tableau $T$ with $d$ columns such that $w^-_d \le w$ forms a monomial basis for $X(w)$, where $\textbf{w}_-(T) = (w^-_1, \dots, w^-_d)$ is the unique minimum defining chain for $T$.
\end{theorem}

If the tableau is not clear from the context, we write $w_i^-(T)$ for $w_i^-$, where $i \in \{1, \dots, d \}$. In this section we will show that the following.

\begin{theorem}\label{thm:w_312_ssyt}
If $w$ is $312$-free then the semi-standard Young tableaux whose columns $I$ satisfy $I \le w$ form a monomial basis for the Schubert variety $X(w)$.
\end{theorem}

\begin{proof}
We show that each semi-standard Young tableau $T$ is standard for the Schubert variety $X(w)$. Note that the ideal of the Schubert variety $X(w)$ is generated in degree two so it suffices to check all tableaux with at most two columns. If $T$ has a single column $I$ then the minimum defining sequence for $T$ has a single permutation $v$ which is the Grassmannian permutation defined by $I$. Hence $v \le w$ and $T$ is standard for $X(w)$. By Lemma~\ref{lem:w_312_ssyt_2_cols}, all semi-standard Young tableaux with two columns are standard for $X(w)$ and we are done.
\end{proof}

We prove Theorem~\ref{thm:w_312_ssyt} in two steps. We begin with tableaux that have exactly one column of size one. We then use this to show the general case.

\medskip

\noindent\textbf{Notation.} Let $\mathcal Q = (Q_1, Q_2, \dots, Q_k )$ be a partition of $[n]$ where each $Q_i \subseteq [n]$ is non-empty and disjoint. Write $Q_i = \{Q_{i,1} < q_{i,2} < \dots < q_{i, |Q_i|}\}$ for each $i \in [k]$. We define the permutation 
\[
(Q_1, Q_2, \dots, Q_k) = (q_{1,1}, q_{1,2}, \dots, q_{1, |Q_1|}, q_{2,1}, \dots, q_{2, |Q_2|}, q_{3,1}, \dots, q_{k, |Q_k|}).
\]
In particular if $I \subseteq [n]$ is a subset then $(I, [n] \backslash I)$ is the Grassmannian permutation defined by $I$.

\begin{lemma}\label{lem:v_2_structure}
Let $T$ be a semi-standard Young tableau with two columns $I = \{i_1 < \dots < i_t \}$ and $J = \{j_1\}$. The minimum defining sequence for $T$ is $(v_1, v_2)$ where
$v_1 = (I, [n] \backslash I)$, $v_2 = (j_1, I \backslash i_s, ([n] \backslash (I \cup j_1)) \cup i_s)$ and $s = \max\{k \in [t] : i_k \le j_1 \}$.
\end{lemma}

\begin{proof}
Let $(w_1, w_2)$ be a defining sequence for $T$. By definition of defining sequence we have $w_1 \ge v_1 = (I, [n] \backslash I)$. Since $v_1 \le w_1 \le w_2$, therefore the smallest possible permutation for $w_2$ with respect to the Bruhat order is $v_2$.
\end{proof}

\begin{lemma}\label{lem:3_part_perms}
Let $Q_1, Q_2, Q_3$ be a partition of $[n]$. Let $w$ be a permutation such that $(Q_1, Q_2, Q_3) \nleq w$. If $Q_1 \le w$ then $Q_1 \cup Q_2 \nleq w$. 
\end{lemma}

\begin{proof}
If $Q_1 \le w$ and $Q_1 \cup P_2 \le w$ then it follows that $(Q_1, Q_2, Q_3) \le w$.
\end{proof}

\begin{lemma}\label{lem:min_defining_seq_structure}
Let $T$ be a semi-standard Young tableau with two columns $I$ and $J$. Let $(v_1, v_2)$ be the minimum defining sequence for $T$. We have $v_1 = (I, [n] \backslash I)$ and $v_2 = (J, \tilde I, [n] \backslash (J \cup \tilde I) )$ for some subset $\tilde I \subseteq I$.
\end{lemma}

\begin{proof}
It is clear that $(I, [n] \backslash I)$ is the smallest permutation $u = (u_1, \dots, u_n)$ such that $\{u_1, \dots, u_{|I|} \} = I$. And so $v_2$ is the smallest permutation $u = (u_1, \dots, u_n)$ such that $v_1 \le u$ and $\{ u_1, \dots, u_{|J|}\} = J$. It follows easily that the smallest such permutation has the form $(J, \tilde I, [n] \backslash (J \cup \tilde I))$ for some $\tilde I \subseteq I$.
\end{proof}

\begin{lemma}\label{lem:w_312_ssyt_2111_shape}
Suppose $T$ is a semi-standard Young tableau with columns $I$ and $J$ where $|J| = 1$. If $w$ is $312$-free and $I, J \le w$ then $T$ is standard for $X(w)$.
\end{lemma}

\begin{proof}
Write $I = \{i_1 < \dots < i_t \}$ and $J = \{ j_1 \}$ for the columns of $T$ and write $w = (w_1, \dots, w_n)$ for the permutation. We define $(\widehat w_1, \dots, \widehat w_t) = \{w_1, \dots, w_t \}^{\uparrow}$. By assumption we have $i_k \le \widehat w_k$ for each $k \in [t]$ and $j_1 \le w_1$. Let $w^- = (v_1, v_2)$ be the minimum defining sequence for $T$. By Lemma~\ref{lem:v_2_structure} we have that $v_2 = (j_1, I \backslash i_s, ([n] \backslash I) \cup i_s)$ where $s = \max\{k \in [t] : i_s \le j_1 \}$. If $s = t$ then $v_2 = ((I \backslash i_t) \cup j_1, [n] \backslash ((I \backslash i_t) \cup j_1)) \le w$ and we are done. Suppose that $s < t$ and, in this case, we assume by contradiction that $v_2 \nleq w$. By Lemma~\ref{lem:3_part_perms}, we have that $(I \backslash i_s) \cup j_1 = \{i_1 < \dots i_{s-1} < j_1 < i_{s+1} < \dots < i_t \} \nleq w$. Since $i_k \le \widehat w_k$ for each $k \in [t]$, it follows that $j_1 > \widehat w_s$. Since $s < t$, therefore $\widehat w_s < j_1 < i_{s+1} \le \widehat w_{s+1}$. And so there exist unique values $p \in \{2, \dots, t \}$ and $q \in \{t+1, \dots, n \}$ such that $w_p = \widehat w_s$ and $w_q = j_1$. And so $w_1, w_p, w_q$ has type $312$ a contradiction.
\end{proof}

\begin{lemma}\label{lem:w_312_ssyt_2_cols}
Suppose $T$ is a semi-standard Young tableau with two columns $I$ and $J$. If $w$ is $312$-free and $I, J \le w$ then $T$ is standard for $X(w)$.
\end{lemma}

\begin{proof}
Write $I = \{i_1 < \dots < i_t \}$ and $J = \{j_1 < \dots < j_s \}$. We assume that $I$ is the leftmost column of $T$ and so $s \le t$. Let $w^- = (v_1, v_2)$ be the minimum defining sequence for $T$. If $s = t$ then we have that $v_1 = (I, [n] \backslash I)$ and $v_2 = (J, [n] \backslash J)$. In particular we have $v_2 \le w$ and so $T$ is standard for $X(w)$. So from now on, we assume that $s < t$. 

We proceed by induction on $s$. Note that if $s = 1$ then we are done by Lemma~\ref{lem:w_312_ssyt_2111_shape}. So assume that $s > 1$. Without loss of generality we may assume that $w_1 < \dots < w_s$. Let us assume by contradiction that $v_2 \nleq w$. We define $T'$ to be the tableau with columns $I = \{i_1 < \dots < i_t \}$ and $J' = \{j_1 < \dots < j_{s-1}\}$. We have $J' \le w$ and so by induction $T'$ is a standard tableau for $X(w)$. Let $(v_1', v_2')$ be the minimum defining sequence for $T'$. By Lemma~\ref{lem:min_defining_seq_structure} we have $v_2' = (J', \tilde I', [n] \backslash (J' \cup \tilde I'))$ for some subset $\tilde I' \subseteq I$. Write $\tilde I' = \{r_1 < r_2 < \dots < r_{t-s+1} \}$. Let $p = \max\{k \in [t-s+1] : r_k \le j_s \}$. By Lemma~\ref{lem:min_defining_seq_structure} we have $v_2 = (J, \tilde I, [n] \backslash (J \cup \tilde I))$ for some $\tilde I \subseteq I$. It is easy to show that $\tilde I = \tilde I' \backslash r_p$. By assumption $J \le w$ and $v_2 \nleq w$ and so by Lemma~\ref{lem:3_part_perms} we have $J \cup \tilde I \nleq w$. Since $v' \le w$, we have $J' \cup \tilde I' \le w$. Write
\begin{align*}
    J' \cup \tilde I' &= 
        \{u_1 < \dots < u_{q-1} < u_q = r_p < u_{q+1} < \dots < u_t \}, \\
    J \cup \tilde I &=
        \{u_1 < \dots < u_{q-1} < j_s < u_{q+1} < \dots < u_t \}, \\
    \{w_1, \dots, w_s \}^{\uparrow} &=
        \{\widetilde w_1 < \dots < \widetilde w_s \} \textrm{ and } 
    \{w_1, \dots, w_t \}^{\uparrow} =
        \{\widehat w_1 < \dots < \widehat w_t\}.
\end{align*}
Since $J \le w$ we have that $j_s \le \widetilde w_s$. Since $J' \cup \tilde I' \le w$ we have that $u_k \le \widehat w_k$ for all $k \in [t]$. Since $J \cup \tilde I \nleq w$ we must have that $\widehat w_q < j_s$. Since $j_s \le \widetilde w_s$ therefore $q < t$ and so $\widehat w_q < j_s < u_{q+1} \le \widehat w_{q+1}$. If $w_k \ge j_s$ for all $k \in \{s+1, \dots, t\}$ then $|\{k \in [t] : w_k < j_s \}| = |\{k \in [s] : w_k < j_s  \}| < s$ because $w_s \ge j_s$. However we have $j_1 < \dots < j_s$ and so $q \ge s$. Since $\widehat w_1 < \dots < \widehat w_q < j_s$ so $|\{k \in [t] : w_k < j_s \}| \ge s$, a contradiction. Therefore there exists $a \in \{s+1, \dots, t\}$ such that $w_a < j_s$. Since $\widehat w_q < j_s < u_{q+1} \le \widehat w_{q+1}$, there exists $b \in \{t+1, \dots, n\}$ such that $w_b = j_s$. Note that $j_s < \widehat w_{q+1} \le w_s$. And so $w_s, w_a, w_b$ is a subsequence of type $312$ in $w$, a contradiction.

\end{proof}

\subsection{Monomial bases for matching field ideals}\label{sec:mon_bases_main}

In this section we prove Theorem~A by considering a collection of standard monomials for the ideals $F_{n,\ell,w}$ and $\inwb(I(X(w)))$. 
We begin by stating the proof that relies on Lemma~\ref{lem:SSYT_flag_biject_Q_perms}, which we show following a proof of Theorem~\ref{thm:P_ell=T_n+Z_n}. 

\begin{proof}[Proof of Theorem~A]
Suppose that $J_1$ is monomial free. We will show that $J_1 = J_2 = J_3$ and in particular we have that $J_2 = J_3$, hence the initial ideal of the Schubert variety is toric.
By Lemmas~\ref{lem:J_1=J_3_binomial} and \ref{lem:J_1_subset_J_2} we have $J_1 = J_3 \subseteq J_2$. So for all $d \ge 1$, any collection of standard monomials for $J_2$ of degree $d$ is linearly independent in $R / J_3$. Since $J_2 = \textrm{in}_{{\bf w}_\ell}(I(X(w)))$ is an initial ideal of a homogeneous ideal, the number of standard monomials of degree $d$ coincides with the number of standard monomials of degree $d$ of $I(X(w))$. By Theorem~\ref{thm:w_312_ssyt}, the semi-standard Young tableaux with $d$-columns, such that each column $I$ satisfies $I \le w$, are in canonical bijection with a collection of standard monomials of $I(X(w))$ of degree $d$.

Suppose $\ell = 0$. We have that two monomials are equal in $R/J_3$ if and only if their corresponding tableaux are row-wise equal. Therefore, the semi-standard Young tableaux are in bijection with standard monomials for $J_3$. And so we have $J_1 = J_2 = J_3$.

Suppose that $J_2$ is generated in degree two. By Lemma~\ref{lem:SSYT_flag_biject_Q_perms} we have that $J_1$ and $J_2$ have the same number of standard monomials in degree two. This, together with the fact that $J_1$ and $J_2$ are both generated in degree two and $J_1 \subseteq J_2$, implies that $J_1 = J_2$.
\end{proof}

\medskip

We now give an alternative description of the permutations $w \in T_{n,\ell} \cup Z_n$ such that $J_1$ is monomial-free. We will write this set $P_\ell$ and prove that $P_\ell = T_{n,\ell} \cup Z_n$. In the proofs to follow, we write $w \in P_\ell$ to indicate that $w$ satisfies Definition~\ref{def:perms_P_ell}.

\begin{remark}
Recall the definition of the collection of permutations $P_\ell$. Most permutations in this set are $312$-free. If $w \in S_n$ is $312$-free then $w$ has the descending property. Moreover all restrictions of $w$ also have the descending property. It is easy to show that a permutation $w$ is $312$-free if and only if all restrictions of $w$ have the descending property. 
\end{remark}

\begin{remark}
The set $P_n$ is defined to be the collection of $312$-free permutations. It is straightforward to show that $w \in S_n$ is $312$-free if and only if it satisfies the two bullet pointed conditions in Definition~\ref{def:perms_P_ell} where $\ell = n$. \end{remark}

\begin{proof}[Proof of Theorem~\ref{thm:P_ell=T_n+Z_n}]
Throughout this proof we use Theorem~C which shows the following.
\begin{itemize}
    \item $T_{n,n} = A_1 \cup A_2$,
    \item $T_{n,n-1} = A_1 \cup \tilde A_1 \cup \tilde A_2$,
    \item $T_{n,\ell} = A_1 \cup A_2' \cup A_3 \cup \{(n, \ell, n-1, \dots, \ell+1, \ell-1, \dots, 1)\}$.
\end{itemize}
We show that $P_\ell = T_{n,\ell} \cup Z_n$. We will write $P_\ell^n \subseteq S_n$ for the set $P_\ell$ to distinguish $n$. Given a permutation $w \in S_n$ we write $w = (w_1, \dots, w_n)$ and write $\underline w = (\underline w_1, \dots, \underline w_{n-1})$ for the restriction of $w$ to $[n-1]$.

We show $P_\ell^n = T_{n,\ell} \cup Z_n$ by induction on $n$. For the base case we consider $n = 3$. The permutations in $P_\ell^3$, $T_{3,\ell}$ and $Z_3$ are shown below.
\begin{center}
    \begin{tabular}{|c|c|c|c|}
    \hline
       $\ell$ & \multicolumn{1}{c|}{$P_\ell^3$}  &  \multicolumn{1}{c|}{$T_{3,\ell}$}   &  \multicolumn{1}{c|}{$Z_{3}$} \\
    \hline
        $0$ & $123, 132, 213, 231, 321$ 
        & $231, 321$ & $123, 132, 213$ \\
        $1$ & $123, 132, 213, 312, 321$
        & $312, 321$ & $123, 132, 213$ \\
        $2$ & $123, 132, 213, 321$
        & $321$ & $123, 132, 213$ \\
    \hline
    \end{tabular}
\end{center}
So we have the base case $P_\ell^3 = T_{3,\ell} \cup Z_3$.

\smallskip

Let $n \ge 4$ and assume that $P_\ell^{n-1} = T_{n-1, \ell} \cup Z_{n-1}$. We take cases on values of $\ell$.

\smallskip

\textbf{Case 1.} Assume $\ell = n$. Let $w \in P_n^{n}$. It follows that $\underline w$ is $312$-free and so $\underline w \in P_{n-1}^{n-1} = T_{n-1, n-1} \cup Z_{n-1}$. Note that $w$ is $312$-free and so has the descending property. 
\begin{itemize}
    \item If $\underline w \in Z_{n-1}$ then $w \in \underline Z_n$. By definition of $Z_n$ we have that either $\underline w = (\dots, n-1)$ or $\underline w = (\dots, n-1, n-2)$. Note that $w$ has the descending property. If $\underline w = (\dots, n-1)$ then $w = (\dots, n-1, n)$ or $w = (\dots, n, n-1)$ and so $w \in Z_n$. If $\underline w = (\dots, n-1, n-2)$ then either $w = (\dots, n-1, n-2, n)$ in which case $w \in Z_n$, or $w = (\dots, n-1, n, n-2)$ or $(\dots, n, n-1, n-2)$ in which case $w \in A_1$ and so $w \in T_{n,n}$.
    
    \item If $\underline w \in T_{n-1,n-1}$ then $w \in \underline T_{n,n}$. Since $w$ is $312$-free, $w$ has the descending property and so $w \in A_2$ hence $w \in T_{n,n}$.
\end{itemize}

\smallskip

Conversely take $w \in T_{n,n} \cup Z_n$. If $w \in Z_n$ then it is easy to check that $w$ is $312$-free and so $w \in P_n^n$. Suppose $w \in T_{n,n} = A_1 \cup A_2$. 

\begin{itemize}
    \item If $w \in A_1$ then we have $w \in \underline Z_n$ and so $\underline w$ is $312$-free since $\underline w \in Z_{n-1}$. So it suffices to show that $w$ contains no $312$-type subsets of the form $w_i, w_j, w_k$ where $i < j < k$ and $w_i = n$. However by construction, if $w \in A_1$ then either $w = (\dots, n, n-1, n-2)$ or $w = (\dots, n-1, n, n-2)$. And so $w$ is $312$-free.
    
    \item If $w \in A_2$ then we have that $w \in \underline T_{n,n}$ and so by induction $\underline w \in T_{n-1, n-1} \subseteq P_{n-1}^{n-1}$ is $312$-free. Let $t, s \in \{1, \dots, n \}$ such that $w_s = n-1$ and $w_t = n$. Since $\underline w$ is $312$-free, it suffices to show that $w$ contains no $312$-type subsets of the form $w_t, w_j, w_k$ where $t < j < k$. Suppose there exists such a subset. Since $w \in A_2$ we have $t \ge s-1$ and so $w_s, w_j, w_k$ is also of type $312$ but lies in $\underline w$, a contradiction. Therefore so $w$ is $312$-free.
\end{itemize}

And so we have shown that $w \in P_n^n$.

\smallskip

\textbf{Case 2.} Assume $1 \le \ell < n-1$. Let $w \in P_\ell^n$. Suppose that $w$ is $312$-free. Then $\underline w$ is $312$-free and if $\underline w|_m = (m-1,m,m-2, \dots, 1)$ for some $3 \le m \le n$ then we have $\underline w|_m = w|_m$ so $w_1 < w_2 \le \ell < n$ so $\underline w_1 = w_1 < w_2 = \underline w_2 \le \ell$. Hence $\underline w \in P_\ell^{n-1} = T_{n-1, \ell} \cup Z_n$.
\begin{itemize}
    \item If $\underline w \in Z_{n-1}$ then, similarly to the diagonal case, we have $w \in Z_n \cup A_1 \subseteq T_{n,\ell} \cup Z_n$.
    \item If $\underline w \in T_{n-1, \ell}$ then $w \in \underline T_{n,\ell}$. Since $w$ is $312$-free, $w$ has the descending property and so $w \in A_2$. By definition we have $(n-1, n, n-2, \dots, 1) \notin P_\ell^n$ since $\ell < n$. So $w \in A_2' \subseteq T_{n,\ell}$.
\end{itemize}

Suppose $w$ is not $312$-free. By definition we have $w_1 > w_2 = \ell$. 

\begin{itemize}
    \item If $w_1 = n$ then by Lemma~\ref{lem:Q_perm_restrict_le_m} we have $w = (n, \ell, n-1, \dots, \ell+1, \ell-1, \dots, 1)$ and so $w \in T_{n,\ell}$.
    \item If $w_1 = n-1$ then we similarly we have that $\underline w = (n-1, \ell, n-2, \dots, \ell+1, \ell-1, \dots, 1) \in T_{n-1, \ell}$. Let $t \in \{1, \dots, n \}$ such that $w_t = n$. Since $w\backslash\ell$ is $312$-free it follows that $t \ge 3$. Hence $w \in A_3 \subseteq T_{n,\ell}$. 
    \item If $w_1 < n-1$ then since $w\backslash \ell$ is $312$-free it follows that $\underline w$ has the descending property. Let $s, t \in \{ 1, \dots, n\}$ such that $w_s = n-1$ and $w_t = n$. If $t < s-1$ then let $i \in \{t+1, \dots, s-1 \}$. We have $w_t, w_i, w_s$ is of type $312$ that lies in $w \backslash \ell$, a contradiction. Therefore $t \ge s-1$ and so $w \in A_2$ hence $w \in T_{n, \ell}$. 
\end{itemize} 
And so we have shown that $w \in T_{n,\ell} \cup Z_n$.

\smallskip

Conversely let $w \in T_{n,\ell} \cup Z_n$. If $w \in Z_n \cup A_1$ then it is straightforward to check that $w$ is $312$-free and for all $3 \le m \le n$ we have $w|_m \neq (m-1, m, m-2, \dots, 1)$. Also note that $(n, \ell, n-1, \dots, \ell+1, \ell-1, \dots, 1) \in P_\ell^n$. So let $w \in T_{n,\ell} = A_2' \cup A_3$.

\begin{itemize}
    \item If $w \in A_2'$ then we have $\underline w \in T_{n-1,\ell} \subseteq P_\ell^{n-1}$. Suppose $\underline w$ is $312$-free. Since $w$ has the descending property we have $w$ is also $312$-free. If, for some $ 3 \le m \le n$, we have $w|_m = (m-1, m, m-2, \dots, 1)$ then $(n-1, n, n-2, \dots, 1) \notin A_2'$ hence $m < n$. Therefore $w|_m = \underline w|_m$. Since $\underline w \in P_\ell^{n-1}$, we have $\underline w_1 = w_1 < \underline w_2 = w_2 \le \ell$. Since $w_2 \le \ell < n-1$ we have $\underline w|_{\underline w_2} = w|_{w_2}$. And so $w \in P_\ell^n$.
    
    Suppose $\underline w$ is not $312$-free then by definition of $P_\ell^{n-1}$ we have $\underline w_1 > \underline w_2 = \ell$. Since $w \in A_2'$ we have that $\underline w$ has the descending property. If $\underline w_1 = n-1$ then by the descending property it follows that $\underline w = (n-1, n-2, \dots, 1)$, which is $312$-free, a contradiction. So $\underline w_1 < n-1$. Let $s, t \in \{1, \dots, n \}$ such that $w_s = n-1$ and $w_t = n$. Since $\underline w_1, \underline w_2 < n-1$ we have $s \ge 3$. Since $P_\ell^{n-1}$, we have $\underline w\backslash\ell$ is $312$-free. By definition of $A_2'$ we have $t \ge s-1$ and so $w \backslash \ell$ is $312$-free, otherwise if $w_t, w_i, w_j$ is of type $312$ then so is $w_s w_i w_j$. Hence $w \in P_\ell^n$.
    
    \item If $w \in A_3$ then we have $\underline w \in T_{n-1, \ell} \subseteq P_\ell^n$. Also $\underline w$ does not have the descending property and so $\underline w$ is not $312$-free. By definition of $P_\ell^{n-1}$ we have $\underline w_1 > \underline w_2 = \ell$ and $\underline w \backslash \ell$ is $312$-free. If $\underline w_1 \neq n-1$ then, since $\underline w \backslash \ell$ is $312$-free and $\underline w_2 = \ell$, it follows that $\underline w$ has the descending property, a contradiction. Therefore $\underline w_1 = n-1$. By Lemma~\ref{lem:Q_perm_restrict_le_m} we have $\underline w = (n-1, \ell, n-2, \dots, \ell+1, \ell-1, \dots, 1)$. Let $w_t = n$ for some $t$. By definition of $A_3$ we have $t \ge 3$ hence $w_1 = \underline w_1 $ and $w_2 = \underline w_2 = \ell$. And so have shown that $w$ is not $312$-free and $w_1 > w_2 = \ell$. We have $\underline w \backslash \ell = (n-1, n-2, \dots, 1)$ and so for any $t \ge 3$, it follows that $w \backslash \ell$ is $312$-free. Hence $w \in P_\ell^n$. 
\end{itemize}
And so we have shown that $w \in P_\ell^n$.

\smallskip

\textbf{Case 3.} Assume $\ell = n-1$. Let $w \in P_{n-1}^n$. Suppose that $w$ is not $312$-free. Then we have $w_1 > w_2 = \ell = n-1$. Therefore $w_1 = n$. Note that $w \backslash \ell$ is $312$-free and so $w \backslash \ell$ has the descending property. Therefore $w\backslash\ell = (n, n-2, \dots, 1)$ and so $w = (n, n-1, \dots, 1)$. However $w$ is $312$-free a contradiction. Therefore $w$ is $312$-free. It follows that $\underline w$ is $312$-free and so $\underline w \in P_{n-1}^{n-1} = T_{n-1,n-1} \cup Z_{n-1}$.

\begin{itemize}
    \item If $\underline w \in Z_{n-1}$ then, similarly to the diagonal case, we have $w \in Z_n \cup A_1 \subseteq T_{n, n-1}\cup Z_n$
    \item If $\underline w \in T_{n-1, n-1}$ then $w \in \underline T_{n,n-1}$. Since $w$ is $312$-free, $\underline w$ has the descending property. Also $w$ has the descending property so if $w_s = n-1$ and $w_t = n$ then  $t \ge s-1$, otherwise there exists $i \in \{t+1,\dots, s-1 \}$ and so $w_t, w_i, w_s$ has type $312$. Hence $w \in A_2$. By definition it follows that $(n-1, n, n-2, \dots, 1) \notin P_{n-1}^n$ and so $w \in A_2'$.
    
    Suppose that $w \in A_2' \backslash \underline T_{n, n-2}$. We will show that $w \in \tilde A_2$ as follows. By assumption we have $\underline w \in (T_{n-1, n-1} \backslash T_{n-1, n-2})$. By induction we have $P_{n-1}^{n-1} = T_{n-1,n-1} \cup Z_{n-1}$ and $P_{n-2}^{n-1} = T_{n-1, n-2} \cup Z_{n-1}$. And so we have $\underline w \in P_{n-1}^{n-1} \backslash P_{n-2}^{n-1}$. By assumption $\underline w$ is $312$-free, so suppose $\underline w|_m = (m-1, m, m-2, \dots, 1)$ for some $3 \le m \le n-1$. Since $\underline w \notin P_{n-2}^{n-1}$ it follows that $m = n-1$ and $\underline w = (n-2, n-1, n-3,\dots, 1)$. Let $w_t = n$ for some $t$. Since $w$ is $312$-free it follows that $t \ge 2$, otherwise if $t = 1$ we have $w_1, w_2, w_3$ has type $312$. If $t = 2$ the we have $w = (n-2, n, n-1, n-3, \dots, 1) \notin P_{n-1}^n$. So we must have $t \ge 3$. Hence $w \in \tilde A_2$.
\end{itemize}

And so we have shown that $w \in Z_n \cup A_1 \cup \tilde A_1 \cup \tilde A_2 = T_{n,n-1} \cup Z_n$.

\smallskip

Conversely let $w \in T_{n, n-1} \cup Z_n$. If $w \in Z_n \cup A_1$ then it is straightforward to check that $w$ is $312$-free and for all $3 \le m \le n$ we have $w|_m \neq (m-1, m, m-2, \dots, 1)$. Let $w \in \tilde A_1 \cup \tilde A_2$.

\begin{itemize}
    \item If $w \in \tilde A_1$ then $\underline w \in T_{n-1,n-1} \subseteq P_{n-1}^{n-1}$ so $\underline w$ is $312$-free. Since $w \in A_2$ we have that $w$ has the descending property and so $w$ is also $312$-free. Suppose that $w|_m = (m-1, m, m-2, \dots, 1)$ for some $3 \le m \le n$. Since $w \in A_2'$ we have $w \neq (n-1, n, n-2, \dots, 1)$ so $m \le n-1$. Therefore $w|_m = \underline w|_m$. Since $\underline w \in P_{n-2}^{n-1}$ it follows that $\underline w_1 < \underline w_2 \le n-2$ and $\underline w|_{\underline w_2} = (\underline w_1, \underline w_2, \underline w_1 - 1, \dots, \underline w_2 +1, \underline w_2 - 1, \dots, 1)$. Let $w_t = n$ for some $t$. Since $\underline w_2 \le n-2$ we must have $t \ge 3$. And so $w_1 = \underline w_1$, $w_2 = \underline w_2$ and $w|_{w_2} = \underline w|_{\underline w_2}$. And so we have shown that $w \in P_{n-1}^n$.
    
    \item If $w \in \tilde A_2$ then $\underline w \in T_{n-1, n-1} \backslash T_{n-1, n-2} = P_{n-1}^{n-1} \backslash P_{n-2}^{n-1}$ so $\underline w$ is $312$-free. Suppose $\underline w|_m = (m-1, m, m-2, \dots, 1)$ for some $3 \le m \le n-1$. Since $\underline w \notin P_{n-2}^{n-1}$ it follows that $m = n-1$ and $\underline w = (n-2, n-1, n-3,\dots, 1)$. Let $w_t = n$ for some $t$. Since $w \in \tilde A_2$, we have that $t \ge 2$. And so $w_1 = \underline w_1 = n-2$ and $w_2 = \underline w_2 = n-1$. And so we have $w \in P_{n-1}^n$.
\end{itemize}
\end{proof}

Using the particularly nice description of  
$P_\ell$ we can now prove the following for $w\in P_\ell$.

\begin{lemma}\label{lem:SSYT_flag_biject_Q_perms}
The number of standard monomials for $J_1$ in degree two is $|SSYT_2(w)|$.
\end{lemma}

To prove this, we construct a bijection between the semi-standard Young tableaux $SSYT_d(w)$ to some matching field tableaux whose image forms a monomial basis for $J_1$.

\begin{definition}
We define $\Gamma_\ell$ be the map taking the SSYT tableau $T$, with columns $I, J$, often written $T = [IJ]$, to the matching field tableau $T' = [I'J']$, for $B_\ell$, by the following cases. Write $I = \{i_1 < i_2 < \dots < i_t \}$ and $J = \{j_1 < j_2 < \dots < j_s \}$ where $t \ge s$.
\begin{itemize}
    \item If $2 \le s \le t$ then $T'$ is obtained by applying $\Gamma_\ell$ from the Grassmannian case to the rectangular part of $T$: $\{i_1, \dots, i_s; j_1, \dots, j_s \}$, and fix the other entries.
    \item If $s = 1$ and $t \ge 2$:
    \begin{itemize}
        \item If $i_1 \in \{1, \dots, \ell \}$, $i_2, i_3, \dots, i_t, j_1 \in \{\ell+1, \dots, n \}$ and $i_2 < j_1$. If $t \ge 3$ then we also require $j_1 < i_3$. Then we define $I' = \{i_1, j_1, i_3, \dots, i_t\}$ and $J' = \{ i_2\}$.
        \item If $i_1 \in \{1, \dots, \ell \}$, $i_2, i_3, \dots, i_t, j_1 \in \{\ell+1, \dots, n \}$ and $j_1 < i_2$ then we define $I' = \{j_1, i_2, i_3, \dots, i_t \}$ and $J' = \{i_1\}$.
        \item Otherwise we define $I' = I$ and $J' = J$.
    \end{itemize}
    \item If $s = t = 1$ then $T' = T$.
\end{itemize}
\end{definition}

\begin{example}
If a semi-standard Young tableau has at least two rows in each column then we use the definition of the $\Gamma_\ell$ from the Grassmannian case. For example
\[
\Gamma_1 \left( \,
\begin{tabular}{cc}
    \hline
    \multicolumn{1}{|c|}{$1$} & \multicolumn{1}{c|}{$2$} \\
    \hline
    \multicolumn{1}{|c|}{$3$} & \multicolumn{1}{c|}{$4$} \\
    \hline
    \multicolumn{1}{|c|}{$5$} & \\
    \cline{1-1}
\end{tabular}
\, \right)
=
\begin{tabular}{cc}
    \hline
    \multicolumn{1}{|c|}{$2$} & \multicolumn{1}{c|}{$4$} \\
    \hline
    \multicolumn{1}{|c|}{$3$} & \multicolumn{1}{c|}{$1$} \\
    \hline
    \multicolumn{1}{|c|}{$5$} & \\
    \cline{1-1}
\end{tabular} \, , \quad
\Gamma_2 \left( \,
\begin{tabular}{cc}
    \hline
    \multicolumn{1}{|c|}{$1$} & \multicolumn{1}{c|}{$2$} \\
    \hline
    \multicolumn{1}{|c|}{$3$} & \multicolumn{1}{c|}{$4$} \\
    \hline
    \multicolumn{1}{|c|}{$5$} & \\
    \cline{1-1}
\end{tabular}
\, \right)
=
\begin{tabular}{cc}
    \hline
    \multicolumn{1}{|c|}{$3$} & \multicolumn{1}{c|}{$4$} \\
    \hline
    \multicolumn{1}{|c|}{$1$} & \multicolumn{1}{c|}{$2$} \\
    \hline
    \multicolumn{1}{|c|}{$5$} & \\
    \cline{1-1}
\end{tabular} \, .
\]
If a semi-standard Young tableau has exactly one column which with a single row then we check the entries in the first two or three rows to determine the image of $\Gamma_\ell$. For example
\[
\Gamma_1 \left( \,
\begin{tabular}{cc}
    \hline
    \multicolumn{1}{|c|}{$1$} & \multicolumn{1}{c|}{$2$} \\
    \hline
    \multicolumn{1}{|c|}{$3$} & \\
    \cline{1-1}
    \multicolumn{1}{|c|}{$4$} & \\
    \cline{1-1}
\end{tabular}
\, \right)
=
\begin{tabular}{cc}
    \hline
    \multicolumn{1}{|c|}{$2$} & \multicolumn{1}{c|}{$1$} \\
    \hline
    \multicolumn{1}{|c|}{$3$} & \\
    \cline{1-1}
    \multicolumn{1}{|c|}{$4$} & \\
    \cline{1-1}
\end{tabular} \, , \quad
\Gamma_1 \left( \,
\begin{tabular}{cc}
    \hline
    \multicolumn{1}{|c|}{$1$} & \multicolumn{1}{c|}{$3$} \\
    \hline
    \multicolumn{1}{|c|}{$2$} & \\
    \cline{1-1}
    \multicolumn{1}{|c|}{$4$} & \\
    \cline{1-1}
\end{tabular}
\, \right)
=
\begin{tabular}{cc}
    \hline
    \multicolumn{1}{|c|}{$3$} & \multicolumn{1}{c|}{$2$} \\
    \hline
    \multicolumn{1}{|c|}{$1$} & \\
    \cline{1-1}
    \multicolumn{1}{|c|}{$4$} & \\
    \cline{1-1}
\end{tabular} \, , \quad
\Gamma_1 \left( \,
\begin{tabular}{cc}
    \hline
    \multicolumn{1}{|c|}{$1$} & \multicolumn{1}{c|}{$3$} \\
    \hline
    \multicolumn{1}{|c|}{$2$} & \\
    \cline{1-1}
    \multicolumn{1}{|c|}{$3$} & \\
    \cline{1-1}
\end{tabular}
\, \right)
=
\begin{tabular}{cc}
    \hline
    \multicolumn{1}{|c|}{$2$} & \multicolumn{1}{c|}{$3$} \\
    \hline
    \multicolumn{1}{|c|}{$1$} & \\
    \cline{1-1}
    \multicolumn{1}{|c|}{$3$} & \\
    \cline{1-1}
\end{tabular} \, . 
\]
\end{example}

\begin{lemma}\label{lem:SSYT_flag_injective}
Let $T, T'$ be semi-standard Young tableaux. If $\Gamma_\ell(T)$ and $\Gamma_\ell(T')$ are row-wise equal then $T = T'$.
\end{lemma}

\begin{proof}
Let us write $T = [IJ]$ and $T' =[I'J']$ where $I = \{i_1, \dots, i_t \}$, $J = \{j_1, \dots, j_s \}$ and $s \le t$. We note that $\Gamma_\ell$ does not alter the shape of the tableau and keeps the row-wise contents of the third row and all rows below. Therefore we write $I' = \{i_1', i_2', i_3, \dots, i_t \}$ and $J' = \{j_1', j_2', j_3, \dots, j_s \}$. We also note that $\Gamma_\ell$, does not change the contents of the tableau as a multi-set so we have $\{i_1, i_2, j_1, j_2 \} = \{i_1', i_2', j_1', j_2' \}$. If $s = |J| \ge 2$ then we are done by the Grassmannian $\textrm{Gr}(2, n)$ case. So let us assume that $s = 1$. We proceed by taking cases on $r = |\{i_1, i_2, j_1 \} \cap \{1, \dots, \ell \}| \in [3]$.

\smallskip

\textbf{Case 1.} Assume $r = 0$ or $3$. Then we have that $\Gamma_\ell(T) = T$ and $\Gamma_\ell(T') = T'$. So $T$ and $T'$ are row-wise equal semi-standard Young tableau, hence $T = T'$.

\smallskip

\textbf{Case 2.} Assume $r = 1$. It follows that $i_1 = i_1' \in \{1, \dots, \ell \}$. Let us assume, by contradiction, that $T \neq T'$. Since $\{i_1, i_2, j_1 \} = \{i_1', i_2', j_1 \}$ and $i_1 = i_1'$, we may assume without loss of generality that $i_2 = j_1' < j_1 = i_2'$. By definition of $\Gamma_\ell$ we have
\[
\Gamma_\ell(T) = \Gamma_\ell \left( \,
\begin{tabular}{cc}
    \hline
    \multicolumn{1}{|c|}{$i_1$} & \multicolumn{1}{c|}{$j_1$} \\
    \hline
    \multicolumn{1}{|c|}{$i_2$} & \\
    \cline{1-1}
    \multicolumn{1}{|c|}{$\vdots$} & \\
    \cline{1-1}
\end{tabular}
\, \right)
=
\begin{tabular}{cc}
    \hline
    \multicolumn{1}{|c|}{$j_1$} & \multicolumn{1}{c|}{$i_2$} \\
    \hline
    \multicolumn{1}{|c|}{$i_1$} & \\
    \cline{1-1}
    \multicolumn{1}{|c|}{$\vdots$} & \\
    \cline{1-1}
\end{tabular} \, , \quad
\Gamma_\ell(T') = \Gamma_\ell \left( \,
\begin{tabular}{cc}
    \hline
    \multicolumn{1}{|c|}{$i_1$} & \multicolumn{1}{c|}{$i_2$} \\
    \hline
    \multicolumn{1}{|c|}{$j_1$} & \\
    \cline{1-1}
    \multicolumn{1}{|c|}{$\vdots$} & \\
    \cline{1-1}
\end{tabular}
\, \right)
=
\begin{tabular}{cc}
    \hline
    \multicolumn{1}{|c|}{$i_2$} & \multicolumn{1}{c|}{$i_1$} \\
    \hline
    \multicolumn{1}{|c|}{$j_1$} & \\
    \cline{1-1}
    \multicolumn{1}{|c|}{$\vdots$} & \\
    \cline{1-1}
\end{tabular} \, .
\]
However $\Gamma_\ell(T)$ and $\Gamma_\ell(T')$ are not row-wise equal, a contradiction.

\smallskip

\textbf{Case 3.} Assume $r = 2$. Since $T$ and $T'$ are semi-standard Young tableau and $\{i_1, i_2, j_1 \} = \{i_1', i_2', j_1' \}$, we have that $i_1 = i_1' \in \{1, \dots, \ell \}$. Assume, by contradiction, that $T \neq T'$. Without loss of generality we have that $i_2 = j_1' \in \{1, \dots, \ell \}$ and $j_1 = i_2' \in \{\ell+1, \dots, n \}$.
Therefore
\[
\Gamma_\ell(T) = \Gamma_\ell \left( \,
\begin{tabular}{cc}
    \hline
    \multicolumn{1}{|c|}{$i_1$} & \multicolumn{1}{c|}{$j_1$} \\
    \hline
    \multicolumn{1}{|c|}{$i_2$} & \\
    \cline{1-1}
    \multicolumn{1}{|c|}{$\vdots$} & \\
    \cline{1-1}
\end{tabular}
\, \right)
=
\begin{tabular}{cc}
    \hline
    \multicolumn{1}{|c|}{$i_1$} & \multicolumn{1}{c|}{$j_1$} \\
    \hline
    \multicolumn{1}{|c|}{$i_2$} & \\
    \cline{1-1}
    \multicolumn{1}{|c|}{$\vdots$} & \\
    \cline{1-1}
\end{tabular} \, , \quad
\Gamma_\ell(T') = \Gamma_\ell \left( \,
\begin{tabular}{cc}
    \hline
    \multicolumn{1}{|c|}{$i_1$} & \multicolumn{1}{c|}{$i_2$} \\
    \hline
    \multicolumn{1}{|c|}{$j_1$} & \\
    \cline{1-1}
    \multicolumn{1}{|c|}{$\vdots$} & \\
    \cline{1-1}
\end{tabular}
\, \right)
=
\begin{tabular}{cc}
    \hline
    \multicolumn{1}{|c|}{$j_1$} & \multicolumn{1}{c|}{$i_2$} \\
    \hline
    \multicolumn{1}{|c|}{$i_1$} & \\
    \cline{1-1}
    \multicolumn{1}{|c|}{$\vdots$} & \\
    \cline{1-1}
\end{tabular} \, .
\]
However $\Gamma_\ell(T)$ and $\Gamma_\ell(T')$ are not row-wise equal, a contradiction.
\end{proof}

\begin{lemma}\label{lem:SSYT_flag_surjective} Let $T$ be a matching field tableau for $B_\ell$ with two columns. Then there exists a semi-standard Young tableau $T'$ such that $T$ and $\Gamma_\ell(T')$ are row-wise equal.
\end{lemma}

\begin{proof}
Write $T = [IJ]$ a matching field tableau for $B_\ell$ with columns $I = \{i_1, \dots, i_t \}$ and $J = \{j_1, \dots, j_s \}$ where $s \le t$. 

If $2 \le s \le t$ then consider the tableau $\underline T = [\underline I \underline J]$ where $\underline I = \{i_1, \dots, i_s\}$ and $\underline J = \{j_1, \dots, j_s \}$. By the Grassmannian case, there exists a semi-standard Young tableau $\underline T' = [\underline I' \underline J']$ such that $\Gamma_\ell(\underline T') = \underline T$. Let $T' = [I'J']$ be the tableau with $I' = \{i_1', \dots, i_s',i_{s+1}, \dots, i_t\}$ and $J' = \underline J'$. It follows that $T'$ is a semi-standard Young tableau and $\Gamma_\ell(S) = T$.

Let us now assume that $|J|= 1$. If $|I| = 1$ then the result holds easily so we assume $|I| \ge 2$. We proceed by taking cases on $r = |\{i_1, i_2, j_1 \} \cap \{1, \dots, \ell \}|$. We count elements with multiplicity in case $j_1 = i_1$ or $j_1 = i_2$.

\textbf{Case 1.} Assume that $r = 0$ or $3$. Then we can perform row-wise swaps on entries of $T$ to put the tableau in semi-standard form, call this tableau $T'$. In this case we have $\Gamma_\ell(T') = T'$.

\textbf{Case 2.} Assume that $r = 1$. If $i_2 \in \{1, \dots, \ell \}$ let $\alpha = \min\{i_1, j_1 \}$ and $\beta = \max\{i_1, j_1 \}$. Then we have
\[
\Gamma_\ell \left( \,
\begin{tabular}{cc}
    \hline
    \multicolumn{1}{|c|}{$i_2$} & \multicolumn{1}{c|}{$\alpha$} \\
    \hline
    \multicolumn{1}{|c|}{$\beta$} & \\
    \cline{1-1}
    \multicolumn{1}{|c|}{$\vdots$} & \\
    \cline{1-1}
\end{tabular} 
\, \right) =
\begin{tabular}{cc}
    \hline
    \multicolumn{1}{|c|}{$\beta$} & \multicolumn{1}{c|}{$\alpha$} \\
    \hline
    \multicolumn{1}{|c|}{$i_2$} & \\
    \cline{1-1}
    \multicolumn{1}{|c|}{$\vdots$} & \\
    \cline{1-1}
\end{tabular} 
\textrm{ or }
\begin{tabular}{cc}
    \hline
    \multicolumn{1}{|c|}{$\alpha$} & \multicolumn{1}{c|}{$\beta$} \\
    \hline
    \multicolumn{1}{|c|}{$i_2$} & \\
    \cline{1-1}
    \multicolumn{1}{|c|}{$\vdots$} & \\
    \cline{1-1}
\end{tabular} 
\]
If $j_1 \in \{1, \dots, \ell \}$ then we have
\[
\Gamma_\ell \left( \,
\begin{tabular}{cc}
    \hline
    \multicolumn{1}{|c|}{$j_1$} & \multicolumn{1}{c|}{$i_2$} \\
    \hline
    \multicolumn{1}{|c|}{$i_1$} & \\
    \cline{1-1}
    \multicolumn{1}{|c|}{$\vdots$} & \\
    \cline{1-1}
\end{tabular} 
\, \right) =
\begin{tabular}{cc}
    \hline
    \multicolumn{1}{|c|}{$i_1$} & \multicolumn{1}{c|}{$j_1$} \\
    \hline
    \multicolumn{1}{|c|}{$i_2$} & \\
    \cline{1-1}
    \multicolumn{1}{|c|}{$\vdots$} & \\
    \cline{1-1}
\end{tabular} \, .
\]

\textbf{Case 3.} Assume $r = 2$. If $i_1, i_2 \in \{1, \dots, \ell \}$ then $\Gamma_\ell(T) = T$. If $i_2, j_1 \in \{1, \dots, \ell \}$ then
\[
\Gamma_\ell \left( \,
\begin{tabular}{cc}
    \hline
    \multicolumn{1}{|c|}{$i_2$} & \multicolumn{1}{c|}{$j_1$} \\
    \hline
    \multicolumn{1}{|c|}{$i_1$} & \\
    \cline{1-1}
    \multicolumn{1}{|c|}{$\vdots$} & \\
    \cline{1-1}
\end{tabular} 
\, \right) =
\begin{tabular}{cc}
    \hline
    \multicolumn{1}{|c|}{$i_1$} & \multicolumn{1}{c|}{$j_1$} \\
    \hline
    \multicolumn{1}{|c|}{$i_2$} & \\
    \cline{1-1}
    \multicolumn{1}{|c|}{$\vdots$} & \\
    \cline{1-1}
\end{tabular} \, .
\]
\end{proof}

\begin{lemma}\label{lem:SSYT_flag_preimage_non_vanish}
Let $w \in S_n$ be any permutation and $T = [IJ]$ be a semi-standard Young tableau with two columns. Write $T' = [I'J'] = \Gamma_\ell(T)$. If $I', J' \le w$ then $I, J \le w$.
\end{lemma}

\begin{proof}
We write $I = \{i_1 < \dots < i_t \}$ and $J = \{j_1 < \dots < j_s \}$ for the columns of $T$. If $T$ and $T'$ have the same column-wise contents then the result holds trivially. So let us assume $T$ and $T'$ have different column-wise contents. If $|I| = |J|$ then we are done by the Grassmannian case. And so we assume $s < t$ and we take cases on the contents of the first two rows of $I, J$.

\textbf{Case 1.} Assume $2 \le s$, $j_1 < i_2$, $i_1 \in \{1, \dots, \ell \}$ and $i_2, j_1, j_2 \in \{ \ell+1, \dots, n\}$. So the tableaux are
\[
\Gamma_\ell \left( \,
\begin{tabular}{cc}
    \multicolumn{1}{c}{$I$} & \multicolumn{1}{c}{$J$} \\
    \hline
    \multicolumn{1}{|c|}{$i_1$} & \multicolumn{1}{c|}{$j_1$} \\
    \hline
    \multicolumn{1}{|c|}{$i_2$} & 
    \multicolumn{1}{|c|}{$j_2$} \\
    \hline
    \multicolumn{1}{|c|}{$\vdots$} &
    \multicolumn{1}{|c|}{$\vdots$} \\
    \hline
\end{tabular} 
\, \right) =
\begin{tabular}{cc}
    \multicolumn{1}{c}{$I'$} & \multicolumn{1}{c}{$J'$} \\
    \hline
    \multicolumn{1}{|c|}{$j_1$} & \multicolumn{1}{c|}{$j_2$} \\
    \hline
    \multicolumn{1}{|c|}{$i_2$} & 
    \multicolumn{1}{|c|}{$i_1$} \\
    \hline
    \multicolumn{1}{|c|}{$\vdots$} &
    \multicolumn{1}{|c|}{$\vdots$} \\
    \hline
\end{tabular} \, .
\]
Since $I' \le w$ we have $j_1 \le \widehat w_1$ and $i_2 \le \widehat w_2$. Since $J' \le w$ we have $i_1 \le \widetilde w_1$ and $j_2 \le \widetilde w_2$. So we have that $i_1 < j_1 \le \widehat w_1$ and $i_2 \le \widehat w_2$ hence $I \le w$. We also have that $j_1 \le \widehat w_1 \le \widetilde w_1$ and $j_2 \le \widetilde w_2$, hence $J \le w$. 

\smallskip

\textbf{Case 2.} Assume $2 \le s$, $j_1 < i_2$, $i_1, i_2, j_1 \in \{1, \dots, \ell \}$ and $j_2 \in \{ \ell+1, \dots, n\}$. So the tableaux are
\[
\Gamma_\ell \left( \,
\begin{tabular}{cc}
    \multicolumn{1}{c}{$I$} & \multicolumn{1}{c}{$J$} \\
    \hline
    \multicolumn{1}{|c|}{$i_1$} & \multicolumn{1}{c|}{$j_1$} \\
    \hline
    \multicolumn{1}{|c|}{$i_2$} & 
    \multicolumn{1}{|c|}{$j_2$} \\
    \hline
    \multicolumn{1}{|c|}{$\vdots$} &
    \multicolumn{1}{|c|}{$\vdots$} \\
    \hline
\end{tabular} 
\, \right) =
\begin{tabular}{cc}
    \multicolumn{1}{c}{$I'$} & \multicolumn{1}{c}{$J'$} \\
    \hline
    \multicolumn{1}{|c|}{$j_1$} & \multicolumn{1}{c|}{$j_2$} \\
    \hline
    \multicolumn{1}{|c|}{$i_2$} & 
    \multicolumn{1}{|c|}{$i_1$} \\
    \hline
    \multicolumn{1}{|c|}{$\vdots$} &
    \multicolumn{1}{|c|}{$\vdots$} \\
    \hline
\end{tabular} \, .
\]
The result follows by the same argument as Case~1.

\smallskip

\textbf{Case 3.} Assume $s = 1, t \ge 2$, $j_1 < i_2$, $i_1 \in \{1, \dots, \ell \}$ and $j_1, i_2 \in \{\ell+1, \dots, n \}$. The tableaux are
\[
\Gamma_\ell \left( \,
\begin{tabular}{cc}
    \multicolumn{1}{c}{$I$} & \multicolumn{1}{c}{$J$} \\
    \hline
    \multicolumn{1}{|c|}{$i_1$} & \multicolumn{1}{c|}{$j_1$} \\
    \hline
    \multicolumn{1}{|c|}{$i_2$} & \\
    \cline{1-1}
    \multicolumn{1}{|c|}{$\vdots$} & \\
    \cline{1-1}
\end{tabular} 
\, \right) =
\begin{tabular}{cc}
    \multicolumn{1}{c}{$I'$} & \multicolumn{1}{c}{$J'$} \\
    \hline
    \multicolumn{1}{|c|}{$j_1$} & \multicolumn{1}{c|}{$i_1$} \\
    \hline
    \multicolumn{1}{|c|}{$i_2$} & \\
    \cline{1-1}
    \multicolumn{1}{|c|}{$\vdots$} & \\
    \cline{1-1}
\end{tabular} \, .
\]
Since $J' \le w$, we have $i_1 \le w$. Since $I' \le w$, we have $j_1 \le \widehat w_1$ and $i_2 \le \widehat w_2$. And so we have $i_1 < j_1 \le \widehat w_1$ and $i_2 \le \widehat w_2$, hence $I \le w$. We also have $j_1 \le \widehat w_1 \le w_1$, hence $J \le w$.

\smallskip

\textbf{Case 4.} Assume $s = 1, t \ge 2$, $i_2 < j_1$, $i_1 \in \{1, \dots, \ell \}$ and $i_1, j_2 \in \{\ell+1, \dots, n \}$. If $t \ge 3$ then we assume that $j_1 < i_3$. The tableaux are
\[
\Gamma_\ell \left( \,
\begin{tabular}{cc}
    \multicolumn{1}{c}{$I$} & \multicolumn{1}{c}{$J$} \\
    \hline
    \multicolumn{1}{|c|}{$i_1$} & \multicolumn{1}{c|}{$j_1$} \\
    \hline
    \multicolumn{1}{|c|}{$i_2$} & \\
    \cline{1-1}
    \multicolumn{1}{|c|}{$\vdots$} & \\
    \cline{1-1}
\end{tabular} 
\, \right) =
\begin{tabular}{cc}
    \multicolumn{1}{c}{$I'$} & \multicolumn{1}{c}{$J'$} \\
    \hline
    \multicolumn{1}{|c|}{$j_1$} & \multicolumn{1}{c|}{$i_2$} \\
    \hline
    \multicolumn{1}{|c|}{$i_1$} & \\
    \cline{1-1}
    \multicolumn{1}{|c|}{$\vdots$} & \\
    \cline{1-1}
\end{tabular} \, .
\]
Since $I' \le w$ we have $i_1 \le \widehat w_1$ and $j_1 \le \widehat w_2$. Since $J' \le w$ we have $i_2 \le w_1$. So $i_1 \le \widehat w_1$ and $i_2 < j_1 \le \widehat w_2$ hence $I \le w$. We also have $j_1 \le \widehat w_1 \le w_1$ hence $J \le w$.
\end{proof}

We now give some important properties for permutations $w\in P_\ell$.

\begin{lemma}\label{lem:Q_perms_312}
Let $w \in P_\ell$. If $i<j<k$ and $w_i, w_j, w_k$ is of type $312$, then $w_i = w_1$ and $w_j = \ell$.
\end{lemma}

\begin{proof}
By definition $w \backslash \ell$ is $312$-free, so we have that $\ell \in \{w_i, w_j, w_k \}$. Also by definition, we have $w_1 > w_2 = \ell$. So $w_k \neq \ell$ because $k \ge 3$. We have $w_i \neq \ell$ otherwise $w_1, w_j, w_k$ would be of type $312$ lying in $w \backslash \ell$. So we have $w_j = \ell = w_2$. Since $i < j = 2$, we have $w_i = w_1$.
\end{proof}

\begin{lemma}\label{lem:Q_perm_restrict_le_m}
If $w \in P_\ell$ and $w$ is not $312$-free then $w|_{w_1} = (w_1, \ell, w_1 - 1, \dots, \ell+1, \ell-1, \dots, 1)$.
\end{lemma}

\begin{proof}
Let $3 \le i < j \le n$. If $w_i, w_j < w_1$ then we have $w_i > w_j$ otherwise $w_1, w_i, w_j$ is a $312$ in $w$ that also appears in $w \backslash \ell$. So $w|_{w_1} = (w_1, \ell, w_1 - 1, \dots, \ell+1, \ell-1, \dots, 1)$.
\end{proof}

\begin{lemma}\label{lem:Q_perm_w1=m-1}
If $w \in P_\ell$ and $w|_m = (m-1, m, m-2, \dots, 1)$ for some $3 \le m \le n$ then $w_1 = m-1$.
\end{lemma}

\begin{proof}
By definition $w_1 < w_2 \le \ell$ and $w|_{w_2} = (w_1, w_2, w_2 - 1, \dots, w_1+1, w_1-1, \dots, 1)$. If $m < w_1$ then $w|_m = (w|_{w_2})|_m = (m, m-1, \dots, 1) \neq (m-1, m, m-2, \dots, 1)$, a contradiction. So $m \ge w_1$. Therefore $w|_m = (w_1, \dots) = (m-1, \dots)$ hence $w_1 = m-1$.
\end{proof}

\begin{lemma}\label{lem:Q_perm_w1_w2_consecutive}
Let $w \in P_\ell$. If there exists $3 \le m \le n$ such that $w|_m = (m-1, m, \dots )$ then either $m - 1 = w_1 < w_2 \le \ell$ or $w_1 > w_2 = \ell = m-1$.
\end{lemma}

\begin{proof}
If $w$ is $312$-free then write $w|_m = (w_1, w_2, \dots, w_m) = (m-1, m, \dots)$. For any $3 \le i < j \le n$ we have $w_i > w_j$, otherwise $w_2, w_i, w_j$ is of type $312$. Hence $w|_m = (m-1, m, m-2, \dots, 1)$ so by definition of $P_\ell$ we have $w_1 < w_2 \le \ell$. And by Lemma~\ref{lem:Q_perm_w1=m-1} we have $w_1 = m-1$.

If $w$ is not $312$-free then write $w = (w_1, \dots, w_n)$. By definition of $P_\ell$, we have $w_1 > w_2 = \ell$. If $m \ge w_1$ then we have $w|_m = (w_1, \ell, \dots) \neq (m-1, m, \dots)$, a contradiction. If $m < w_1$ then by Lemma~\ref{lem:Q_perm_restrict_le_m} we either have $w|_m = (m, m-1,\dots) \neq (m-1, m, \dots)$ a contradiction or $w|_m = (\ell, m, m-1, \dots, \ell+1, \ell-1, \dots, 1)$ and so we have $\ell = m-1$.
\end{proof}

\begin{lemma}\label{lem:Q_perm_w1_w2_gap} Let $w \in P_\ell$ and $t \ge 2$. Write $\widehat w_1 = \min\{w_1, \dots, w_t \}$ and $\widehat w_2 = \min(\{w_1, \dots, w_t \} \backslash \widehat w_1)$. If $\widehat w_1 \ge 2$ and $\widehat w_2 > \widehat w_1 + 1$ then either
$\widehat w_1 = w_1 <\widehat w_2 \le w_2 \le \ell$ or $w_1 > w_2 = \ell = \widehat w_1$.
\end{lemma}

\begin{proof}
Let $w_d = \widehat w_1 + 1$. Since $\widehat w_1 < w_d < \widehat w_2$, we have $d \in \{t+1, \dots, n \}$. If there exists $i \in \{t+1, \dots, d-1 \}$ such that $w_i < w_d - 1$ then $\widehat w_2, w_i, w_d$ is of type $312$. So by Lemma~\ref{lem:Q_perms_312} we have $w_i = \ell$ however $w_i < \min\{ w_1, \dots, w_t\} \le w_2 = \ell$, a contradiction. So for all $i \in \{ t+1, \dots, d-1\}$ we have $w_i > w_d$. And so $w|_{w_d} = (w_d-1, w_d, \dots)$. By Lemma~\ref{lem:Q_perm_w1_w2_consecutive} we have either $w_d - 1 = w_1 < w_2 \le \ell$ or $w_1 > w_2 = \ell = w_d - 1$. In the former case note that $\widehat w_2 \le w_2 \le \ell$.
\end{proof}

\begin{lemma} \label{lem:SSYT_flag_image_non_vanish_Q_ell}
Fix $w \in P_\ell$. Let $[IJ]$ be a semi-standard Young tableau and write $[I'J'] = \Gamma_\ell([IJ])$. If $I, J \le w$ then $I', J' \le w$.
\end{lemma}

\begin{proof}
We write $I = \{i_1 < \dots < i_t \}$ and $J = \{ j_1 < \dots < j_s\}$ where $s \le t$. We use the following notation to denote ordered initial segments of $w$. We let $\{w_1, \dots w_t\} = \{\widehat w_1 < \dots < \widehat w_t \}$ and $\{w_1, \dots, w_s \} = \{\widetilde w_1 < \dots < \widetilde w_s \}$. We have that $i_3 \le \widehat w_3, \dots, i_t \le \widehat w_3$ and $j_3 \le \widetilde w_3, \dots, j_s \le \widetilde w_s$. Since $\Gamma_\ell$ fixes all entries in the third row and below it remains to check the entries in the first two rows of $\Gamma_\ell(T)$. Note that $I \le w$ implies $i_1 \le \widehat w_1$ and $i_2 \le \widehat w_2$. Also $J \le w$ implies $j_1 \le \widetilde w_1$ and $j_2 \le \widetilde w_2$.

If $s = t$, then the result holds by the Grassmannian case so we assume that $s < t$. If $\Gamma_\ell$ fixes the entries in each column then the result trivially holds. So we take cases on the tableau $T$ whose columns are not fixed by $\Gamma_\ell$.

\smallskip

\textbf{Case 1.} Assume $2 \le s$, $j_1 < i_2$, $i_1 \in \{1, \dots, \ell \}$ and $i_2, j_1, j_2 \in \{ \ell+1, \dots, n\}$. So the tableaux are
\[
\Gamma_\ell \left( \,
\begin{tabular}{cc}
    \multicolumn{1}{c}{$I$} & \multicolumn{1}{c}{$J$} \\
    \hline
    \multicolumn{1}{|c|}{$i_1$} & \multicolumn{1}{c|}{$j_1$} \\
    \hline
    \multicolumn{1}{|c|}{$i_2$} & 
    \multicolumn{1}{|c|}{$j_2$} \\
    \hline
    \multicolumn{1}{|c|}{$\vdots$} &
    \multicolumn{1}{|c|}{$\vdots$} \\
    \hline
\end{tabular} 
\, \right) =
\begin{tabular}{cc}
    \multicolumn{1}{c}{$I'$} & \multicolumn{1}{c}{$J'$} \\
    \hline
    \multicolumn{1}{|c|}{$j_1$} & \multicolumn{1}{c|}{$j_2$} \\
    \hline
    \multicolumn{1}{|c|}{$i_2$} & 
    \multicolumn{1}{|c|}{$i_1$} \\
    \hline
    \multicolumn{1}{|c|}{$\vdots$} &
    \multicolumn{1}{|c|}{$\vdots$} \\
    \hline
\end{tabular} \, .
\]
We have that $i_1 \le j_1 \le \widetilde w_1$ and $j_2 \le \widetilde w_2$ and so $J' \le w$.

Assume by contradiction that $I' \not\le w$. Since $i_2 \le \widehat w_2$, we must have $j_1 > \widehat w_1$. Therefore $\widehat w_1 < j_1 < i_2 \le \widehat w_2 \le \widetilde w_1$ and so $\widehat w_1 + 1 \in \{w_{t+1}, \dots, w_n\}$. Since $j_1 \le \widetilde w_1$ we must have that $\widehat w_1 \in \{w_{s+1}, \dots, w_t\}$. And so $\widetilde w_1, \widehat w_1, \widehat w_1 + 1$ is of type $312$. By definition of $P_\ell$ if $w$ is not $312$-free then $w_2 = \ell$. However $j_1 \le \widetilde w_1 = \min\{w_1, w_2 \dots, w_s \} \le w_2 = \ell$, a contradiction.

\smallskip

\textbf{Case 2.} Assume $2 \le s$, $j_1 < i_2$, $i_1, i_2, j_1 \in \{1, \dots, \ell \}$ and $j_2 \in \{ \ell+1, \dots, n\}$. So the tableaux are
\[
\Gamma_\ell \left( \,
\begin{tabular}{cc}
    \multicolumn{1}{c}{$I$} & \multicolumn{1}{c}{$J$} \\
    \hline
    \multicolumn{1}{|c|}{$i_1$} & \multicolumn{1}{c|}{$j_1$} \\
    \hline
    \multicolumn{1}{|c|}{$i_2$} & 
    \multicolumn{1}{|c|}{$j_2$} \\
    \hline
    \multicolumn{1}{|c|}{$\vdots$} &
    \multicolumn{1}{|c|}{$\vdots$} \\
    \hline
\end{tabular} 
\, \right) =
\begin{tabular}{cc}
    \multicolumn{1}{c}{$I'$} & \multicolumn{1}{c}{$J'$} \\
    \hline
    \multicolumn{1}{|c|}{$j_1$} & \multicolumn{1}{c|}{$j_2$} \\
    \hline
    \multicolumn{1}{|c|}{$i_2$} & 
    \multicolumn{1}{|c|}{$i_1$} \\
    \hline
    \multicolumn{1}{|c|}{$\vdots$} &
    \multicolumn{1}{|c|}{$\vdots$} \\
    \hline
\end{tabular} \, .
\]
We have that $i_1 \le j_1 \le \widetilde w_1$ and $j_2 \le \widetilde w_2$ and so $J' \le w$.

Assume by contradiction that $I' \not\le w$. Since $i_2 \le \widehat w_2$, we must have $j_1 > \widehat w_1$. Therefore $\widehat w_1 < j_1 < i_2 \le \widehat w_2 \le \widetilde w_1$ and so $\widehat w_1 + 1 \in \{w_{t+1}, \dots, w_n\}$. Since $j_1 \le \widetilde w_1$ we must have that $\widehat w_1 \in \{w_{s+1}, \dots, w_t\}$. And so $\widetilde w_1, \widehat w_1, \widehat w_1 + 1$ is of type $312$. So by Lemma~\ref{lem:Q_perms_312}, $\widehat w_1 = \ell$. However $\widehat w_1 < j_1$ and $j_1 \in \{1, \dots, \ell \}$, a contradiction.

\smallskip

\textbf{Case 3.} Assume $s = 1, t \ge 2$, $j_1 < i_2$, $i_1 \in \{1, \dots, \ell \}$ and $j_1, i_2 \in \{\ell+1, \dots, n \}$. The tableaux are
\[
\Gamma_\ell \left( \,
\begin{tabular}{cc}
    \multicolumn{1}{c}{$I$} & \multicolumn{1}{c}{$J$} \\
    \hline
    \multicolumn{1}{|c|}{$i_1$} & \multicolumn{1}{c|}{$j_1$} \\
    \hline
    \multicolumn{1}{|c|}{$i_2$} & \\
    \cline{1-1}
    \multicolumn{1}{|c|}{$\vdots$} & \\
    \cline{1-1}
\end{tabular} 
\, \right) =
\begin{tabular}{cc}
    \multicolumn{1}{c}{$I'$} & \multicolumn{1}{c}{$J'$} \\
    \hline
    \multicolumn{1}{|c|}{$j_1$} & \multicolumn{1}{c|}{$i_1$} \\
    \hline
    \multicolumn{1}{|c|}{$i_2$} & \\
    \cline{1-1}
    \multicolumn{1}{|c|}{$\vdots$} & \\
    \cline{1-1}
\end{tabular} \, .
\]
We have $i_1 < j_1 \le w_1 = \widetilde w_1$ hence $J' \le w$.

Assume by contradiction that $I' \not\le w$. Since $i_2 \le \widehat w_2$, we must have $j_1 > \widehat w_1$. Therefore we have $\widehat w_1 < j_1 < i_2 \le \widehat w_2$. And so we have $\widehat w_1 + 1 \in \{w_{t+1}, \dots, w_n \}$. Since $j_1 \le w_1$ therefore $\widehat w_1 \in \{w_2, \dots, w_t \}$. So $w_1, \widehat w_1, \widehat w_1 + 1$ is of type $312$. So by Lemma~\ref{lem:Q_perms_312} we have $\widehat w_2 = \ell$. And so $j_1 < i_2 \le \ell$, a contradiction.

\smallskip

\textbf{Case 4.} Assume $s = 1, t \ge 2$, $i_2 < j_1$, $i_1 \in \{1, \dots, \ell \}$ and $i_2, j_1 \in \{\ell+1, \dots, n \}$. If $t \ge 3$ then we assume that $j_1 < i_3$. The tableaux are
\[
\Gamma_\ell \left( \,
\begin{tabular}{cc}
    \multicolumn{1}{c}{$I$} & \multicolumn{1}{c}{$J$} \\
    \hline
    \multicolumn{1}{|c|}{$i_1$} & \multicolumn{1}{c|}{$j_1$} \\
    \hline
    \multicolumn{1}{|c|}{$i_2$} & \\
    \cline{1-1}
    \multicolumn{1}{|c|}{$\vdots$} & \\
    \cline{1-1}
\end{tabular} 
\, \right) =
\begin{tabular}{cc}
    \multicolumn{1}{c}{$I'$} & \multicolumn{1}{c}{$J'$} \\
    \hline
    \multicolumn{1}{|c|}{$j_1$} & \multicolumn{1}{c|}{$i_2$} \\
    \hline
    \multicolumn{1}{|c|}{$i_1$} & \\
    \cline{1-1}
    \multicolumn{1}{|c|}{$\vdots$} & \\
    \cline{1-1}
\end{tabular} \, .
\]
We have $i_2 < j_1 \le w_1$ and so $J' \le w$.

Assume by contradiction that $I' \not\le w$. We have $i_1 \le \widehat w_1$ and so we must have $j_1 > \widehat w_2$. Therefore $\widehat w_1 < \widehat w_2 < j_1 \le w_1$. If $t = 2$ then we have a contradiction since $w_1 \in \{\widehat w_1, \widehat w_2 \}$. So we have $t \ge 3$. Therefore $\widehat w_2 < j_1 < i_3 \le \widehat w_3$, and so $\widehat w_2 + 1 \in \{w_{t+1}, \dots, w_n \}$. Since $w_1 > \widehat w_2$, we have $w_1 \ge \widehat w_3$. And so $w_1, \widehat w_2, \widehat w_2 + 1$ is of type $312$. By Lemma~\ref{lem:Q_perms_312}, we have $\widehat w_2 = \ell$. And so $i_2 \le \widehat w_2 = \ell$ but by assumption $i_2 \in \{\ell+1, \dots, n \}$, a contradiction.
\end{proof}

\begin{lemma} \label{lem:SSYT_flag_surject_w_Q_ell}
Fix $w \in P_\ell$. If $[IJ]$ is a matching field tableau for $B_\ell$ such that $I, J \le w$ then there exists a semi-standard Young tableau $[\tilde I \tilde J]$ such that $\tilde I, \tilde J \le w$ and $\Gamma_\ell([\tilde I \tilde J])$ is row-wise equal to $[IJ]$.
\end{lemma}

\begin{proof}
By Lemma~\ref{lem:SSYT_flag_surjective} we have for any $T = [IJ]$ a matching field tableau for $B_\ell$ there exists a semi-standard Young tableau $\tilde T$ such that $\Gamma_\ell(\tilde T) = T' = [I'J']$ is row-wise equal to $T$. We proceed to show that if $I, J \le w$ then $I', J' \le w$. Once we show this we have $\tilde T \in SSYT_2(w)$ by Lemma~\ref{lem:SSYT_flag_preimage_non_vanish}.

We write $I = \{i_1, i_2, \dots, i_t \}$ and $J = \{j_1, j_2, \dots, j_s \}$ where $s < t$. We use the following notation to denote ordered initial segments of $w$. Let $\{w_1, \dots w_t\} = \{\widehat w_1 < \dots < \widehat w_t \}$ and $\{w_1, \dots, w_s \} = \{\widetilde w_1 < \dots < \widetilde w_s \}$. We proceed by taking cases on $s$ and $t$. Note that if $s = t$ then the result holds by the Grassmannian case. Also note that if $T = T'$ then there is nothing to prove so we will assume $T \neq T'$. Without loss of generality, we assume that the third row and below of the tableau $T$ is in semi-standard form, i.e. $i_3 \le j_3, \dots, i_s \le j_s$ so we will focus on the first two rows of $I'$ and $J'$.

\smallskip

\textbf{Case 1.} Assume $s = 1$ and $t \ge 2$. Since $T$ and $T'$ are different we must have the tableaux
\[
T =
\begin{tabular}{cc}
    \multicolumn{1}{c}{$I$} & \multicolumn{1}{c}{$J$} \\
    \hline
    \multicolumn{1}{|c|}{$i_1$} & \multicolumn{1}{c|}{$j_1$} \\
    \hline
    \multicolumn{1}{|c|}{$i_2$} & \\
    \cline{1-1}
    \multicolumn{1}{|c|}{$\vdots$} & \\
    \cline{1-1}
\end{tabular}\, , \quad
T' =
\begin{tabular}{cc}
    \multicolumn{1}{c}{$I'$} & \multicolumn{1}{c}{$J'$} \\
    \hline
    \multicolumn{1}{|c|}{$j_1$} & \multicolumn{1}{c|}{$i_1$} \\
    \hline
    \multicolumn{1}{|c|}{$i_2$} & \\
    \cline{1-1}
    \multicolumn{1}{|c|}{$\vdots$} & \\
    \cline{1-1}
\end{tabular}\, .
\]
Since $T$ and $T'$ are different we have $i_1 \neq j_1$. We proceed by taking cases on $r = |\{i_1, i_2, j_1 \} \cap \{1, \dots, \ell \}|$.

\textbf{Case 1.1} Assume $r = 0$ or $3$. Therefore the entries of $T'$ are fixed by $\Gamma_\ell$. So $T'$ must be in semi-standard form, hence $j_1 < i_1$. Since $I \le w$ we have $i_1 \le \widehat w_1$ and $i_2 \le \widehat i_2$. Since $J \le w$ we have $j_1 \le w_1$.

We have $j_1 < i_1 \le \widehat w_1$ and $i_2 \le \widehat w_2$ and $I' \le w$. We also have $i_1 \le \widehat w_1 \le w_1$ and so $J' \le w$.

\textbf{Case 1.2} Assume $r = 1$. If either $i_1$ or $j_1$ lie in $\{1, \dots, \ell \}$ then we have that $I$ or $I'$ respectively is not a valid column with respect to the matching field $B_\ell$. So we must have $i_2 \in \{1, \dots, \ell \}$. Since $T'$ lies in the image of $\Gamma_\ell$, we have that two cases for the values of $i_1$ and $j_1$.

\textbf{Case 1.2.1} Assume $i_1 < j_1$. So we have
\[
\Gamma_\ell \left( \,
\begin{tabular}{cc}
    \multicolumn{1}{c}{$\tilde I$} & \multicolumn{1}{c}{$\tilde J$} \\
    \hline
    \multicolumn{1}{|c|}{$i_2$} & \multicolumn{1}{c|}{$j_1$} \\
    \hline
    \multicolumn{1}{|c|}{$i_1$} & \\
    \cline{1-1}
    \multicolumn{1}{|c|}{$\vdots$} & \\
    \cline{1-1}
\end{tabular}
\, \right)
=
\begin{tabular}{cc}
    \multicolumn{1}{c}{$I'$} & \multicolumn{1}{c}{$J'$} \\
    \hline
    \multicolumn{1}{|c|}{$j_1$} & \multicolumn{1}{c|}{$i_1$} \\
    \hline
    \multicolumn{1}{|c|}{$i_2$} & \\
    \cline{1-1}
    \multicolumn{1}{|c|}{$\vdots$} & \\
    \cline{1-1}
\end{tabular}\, .
\]
Since the contents of $\tilde T = [\tilde I \tilde J]$ is column-wise equal to $T$, therefore $\tilde I, \tilde J \le w$.

\textbf{Case 1.2.2} Assume $i_1 > j_1$. We note that $\Gamma_\ell$ acts by permuting the entries of the first two rows of a tableau. By permuting the entries of the first two rows of $T'$, there are two ways to form a semi-standard Young tableau. However we have
\[
\Gamma_\ell \left( \,
\begin{tabular}{cc}
    \hline
    \multicolumn{1}{|c|}{$i_2$} & \multicolumn{1}{c|}{$i_1$} \\
    \hline
    \multicolumn{1}{|c|}{$j_1$} & \\
    \cline{1-1}
    \multicolumn{1}{|c|}{$\vdots$} & \\
    \cline{1-1}
\end{tabular}
\, \right)
=
\begin{tabular}{cc}
    \hline
    \multicolumn{1}{|c|}{$i_1$} & \multicolumn{1}{c|}{$j_1$} \\
    \hline
    \multicolumn{1}{|c|}{$i_2$} & \\
    \cline{1-1}
    \multicolumn{1}{|c|}{$\vdots$} & \\
    \cline{1-1}
\end{tabular}\, , \quad
\Gamma_\ell \left( \,
\begin{tabular}{cc}
    \hline
    \multicolumn{1}{|c|}{$i_2$} & \multicolumn{1}{c|}{$j_1$} \\
    \hline
    \multicolumn{1}{|c|}{$i_1$} & \\
    \cline{1-1}
    \multicolumn{1}{|c|}{$\vdots$} & \\
    \cline{1-1}
\end{tabular}
\, \right)
=
\begin{tabular}{cc}
    \hline
    \multicolumn{1}{|c|}{$j_1$} & \multicolumn{1}{c|}{$i_2$} \\
    \hline
    \multicolumn{1}{|c|}{$i_1$} & \\
    \cline{1-1}
    \multicolumn{1}{|c|}{$\vdots$} & \\
    \cline{1-1}
\end{tabular}\, .
\]
And so each such semi-standard Young tableau does not map to $T'$.

\textbf{Case 1.3} Assume $r = 2$. If $i_1, j_1 \in \{1, \dots, \ell \}$ then $T$ is not a valid tableau for $B_\ell$. So there are two remaining cases, either $i_1, i_2 \in \{1, \dots, \ell \}$ or $j_1, i_2 \in \{1, \dots, \ell \}$.

\textbf{Case 1.3.1} Assume $i_1, i_2 \in \{ 1, \dots, \ell\}$. We have that exactly two entries in the first two rows of $T'$ lie in $\{1, \dots, \ell \}$. The map $\Gamma_\ell$ acts on any tableau with this property by fixing the entries of the tableau column-wise. Hence
\[
\Gamma_\ell \left( \,
\begin{tabular}{cc}
    \multicolumn{1}{c}{$\tilde I$} & \multicolumn{1}{c}{$\tilde J$} \\
    \hline
    \multicolumn{1}{|c|}{$i_1$} & \multicolumn{1}{c|}{$i_2$} \\
    \hline
    \multicolumn{1}{|c|}{$j_1$} & \\
    \cline{1-1}
    \multicolumn{1}{|c|}{$\vdots$} & \\
    \cline{1-1}
\end{tabular}
\, \right)
=
\begin{tabular}{cc}
    \multicolumn{1}{c}{$I'$} & \multicolumn{1}{c}{$J'$} \\
    \hline
    \multicolumn{1}{|c|}{$j_1$} & \multicolumn{1}{c|}{$i_2$} \\
    \hline
    \multicolumn{1}{|c|}{$i_1$} & \\
    \cline{1-1}
    \multicolumn{1}{|c|}{$\vdots$} & \\
    \cline{1-1}
\end{tabular}\, .
\]
Since $I \le w$ we have $i_1 \le \widehat w_1$ and $i_2 \le \widehat w_2$. Since $J \le w$ we have $j_1 \le w_1$. Since $j_1 \in \{\ell+1, \dots, n \}$ and $i_2 \in \{ 1, \dots, \ell\}$ we have $i_2 < j_1 \le w_1$. Hence $J' \le w$.

Let us assume by contradiction that $I' \not\le w$. Since $i_1 \le \widehat w_1$, we must have $\widehat w_2 < j_1$. So we have $\widehat w_1 < \widehat w_2 < j_1 \le w_1$. If $t = 2$ then we have a contradiction since $w_1 \in \{\widehat w_1, \widehat w_2 \}$. So we have $t \ge 3$. Since $\widehat w_2 < w_1$, we have $\widehat w_3 \le w_1$. Therefore $\widehat w_2 < j_1 < i_3 \le \widehat w_3 \le w_1$. And so we have $\widehat w_2 + 1 \in \{w_{t+1}, \dots, w_n \}$. However $w_1, \widehat w_2, \widehat w_2 + 1$ is of type $312$. So by Lemma~\ref{lem:Q_perms_312} we have $\widehat w_2 = \ell$. So $\ell < j_1$ however by assumption $j_1 \in \{1, \dots, \ell \}$, a contradiction.

\textbf{Case 1.3.2} Assume that $j_1, i_2 \in \{1, \dots, \ell \}$. We have that $T'$ is a semi-standard Young tableau and $T' = \Gamma_\ell(T')$. Since $I \le w$ we have $i_2 \le \widehat w_1$ and $i_1 \le \widehat w_2$. Since $J \le w$ we have $j_1 \le w_1$. We have $j_1 < i_2 \le \widehat w_1$ and $i_2 < i_1 \le \widehat w_1$ so $I' \le w$.

Assume by contradiction that $J' \not\le w$. So we have $w_1 < i_1 \le \widehat w_2$. 
Since $\widehat w_1 \le w_1 < \widehat w_2$, so $\widehat w_1 = w_1$. 
Since $j_1 < i_2 \le \widehat w_1$ we have $\widehat w_1 \ge 2$. 
If $\widehat w_2 = \widehat w_1 + 1$ then we have $w|_{w_1 + 1} = (w_1, w_1+1, \dots)$. By Lemma~\ref{lem:Q_perm_w1_w2_consecutive} we have either
$w_1 < w_2 \le \ell$ or $w_1 > w_2 = \ell$. If $w_1 < w_2 \le \ell$ then we have $i_1 \le \widehat w_2 \le w_2 \le \ell$ but $i_1 \in \{\ell+1, \dots, n \}$, a contradiction. If $w_1 > w_2 = \ell$ then we have an immediate contradiction because $\widehat w_1 = w_1 < w_2$.

So $\widehat w_1 + 1 < \widehat w_2$. By Lemma~\ref{lem:Q_perm_w1_w2_gap} we have $w_1 < w_2 \le \ell$ or $w_1 > w_2 = \ell$. If $w_1 < w_2 \le \ell$ then we have $i_1 \le \widehat w_2 \le w_2 \le \ell$ but $i_1 \in \{\ell+1, \dots, n \}$, a contradiction. If $w_1 > w_2 = \ell$ then we have an immediate contradiction because $\widehat w_1 = w_1 < w_2$.

\smallskip

\textbf{Case 2.} Assume $s \ge 2$. Since $T$ and $T'$ are different, we either have that the entries of the first row are changed or the second row are changed. We will treat these as two different cases.

\textbf{Case 2.1} Assume that the entries in the first row of $T$ and $T'$ are different. So we have
\[
T = 
\begin{tabular}{cc}
    \multicolumn{1}{c}{$I$} & \multicolumn{1}{c}{$J$} \\
    \hline
    \multicolumn{1}{|c|}{$i_1$} & \multicolumn{1}{c|}{$j_1$} \\
    \hline
    \multicolumn{1}{|c|}{$i_2$} & 
    \multicolumn{1}{|c|}{$j_2$} \\
    \hline
    \multicolumn{1}{|c|}{$\vdots$} &
    \multicolumn{1}{|c|}{$\vdots$} \\
    \hline
\end{tabular} \, \quad
T' = 
\begin{tabular}{cc}
    \multicolumn{1}{c}{$I'$} & \multicolumn{1}{c}{$J'$} \\
    \hline
    \multicolumn{1}{|c|}{$j_1$} & \multicolumn{1}{c|}{$i_1$} \\
    \hline
    \multicolumn{1}{|c|}{$i_2$} & 
    \multicolumn{1}{|c|}{$j_2$} \\
    \hline
    \multicolumn{1}{|c|}{$\vdots$} &
    \multicolumn{1}{|c|}{$\vdots$} \\
    \hline
\end{tabular} \, .
\]
We proceed by taking cases on $r = |\{i_1, i_2, j_1, j_2\} \cap \{1, \dots, \ell \}|$.

\textbf{Case 2.1.1} Assume $r = 0$ or $4$. We have that $T'$ is a semi-standard Young tableau and $T' = \Gamma_\ell(T')$. Since $I \le w$ we have $i_1 \le \widehat w_1$ and $i_2 \le \widehat w_2$. Since $J \le w$ we also have $j_1 \le \widetilde w_1$ and $j_2 \le \widetilde w_2$. And so we have $i_1 \le \widehat w_1 \le \widetilde w_1 $ and $j_2 \le \widetilde w_2$ hence $J' \le w$. We also have $j_1 \le i_1 \le \widehat w_1$ and $i_2 \le \widehat w_2$ hence $I' \le w$.

\textbf{Case 2.1.2} Assume $r = 1$. So we have that either $i_2 \in \{1, \dots, \ell \}$ or $j_2 \in \{1, \dots, \ell \}$. 

\textbf{Case 2.1.2.1}
Assume $i_2 \in \{1, \dots, \ell \}$. So, by definition of $\Gamma_\ell$, it follows that $i_1 \ge j_1$ and
\[
\Gamma_\ell \left( \,
\begin{tabular}{cc}
    \multicolumn{1}{c}{$\tilde I$} & \multicolumn{1}{c}{$\tilde J$} \\
    \hline
    \multicolumn{1}{|c|}{$i_2$} & \multicolumn{1}{c|}{$i_1$} \\
    \hline
    \multicolumn{1}{|c|}{$j_1$} & 
    \multicolumn{1}{|c|}{$j_2$} \\
    \hline
    \multicolumn{1}{|c|}{$\vdots$} &
    \multicolumn{1}{|c|}{$\vdots$} \\
    \hline
\end{tabular}
\, \right)
=
\begin{tabular}{cc}
    \multicolumn{1}{c}{$I'$} & \multicolumn{1}{c}{$J'$} \\
    \hline
    \multicolumn{1}{|c|}{$j_1$} & \multicolumn{1}{c|}{$i_1$} \\
    \hline
    \multicolumn{1}{|c|}{$i_2$} & 
    \multicolumn{1}{|c|}{$j_2$} \\
    \hline
    \multicolumn{1}{|c|}{$\vdots$} &
    \multicolumn{1}{|c|}{$\vdots$} \\
    \hline
\end{tabular}\, .
\]
Since $T$ and $T'$ are different we have that $i_1 \neq j_1$ hence $i_1 > j_1$. Since $I \le w$ we have $i_2 \le \widehat w_1$ and $i_1 \le \widehat w_2$. Since $J \le w$ we also have $j_1 \le \widetilde w_1$ and $j_2 \le \widetilde w_2$.
We have $i_2 \le \widehat w_1$ and $j_1 < i_1 \le \widehat w_2$ and so $I' \le w$. 

Assume by contradiction that $J' \not\le w$. Since $j_2 \le \widetilde w_2$, we must have $i_1 > \widetilde w_1$. 
So we have $\widetilde w_1 < i_1 < j_2 \le \widetilde w_2$, in particular $\widetilde w_2 > \widetilde w_1 + 1$. Since $\ell \ge 1$, we have $2 \le j_1 \le \widetilde w_1$. By Lemma~\ref{lem:Q_perm_w1_w2_gap} we either have $\widetilde w_1 = w_1 < w_2 \le \ell$ or $w_1 > w_2 = \ell = \widetilde w_1$. If $\widetilde w_1 = w_1 < w_2 \le \ell$ then $i_1 \le \widehat w_2 \le w_2 \le \ell$ and $i_1 \in \{\ell+1, \dots, n \}$, a contradiction. If $w_1 > w_2 = \ell = \widetilde w_1$ then $\ell < j_1 \le \widetilde w_1 \le \ell$, a contradiction.

\textbf{Case 2.1.2.2}
Assume $j_2 \in \{1, \dots, \ell \}$. Then we have
\[
\Gamma_\ell \left( \,
\begin{tabular}{cc}
    \multicolumn{1}{c}{$\tilde I$} & \multicolumn{1}{c}{$\tilde J$} \\
    \hline
    \multicolumn{1}{|c|}{$j_2$} & \multicolumn{1}{c|}{$j_1$} \\
    \hline
    \multicolumn{1}{|c|}{$i_2$} & 
    \multicolumn{1}{|c|}{$i_1$} \\
    \hline
    \multicolumn{1}{|c|}{$\vdots$} &
    \multicolumn{1}{|c|}{$\vdots$} \\
    \hline
\end{tabular}
\, \right)
=
\begin{tabular}{cc}
    \multicolumn{1}{c}{$I'$} & \multicolumn{1}{c}{$J'$} \\
    \hline
    \multicolumn{1}{|c|}{$j_1$} & \multicolumn{1}{c|}{$i_1$} \\
    \hline
    \multicolumn{1}{|c|}{$i_2$} & 
    \multicolumn{1}{|c|}{$j_2$} \\
    \hline
    \multicolumn{1}{|c|}{$\vdots$} &
    \multicolumn{1}{|c|}{$\vdots$} \\
    \hline
\end{tabular}\, .
\]
Since $\tilde T = [\tilde I \tilde J]$ is a semi-standard Young tableau we have $i_2 \le i_1$. However in $T$ we have $i_1 < i_2$, a contradiction.

\smallskip

\textbf{Case 2.1.3} Assume $r = 2$. We must have that $i_2, j_2 \in \{1, \dots, \ell \}$ otherwise at least one of $T$ or $T'$ is not a valid tableau for $B_\ell$. 
We have 
\[
\Gamma_\ell \left( \,
\begin{tabular}{cc}
    \multicolumn{1}{c}{$\tilde I$} & \multicolumn{1}{c}{$\tilde J$} \\
    \hline
    \multicolumn{1}{|c|}{$i_2$} & \multicolumn{1}{c|}{$j_2$} \\
    \hline
    \multicolumn{1}{|c|}{$j_1$} & 
    \multicolumn{1}{|c|}{$i_1$} \\
    \hline
    \multicolumn{1}{|c|}{$\vdots$} &
    \multicolumn{1}{|c|}{$\vdots$} \\
    \hline
\end{tabular}
\, \right)
=
\begin{tabular}{cc}
    \multicolumn{1}{c}{$I'$} & \multicolumn{1}{c}{$J'$} \\
    \hline
    \multicolumn{1}{|c|}{$j_1$} & \multicolumn{1}{c|}{$i_1$} \\
    \hline
    \multicolumn{1}{|c|}{$i_2$} & 
    \multicolumn{1}{|c|}{$j_2$} \\
    \hline
    \multicolumn{1}{|c|}{$\vdots$} &
    \multicolumn{1}{|c|}{$\vdots$} \\
    \hline
\end{tabular}\, .
\]
Since $\tilde T$ is a semi-standard Young tableau we have $i_2 \le j_2$ and $j_1 \le i_1$. Since $I \le w$ we have $i_2 \le \widehat w_1 $ and $i_1 \le \widehat w_2$. Since $J \le w$ we have $j_2 \le \widetilde w_1$ and $j_1 \le \widetilde w_2$. And so we have $i_2 \le \widehat w_2$ and $j_1 \le i_1 \le \widehat w_2$ therefore $I' \le w$. We also have $j_2 \le \widetilde w_1$ and $i_1 \le \widehat w_2 \le \widetilde w_2$ and so $J' \le w$.

\textbf{Case 2.1.4} Assume $r = 3$. We have that either $i_1 \in \{\ell+1, \dots, n \}$ or $j_1 \in \{\ell+1, \dots, n \}$.

\textbf{Case 2.1.4.1} Assume $i_1 \in \{\ell+1, \dots, n \}$. Then we have
\[
\Gamma_\ell \left( \,
\begin{tabular}{cc}
    \multicolumn{1}{c}{$\tilde I$} & \multicolumn{1}{c}{$\tilde J$} \\
    \hline
    \multicolumn{1}{|c|}{$j_1$} & \multicolumn{1}{c|}{$j_2$} \\
    \hline
    \multicolumn{1}{|c|}{$i_2$} & 
    \multicolumn{1}{|c|}{$i_1$} \\
    \hline
    \multicolumn{1}{|c|}{$\vdots$} &
    \multicolumn{1}{|c|}{$\vdots$} \\
    \hline
\end{tabular}
\, \right)
=
\begin{tabular}{cc}
    \multicolumn{1}{c}{$I'$} & \multicolumn{1}{c}{$J'$} \\
    \hline
    \multicolumn{1}{|c|}{$j_1$} & \multicolumn{1}{c|}{$i_1$} \\
    \hline
    \multicolumn{1}{|c|}{$i_2$} & 
    \multicolumn{1}{|c|}{$j_2$} \\
    \hline
    \multicolumn{1}{|c|}{$\vdots$} &
    \multicolumn{1}{|c|}{$\vdots$} \\
    \hline
\end{tabular}\, .
\]
Since $I \le w$ we have $i_2 \le \widehat w_1$ and $i_1 \le \widehat w_2$. Since $J \le w$ we have $j_1 \le \widetilde w_1$ and $j_2 \le \widetilde w_2$. And so $j_1 < i_2 \le \widehat w_1$ and $i_2 \le \widehat w_1 < \widehat w_2$ hence $I' \le w$.

Assume by contradiction that $J' \not\le w$. Since $i_1 \le \widehat w_2 \le \widetilde w_2$ we must have $j_2 > \widetilde w_1$. And so $\widetilde w_1 < j_2 < i_1 \le \widehat w_2$ hence $\widetilde w_1 = \widehat w_1$ and $\widehat w_2 > \widehat w_1 + 1$. Since $j_1 < i_2 \le \widehat w_1$, we have $\widehat w_1 \ge 2$. By Lemma~\ref{lem:Q_perm_w1_w2_gap} we have we either have $\widehat w_1 = w_1 < w_2 \le \ell$ or $w_1 > w_2 = \ell = \widehat w_1$. If $\widehat w_1 = w_1 < w_2 \le \ell$ then $\ell < i_1 \le \widehat w_2 \le \ell$, a contradiction. If $w_1 > w_2 = \ell = \widehat w_1$ then $\ell = \widehat w_1 < j_1 \le \ell $, a contradiction.

\textbf{Case 2.1.4.2} Assume $j_1 \in \{\ell+1, \dots, n \}$. There are no possible semi-standard Young tableau $\tilde T$ such that $\Gamma_\ell(\tilde T) = T'$. Suppose by contradiction that $\tilde T = [\tilde I \tilde J]$ is such a semi-standard Young tableau where $\tilde I = \{i_1' < i'_2 < \dots \}$ and $J = \{j_1' < j_2' < \dots \}$. Then we must have $j_2' = j_1 \in \{\ell+1, \dots, n \}$ as it is the maximum value appearing in the first two rows of $T'$. For any possible values $\{i_1', i_2', j_1' \} \subseteq \{1, \dots, \ell \}$ we have that $j_2'$ lies in the second column of $\Gamma_\ell(\tilde T)$. However $j_2' = j_1$ is contained in the first column of $T'$.

\smallskip

\textbf{Case 2.2} Assume that the entries of the second row of $T$ and $T'$ are different. So we have 
\[
T = 
\begin{tabular}{cc}
    \multicolumn{1}{c}{$I$} & \multicolumn{1}{c}{$J$} \\
    \hline
    \multicolumn{1}{|c|}{$i_1$} & \multicolumn{1}{c|}{$j_1$} \\
    \hline
    \multicolumn{1}{|c|}{$i_2$} & 
    \multicolumn{1}{|c|}{$j_2$} \\
    \hline
    \multicolumn{1}{|c|}{$\vdots$} &
    \multicolumn{1}{|c|}{$\vdots$} \\
    \hline
\end{tabular} \, \quad
T' = 
\begin{tabular}{cc}
    \multicolumn{1}{c}{$I'$} & \multicolumn{1}{c}{$J'$} \\
    \hline
    \multicolumn{1}{|c|}{$i_1$} & \multicolumn{1}{c|}{$j_1$} \\
    \hline
    \multicolumn{1}{|c|}{$j_2$} & 
    \multicolumn{1}{|c|}{$i_2$} \\
    \hline
    \multicolumn{1}{|c|}{$\vdots$} &
    \multicolumn{1}{|c|}{$\vdots$} \\
    \hline
\end{tabular} \, .
\]
We proceed by taking cases on $r = |\{i_1, i_2, j_1, j_2 \} \cap \{1, \dots, \ell \}|$.

\textbf{Case 2.2.1} Assume $r = 0$ or $4$. In this case we have $T'$ is a semi-standard Young tableau and $\Gamma_\ell(T') = T'$. So we have $i_1 \le j_1$ and $j_2 \le i_2$. Since $I \le w$ we have $i_1 \le \widehat w_1$ and $i_2 \le \widehat i_2$. Since $J \le w$ we also have $j_1 \le \widetilde w_1$ and $j_2 \le \widetilde w_2$. So $i_1 \le \widehat w_1$ and $j_2 \le i_2 \le \widehat w_2$ hence $I' \le w$. Also $j_1 \le \widetilde w_1$ and $i_2 \le \widehat w_2 \le \widetilde w_2$ hence $J' \le w$.

\textbf{Case 2.2.2} Assume $r = 1$.
We have that either $j_2 \in \{1, \dots, \ell \}$ or $i_2 \in \{1, \dots, \ell \}$.

\textbf{Case 2.2.2.1} Assume $j_2 \in \{ 1, \dots, \ell\}$. We have
\[
\Gamma_\ell \left( \,
\begin{tabular}{cc}
    \multicolumn{1}{c}{$\tilde I$} & \multicolumn{1}{c}{$\tilde J$} \\
    \hline
    \multicolumn{1}{|c|}{$j_2$} & \multicolumn{1}{c|}{$j_1$} \\
    \hline
    \multicolumn{1}{|c|}{$i_1$} & 
    \multicolumn{1}{|c|}{$i_2$} \\
    \hline
    \multicolumn{1}{|c|}{$\vdots$} &
    \multicolumn{1}{|c|}{$\vdots$} \\
    \hline
\end{tabular}
\, \right)
=
\begin{tabular}{cc}
    \multicolumn{1}{c}{$I'$} & \multicolumn{1}{c}{$J'$} \\
    \hline
    \multicolumn{1}{|c|}{$i_1$} & \multicolumn{1}{c|}{$j_1$} \\
    \hline
    \multicolumn{1}{|c|}{$j_2$} & 
    \multicolumn{1}{|c|}{$i_2$} \\
    \hline
    \multicolumn{1}{|c|}{$\vdots$} &
    \multicolumn{1}{|c|}{$\vdots$} \\
    \hline
\end{tabular}\, .
\]
Since $I \le w$ we have $i_1 \le \widehat w_1$ and $i_2 \le \widehat w_2$. Since $J \le w$ we have $j_1 \le \widetilde w_1$ and $j_1 \le \widetilde w_2$. So $j_2 < i_1 \le \widehat w_1$ and $i_1 \le \widehat w_1 < \widehat w_2$ so $I' \le w$.

Assume by contradiction that $J' \not\le w$. Since $i_2 \le \widehat w_2 \le \widetilde w_2$ therefore we must have $j_1 > \widetilde w_1$. And so $\widetilde w_1 < j_1 < i_2 \le \widehat w_2 \le \widetilde w_2$. Since $\widetilde w_1 < \widehat w_2$ therefore $\widehat w_1 = \widetilde w_1$. Note that $j_2 < i_1 \le \widehat w_1$ so $\widehat w_1 \ge 2$. Let $w_d = \widetilde w_1 + 1 \ge 3$. Let $w_d = \widetilde w_1 + 1$, note that we have $d \in \{t+1, \dots, n \}$. And so we have the restriction $w|_{w_d} = (w_d - 1, w_d, \dots)$. And so by Lemma~\ref{lem:Q_perm_w1_w2_consecutive} we either have $w_1 < w_2 \le \ell$ or $w_1 > w_2 = \ell$. In each case $\ell < i_1 \le \widehat w_1 \le \ell$, a contradiction.

\textbf{Case 2.2.2.2} Assume $i_2 \in \{ 1, \dots, \ell\}$. We have
\[
\Gamma_\ell \left( \,
\begin{tabular}{cc}
    \multicolumn{1}{c}{$\tilde I$} & \multicolumn{1}{c}{$\tilde J$} \\
    \hline
    \multicolumn{1}{|c|}{$i_2$} & \multicolumn{1}{c|}{$i_1$} \\
    \hline
    \multicolumn{1}{|c|}{$j_2$} & 
    \multicolumn{1}{|c|}{$j_1$} \\
    \hline
    \multicolumn{1}{|c|}{$\vdots$} &
    \multicolumn{1}{|c|}{$\vdots$} \\
    \hline
\end{tabular}
\, \right)
=
\begin{tabular}{cc}
    \multicolumn{1}{c}{$I'$} & \multicolumn{1}{c}{$J'$} \\
    \hline
    \multicolumn{1}{|c|}{$i_1$} & \multicolumn{1}{c|}{$j_1$} \\
    \hline
    \multicolumn{1}{|c|}{$j_2$} & 
    \multicolumn{1}{|c|}{$i_2$} \\
    \hline
    \multicolumn{1}{|c|}{$\vdots$} &
    \multicolumn{1}{|c|}{$\vdots$} \\
    \hline
\end{tabular}\, .
\]
Since $\tilde T = [\tilde I \tilde J]$ is a semi-standard Young tableau we have $j_2 \le j_1$. However in $T$ we have $j_1 < j_2$, a contradiction.

\textbf{Case 2.2.3} Assume $r = 2$. We must have that $i_2, j_2 \in \{1, \dots, \ell \}$ otherwise at least one of $T$ or $T'$ is not a valid tableau for $B_\ell$. 
We have 
\[
\Gamma_\ell \left( \,
\begin{tabular}{cc}
    \multicolumn{1}{c}{$\tilde I$} & \multicolumn{1}{c}{$\tilde J$} \\
    \hline
    \multicolumn{1}{|c|}{$j_2$} & \multicolumn{1}{c|}{$i_2$} \\
    \hline
    \multicolumn{1}{|c|}{$i_1$} & 
    \multicolumn{1}{|c|}{$j_1$} \\
    \hline
    \multicolumn{1}{|c|}{$\vdots$} &
    \multicolumn{1}{|c|}{$\vdots$} \\
    \hline
\end{tabular}
\, \right)
=
\begin{tabular}{cc}
    \multicolumn{1}{c}{$I'$} & \multicolumn{1}{c}{$J'$} \\
    \hline
    \multicolumn{1}{|c|}{$i_1$} & \multicolumn{1}{c|}{$j_1$} \\
    \hline
    \multicolumn{1}{|c|}{$j_2$} & 
    \multicolumn{1}{|c|}{$i_2$} \\
    \hline
    \multicolumn{1}{|c|}{$\vdots$} &
    \multicolumn{1}{|c|}{$\vdots$} \\
    \hline
\end{tabular}\, .
\]
Since $\tilde T$ is a semi-standard Young tableau we have $j_2 \le i_2$ and $i_1 \le j_1$. Since $I \le w$ we have $i_2 \le \widehat w_1 $ and $i_1 \le \widehat w_2$. Since $J \le w$ we have $j_2 \le \widetilde w_1$ and $j_1 \le \widetilde w_2$. And so we have $j_2 \le i_2 \le \widehat w_1$ and $i_1 \le \widehat w_2$ therefore $I' \le w$. Also we have $j_1 \le \widetilde w_2$ and $i_2 \le \widehat w_1 \le \widetilde w_1$ and so $J' \le w$.

\textbf{Case 2.2.4} Assume $r = 3$. We have that either $i_1 \in \{\ell+1, \dots, n \}$ or $j_1 \in \{\ell+1, \dots, n \}$.

\textbf{Case 2.2.4.1} Assume $i_1 \in \{\ell+1, \dots, n \}$. Then there are no possible semi-standard Young tableau $\tilde T$ such that $\Gamma_\ell(\tilde T) = T'$, see Case 2.1.4.2.

\textbf{Case 2.2.4.2} Assume $j_1 \in \{ \ell+1, \dots, n\}$. We have $i_1 < i_2$ and $i_1 < j_2$, so 
\[
\Gamma_\ell \left( \,
\begin{tabular}{cc}
    \multicolumn{1}{c}{$\tilde I$} & \multicolumn{1}{c}{$\tilde J$} \\
    \hline
    \multicolumn{1}{|c|}{$i_1$} & \multicolumn{1}{c|}{$i_2$} \\
    \hline
    \multicolumn{1}{|c|}{$j_2$} & 
    \multicolumn{1}{|c|}{$j_1$} \\
    \hline
    \multicolumn{1}{|c|}{$\vdots$} &
    \multicolumn{1}{|c|}{$\vdots$} \\
    \hline
\end{tabular}
\, \right)
=
\begin{tabular}{cc}
    \multicolumn{1}{c}{$I'$} & \multicolumn{1}{c}{$J'$} \\
    \hline
    \multicolumn{1}{|c|}{$i_1$} & \multicolumn{1}{c|}{$j_1$} \\
    \hline
    \multicolumn{1}{|c|}{$j_2$} & 
    \multicolumn{1}{|c|}{$i_2$} \\
    \hline
    \multicolumn{1}{|c|}{$\vdots$} &
    \multicolumn{1}{|c|}{$\vdots$} \\
    \hline
\end{tabular}\, .
\]
Since $I \le w$ we have $i_1 \le \widehat w_1$ and $i_2 \le \widehat w_2$. Also since $J \le w$ we have $j_2 \le \widetilde w_1$ and $j_1 \le \widetilde w_2$. So $i_1 \le \widehat w_1$ and $j_2 \le i_2 \le \widehat w_2$ hence $I' \le w$.

Assume by contradiction that $J' \not\le w$. Since $j_1 \le \widetilde w_2$ we must have $i_2 > \widetilde w_1$. We have $\widetilde w_1 < i_2 < j_1 \le \widetilde w_2$, in particular $\widetilde w_2 > \widetilde w_1 + 1$. Also we have $i_1 < j_2 \le \widetilde w_1$ hence $\widetilde w_1 \ge 2$. So by Lemma~\ref{lem:Q_perm_w1_w2_gap} we either have $\widetilde w_1 = w_1 < w_2 \le \ell$ or $w_1 > w_2 = \ell = \widetilde w_1$. If $\widetilde w_1 = w_1 < w_2 \le \ell$ then $\ell < j_1 \le \widetilde w_2 \le w_2 \le \ell$, a contradiction. If $w_1 > w_2 = \ell = \widetilde w_1$ then $\ell = \widetilde w_1 < i_2$ and $i_2 \in \{1, \dots, \ell \}$, a contradiction.

\medskip

We have shown that any tableau $T = [IJ]$, for the matching field $B_\ell$ with columns $I, J \le w$, is row-wise equal to a tableau $T' = [I'J']$ which lies in the image $\Gamma_\ell$ and $I', J' \le w$.
\end{proof}

We have shown that the map $\Gamma_\ell$ has the desired properties and so we can show that the number of standard monomials for $J_1$ in degree two is $|SSYT_2(w)|$.  
\begin{proof}[Proof of Lemma~\ref{lem:SSYT_flag_biject_Q_perms}]
Let $w \in P_\ell$ be a permutation. For each semi-standard Young tableau $T = [IJ]$ let $T' = [I'J'] = \Gamma_\ell(T)$. By Lemmas~\ref{lem:SSYT_flag_image_non_vanish_Q_ell} and \ref{lem:SSYT_flag_preimage_non_vanish} we have that $I, J \le w$ if and only if $I', J' \le w$. We deduce that $\Gamma_\ell(SSYT_2(w))$ corresponds to a collection of standard monomials for $J_1$ in degree two. By Lemma~\ref{lem:SSYT_flag_injective} we have that the monomials corresponding to $\Gamma_\ell(SSYT_2(w))$ are linearly independent. By Lemma~\ref{lem:SSYT_flag_surject_w_Q_ell} we have for any $T = [IJ]$ a matching field tableau for $B_\ell$ with $I, J \le w$ there exists a semi-standard Young tableau $\tilde T$ such that $\Gamma_\ell(\tilde T) = T' = [I'J']$ is row-wise equal to $T$ and $I', J' \le w$. Therefore, by Lemma~\ref{lem:SSYT_flag_preimage_non_vanish}, $\tilde I, \tilde J \le w$ and so the monomials corresponding to $\Gamma_\ell(SSYT_2(w))$ are a spanning set.
\end{proof}

\bibliographystyle{alpha} 
\bibliography{MEGA-Trop1.bib}

\newcommand{\etalchar}[1]{$^{#1}$}
\begin{thebibliography}{BLMM17}

\bibitem[ACK18]{FvectorGC}
B.H. An, Y.~Cho, and J.S. Kim.
\newblock On the f-vectors of {G}elfand-{T}setlin polytopes.
\newblock {\em European Journal of Combinatorics}, 67:61--77, 2018.

\bibitem[And13]{An13}
D.~Anderson.
\newblock Okounkov bodies and toric degenerations.
\newblock {\em Mathematische Annalen}, 356(3):1183--1202, 2013.

\bibitem[BFF{\etalchar{+}}18]{BFFHL}
L.~Bossinger, X.~Fang, G.~Fourier, M.~Hering, and M.~Lanini.
\newblock Toric degenerations of {G}r$(2,n)$ and {G}r$(3,6)$ via plabic graphs.
\newblock {\em Ann. Combin.}, 22(3):491--512, 2018.

\bibitem[BLMM17]{bossinger2017computing}
L.~Bossinger, S.~Lamboglia, K.~Mincheva, and F.~Mohammadi.
\newblock Computing toric degenerations of flag varieties.
\newblock In {\em Combinatorial algebraic geometry}, pages 247--281. Springer,
  2017.

\bibitem[BMC20]{bossinger2020families}
L.~Bossinger, F.~Mohammadi, and A.~N{\'a}jera Ch{\'a}vez.
\newblock Families of {G}r\"obner degenerations, {G}rassmannians and universal
  cluster algebras.
\newblock {\em arXiv preprint arXiv:2007.14972}, 2020.

\bibitem[Cal02]{caldero2002toric}
P.~Caldero.
\newblock Toric degenerations of {S}chubert varieties.
\newblock {\em Transform. Groups}, 7(1):51--60, 2002.

\bibitem[CCM20]{Ollie4}
N.~{Chary Bonala}, O.~Clarke, and F.~Mohammadi.
\newblock Standard monomial theory and toric degenerations of {R}ichardson
  varieties inside {G}rassmannians and flag varieties.
\newblock {\em In preparation}, 2020.

\bibitem[CHM20]{Akihiro}
O.~Clarke, A.~Higashitani, and F.~Mohammadi.
\newblock Matching field polytopes and their combinatorial mutations.
\newblock {\em In preparation}, 2020.

\bibitem[CLS11]{Cox2}
D.A. Cox, J.B. Little, and H.K. Schenck.
\newblock {\em Toric varieties}.
\newblock American Mathematical Soc., 2011.

\bibitem[CM19]{OllieFatemeh2}
O.~Clarke and F.~Mohammadi.
\newblock Toric degenerations of flag varieties from matching field tableaux.
\newblock {\em To appear in Journal of Pure and Applied Algebra, arXiv preprint
  arXiv:1904.07832}, 2019.

\bibitem[CM20]{OllieFatemeh}
O.~Clarke and F.~Mohammadi.
\newblock Toric degenerations of {G}rassmannians and {S}chubert varieties from
  matching field tableaux.
\newblock {\em Journal of Algebra}, 559:646--678, 2020.

\bibitem[FFL17]{FFL16}
X.~Fang, G.~Fourier, and P.~Littelmann.
\newblock On toric degenerations of flag varieties.
\newblock In {\em Representation theory---current trends and perspectives}, EMS
  Ser. Congr. Rep., pages 187--232. Eur. Math. Soc., Z\"urich, 2017.

\bibitem[FR15]{fink2015stiefel}
A.~Fink and F.~Rinc{\'o}n.
\newblock Stiefel tropical linear spaces.
\newblock {\em Journal of Combinatorial Theory, Series A}, 135:291--331, 2015.

\bibitem[GL96]{GL96}
N.~Gonciulea and V.~Lakshmibai.
\newblock Degenerations of flag and {S}chubert varieties to toric varieties.
\newblock {\em Transformation Groups}, 1(3):215--248, 1996.

\bibitem[GS]{M2}
D.R. Grayson and M.E. Stillman.
\newblock Macaulay2, a software system for research in algebraic geometry.
\newblock Available at https://faculty.math.illinois.edu/Macaulay2/.

\bibitem[Kim15]{kim2015richardson}
G.~Kim.
\newblock Richardson varieties in a toric degeneration of the flag variety.
\newblock 2015.

\bibitem[KM05]{KOGAN}
M.~Kogan and E.~Miller.
\newblock Toric degeneration of {S}chubert varieties and {G}elfand-{T}setlin
  polytopes.
\newblock {\em Advances in Mathematics}, 193(1):1--17, 2005.

\bibitem[KM19]{kaveh2019khovanskii}
K.~Kaveh and C.~Manon.
\newblock Khovanskii bases, higher rank valuations, and tropical geometry.
\newblock {\em SIAM Journal on Applied Algebra and Geometry}, 3(2):292--336,
  2019.

\bibitem[KMS15]{kateri2015family}
M.~Kateri, F.~Mohammadi, and B.~Sturmfels.
\newblock A family of quasisymmetry models.
\newblock {\em Journal of Algebraic Statistics}, 6(1), 2015.

\bibitem[MS05]{MS05}
E.~Miller and B.~Sturmfels.
\newblock {\em Combinatorial {C}ommutative {A}lgebra}, volume 227 of {\em
  Graduate Texts in Mathematics}.
\newblock Springer-Verlag, New York, 2005.

\bibitem[MS15]{M-S}
D.~Maclagan and B.~Sturmfels.
\newblock {\em Introduction to tropical geometry}, volume 161.
\newblock American Mathematical Soc., 2015.

\bibitem[MS19]{KristinFatemeh}
F.~Mohammadi and K.~Shaw.
\newblock Toric degenerations of grassmannians from matching fields.
\newblock {\em Algebraic Combinatorics}, 2(6):1109--1124, 2019.

\bibitem[RW19]{rietsch2017newton}
K.~Rietsch and L.~Williams.
\newblock Newton-{O}kounkov bodies, cluster duality, and mirror symmetry for
  {G}rassmannians.
\newblock {\em Duke Mathematical Journal}, 168(18):3437--3527, 2019.

\bibitem[SS04]{speyer2004tropical}
D.~Speyer and B.~Sturmfels.
\newblock The tropical {G}rassmannian.
\newblock {\em Advances in Geometry}, 4(3):389--411, 2004.

\bibitem[SSBW19]{serhiyenko2019cluster}
K.~Serhiyenko, M.~Sherman-Bennett, and L.~Williams.
\newblock Cluster structures in {S}chubert varieties in the {G}rassmannian.
\newblock {\em Proceedings of the London Mathematical Society},
  119(6):1694--1744, 2019.

\bibitem[Stu96]{sturmfels1996grobner}
B.~Sturmfels.
\newblock {\em Gr{\"o}bner Bases and Convex Polytopes}, volume~8.
\newblock American Mathematical Society, 1996.

\bibitem[SZ93]{sturmfels1993maximal}
B.~Sturmfels and A.~Zelevinsky.
\newblock Maximal minors and their leading terms.
\newblock {\em Advances in Mathematics}, 98(1):65--112, 1993.

\end{thebibliography}

\medskip
\noindent
\footnotesize {\bf Authors' addresses:}

\medskip

\noindent University of Bristol, School of Mathematics,
BS8 1TW, Bristol, UK
\\
\noindent  E-mail addresses: {\tt oliver.clarke@bristol.ac.uk}

\medskip

\noindent
Department of Mathematics: Algebra and Geometry, Ghent University, 9000 Gent, Belgium \\
Department of Mathematics and Statistics,
UiT – The Arctic University of Norway, 9037 Troms\o, Norway
\\ E-mail address: {\tt fatemeh.mohammadi@ugent.be}
\end{document}